\documentclass{amsart}

\usepackage{enumerate, amsmath, amsfonts, amssymb, amsthm, wasysym, graphics, graphicx, xcolor, url, hyperref, hypcap, xargs}
\hypersetup{colorlinks=true, citecolor=darkblue, linkcolor=darkblue, urlcolor=darkblue}
\usepackage[all]{xy}
\usepackage{tikz}\usetikzlibrary{trees,decorations,shapes,arrows,matrix,calc}
\graphicspath{{figures/}}
    
\numberwithin{equation}{section}

\title[Brick polytopes of spherical subword complexes]{Brick polytopes of spherical subword complexes and generalized associahedra}

\thanks{Vincent Pilaud was partially supported by grants MTM2008-04699-C03-02 and MTM2011-22792 of the spanish MICINN, by European Research Project ExploreMaps (ERC StG 208471), and by a postdoctoral grant of the Fields Institute of Toronto.\\
\indent Christian Stump was partially supported by a CRM-ISM postdoctoral fellowship, and by the DFG via the Research Group ``Methods for Discrete Structures'' and grant STU 563/2-1 ``Coxeter-Catalan combinatorics''.
}

\author{Vincent Pilaud}
\address{CNRS \& LIX, \'Ecole Polytechnique, Palaiseau}
\email{vincent.pilaud@lix.polytechnique.fr}
\urladdr{http://www.lix.polytechnique.fr/~pilaud/}

\author{Christian Stump}
\address{Institut f\"ur Mathematik, Freie Universit\"at Berlin, Germany}
\email{christian.stump@fu-berlin.de}
\urladdr{http://homepage.univie.ac.at/christian.stump/}

\newtheorem{theorem}{Theorem}[section]
\newtheorem{corollary}[theorem]{Corollary}
\newtheorem{proposition}[theorem]{Proposition}
\newtheorem{lemma}[theorem]{Lemma}
\newtheorem{definition}[theorem]{Definition}
\newtheorem{conjecture}[theorem]{Conjecture}

\theoremstyle{definition}
\newtheorem{example}[theorem]{Example}
\newtheorem{remark}[theorem]{Remark}

\newtheorem{typeA}{The classical situation}

\newcommand{\R}{\mathbb{R}} % reals
\newcommand{\N}{\mathbb{N}} % naturals
\newcommand{\Z}{\mathbb{Z}} % integers
\newcommand{\cN}{\mathcal{N}} % network
\newcommand{\cF}{\mathcal{F}} % facets
\newcommand{\fA}{\mathfrak{A}} % alternating group
\newcommand{\fS}{\mathfrak{S}} % symmetric group
\newcommand{\sfr}{{\sf r}} % root
\newcommand{\sfR}{{\sf R}} % Root
\newcommand{\sfw}{{\sf w}} % weight
\newcommand{\sfW}{{\sf W}} % Weight
\newcommand{\sfB}{{\sf B}} % Brick
\newcommand{\sfM}{{\sf M}} % Matrix
\newcommand{\sfD}{{\sf D}} % Diagonal matrix

\newcommand{\set}[2]{\left\{ #1 \;\middle|\; #2 \right\}} % set notation
\newcommand{\bigset}[2]{\big\{ #1 \;|\; #2 \big\}} % big set notation
\newcommand{\bigmultiset}[2]{\big\{\!\!\big\{ #1 \;|\; #2 \big\}\!\!\big\}} % big multiset notation
\newcommand{\ssm}{\smallsetminus} % small set minus
\newcommand{\dotprod}[2]{\left\langle \, #1 \,\middle|\, #2 \, \right\rangle} % dot product
\newcommand{\bigdotprod}[2]{\big\langle \, #1 \,\big|\, #2 \, \big\rangle} % big dot product
\newcommand{\smalldotprod}[2]{\langle \, #1 \,|\, #2 \, \rangle} % small dot product
\newcommand{\symdif}{\triangle} % symmetric difference
\newcommand{\one}{\mathbf{1}} % the all one vector
\newcommand{\eqdef}{\mbox{\,\raisebox{0.2ex}{\scriptsize\ensuremath{\mathrm:}}\ensuremath{=}\,}} % :=
\newcommand{\eqfed}{\mbox{~\ensuremath{=}\raisebox{0.2ex}{\scriptsize\ensuremath{\mathrm:}} }} % =:
\newcommand{\polar}{^\diamond} % polar
\newcommand{\covers}[1]{\operatorname{cov}(#1)} % cover relations
\newcommand{\hvector}{h} % h-vector

\newcommand{\fundamentalChamber}{\mathcal{C}} % fundamental chamber
\newcommand{\DemazureProduct}{\delta} % Demazure product
\newcommand{\length}{\ell} % length
\newcommand{\reflength}{\ell_R} % reflection length
\newcommandx{\Perm}[2][1={}, 2=W]{{\sf Perm}^{#1}(#2)} % the permutahedron
\newcommandx{\Asso}[3][1={}, 2=c, 3=W]{{\sf Asso}^{#1}_{#2}(#3)} % the generalized associahedron

\newcommand{\sizeQ}{m} % size of the word \Q
\newcommand{\sq}[1]{{\rm #1}} % sequence of letters (squared symbols)
\newcommand{\Q}{\sq{Q}} % the word Q
\newcommand{\q}{\sq{q}} % the letter q
\newcommand{\w}{\sq{w}} % the letter w
\newcommand{\rotate}[1]{#1_\circlearrowleft} % the rotated word of #1
\newcommand{\wordprod}[2]{\Pi{#1}_{#2}} % product of the letters in Q \ssm I

\newcommand{\subwordComplex}{\mathcal{SC}} % subword complex
\newcommand{\brickVector}{\sfB} % brick vector
\newcommand{\brickPolytope}{\mathcal{B}} % brick polytope
\newcommand{\Root}[2]{{\sfr}(#1,#2)} % the root of #1 at position #2
\newcommand{\Roots}[1]{{\sfR}(#1)} % the root configuration of #1
\newcommand{\RootsPlus}[1]{{\sfR}^+(#1)} % the positive roots in the root configuration of #1
\newcommand{\RootsMinus}[1]{{\sfR}^-(#1)} % the negative of the negative roots in the root configuration of #1
\newcommand{\Weight}[2]{{\sfw}(#1,#2)} % the weight of #1 at position #2
\newcommand{\Weights}[1]{{\sfW}(#1)} % the weight configuration of #1
\newcommand{\projectionMap}{\kappa} % the projection map from the group to the subword complex
\newcommand{\rootCone}{{\sf C}} % root cone
\newcommand{\normalCone}{{\sf C}\polar} % normal cone
\newcommand{\contact}{^\#} % contact graph
\newcommand{\mutmatrixold}[1]{\sfM(#1)} % the B-matrix corresponding to a facet
\newcommand{\mutmatrix}[1]{\sfM^{\star}(#1)} % the B-matrix corresponding to a facet
\newcommand{\diagonal}{\sfD} % diagonal matrix

\newcommand{\sw}[2]{\sq{#1}(\sq{#2})} % the #2-sorting word for #1
\newcommand{\cwo}[1]{\sw{w_\circ}{#1}} % the #1-sorting word for the longest element
\newcommand{\cw}[1]{\sq{#1}\cwo{#1}} % the word for a cluster complex
\newcommand{\translation}{\Omega} % translation vector
\newcommand{\terminal}{T} % terminal positions
\newcommand{\position}{k} % position parameter
\newcommand{\Skips}[2]{\mathcal{C}_{#1}(#2)} % the set C_c(w) from Reading-Speyer
\newcommand{\pidown}{\pi_\downarrow^c} % down map
\newcommand{\clustercomplex}{\subwordComplex(\cw{c})} % the cluster complex
\newcommand{\facet}[2]{\operatorname{\textsc{Fac}}(#2,#1)} % the facets of the cluster complex
\newcommand{\sortable}[2]{\operatorname{\textsc{Sort}}(#2,#1)} % the #1-sortable elements
\newcommand{\cluster}[2]{\operatorname{\textsc{Clus}}(#2,#1)} % #1-clusters
\newcommand{\ncp}[2]{\operatorname{\textsc{NC}}(#2,#1)} % #1-noncrossing partitions
\newcommand{\ncs}[2]{\operatorname{\textsc{NN}}(#2)} % #1-noncrossing subspaces
\newcommand{\sortableCluster}[1]{\mathsf{cl}_{#1}} % the bijection between #1-sortable elements and #1-clusters by Reading
\newcommand{\facetCluster}[1]{\mathsf{lr}_{#1}} % the bijection between facets and #1-clusters by CLS
\newcommand{\clusterNcp}[1]{\mathsf{rp}_{#1}} % the bijection between #1-clusters and #1-noncrossing partitions
\newcommand{\facetSortable}[1]{\mathsf{so}_{#1}} % the bijection between facets and #1-sortable elements
\newcommand{\sortableNcp}[1]{\mathsf{nc}_{#1}} % the bijection between #1-sortable elements and #1-noncrossing partitions by Reading
\newcommand{\ncpNcs}{\mathsf{AST}} % the fixed space of an element
\newcommand{\facetNcp}[1]{\mathsf{lp}_{#1}} % the bijection between facets and #1-noncrossing partitions
\newcommand{\ncsFacet}[1]{\mathsf{sc}_{#1}} % the bijection between #1-noncrossing subspaces and facets
\newcommand{\mutation}[1]{\mu_{#1}} % seed mutation

\newcommand{\Qdup}{\Q^{\textrm{dup}}} % the duplicated word
\newcommand{\Qexm}{\Q^{\textrm{ex}}} % the toy example word
\newcommand{\cexm}{c^{\textrm{ex}}} % the Coxeter element example
\newcommand{\Fexm}{F^{\textrm{ex}}} % the facet example
\newcommand{\Xexm}{X^{\textrm{ex}}} % the cluster example
\newcommand{\Lexm}{L^{\textrm{ex}}} % the noncrossing subspace example

\DeclareMathOperator{\conv}{conv} % convex hull
\DeclareMathOperator{\vect}{vect} % linear span
\DeclareMathOperator{\cone}{cone} % cone hull
\DeclareMathOperator{\inv}{inv} % inversion set

\usepackage{todonotes}

\newcommand{\fref}[1]{Figure~\ref{#1}} % reference figures
\newcommand{\ie}{\textit{i.e.}~} % id est
\newcommand{\eg}{\textit{e.g.}~} % exempli gratia
\newcommand{\viceversa}{vice versa} % vice versa
\newcommand{\ordinal}{\textsuperscript{th}} % th for ordinals
\definecolor{darkblue}{rgb}{0,0,0.7} % darkblue color
\newcommand{\darkblue}{\color{darkblue}} % darkblue command
\newcommand{\defn}[1]{\emph{\darkblue #1}} % emphasis of a definition
\newcommand{\para}[1]{\medskip\paragraph{\textbf{#1}}} % paragraphs

\begin{document}

\begin{abstract}
We generalize the brick polytope of V.~Pilaud and F.~Santos to spherical subword complexes for finite Coxeter groups.
This construction provides polytopal realizations for a certain class of subword complexes containing all cluster complexes of finite types.
For the latter, the brick polytopes turn out to coincide with the known realizations of generalized associahedra, thus opening new perspectives on these constructions.
This new approach yields in particular the vertex description of generalized associahedra, a Minkowski sum decomposition into Coxeter matroid polytopes, and a combinatorial description of the exchange matrix of any cluster in a finite type cluster algebra.

\medskip
\noindent
{\sc keywords.}
Coxeter--Catalan combinatorics, subword complexes, cluster complexes, generalized associahedra, Cambrian lattices, Cambrian fans
\end{abstract}

\vspace*{-.4cm}

\maketitle

\vspace{-.8cm}

\tableofcontents

\vspace{-1.1cm}

\section*{Introduction}
\label{sec:introduction}

The classical associahedron appears in various contexts in mathematics, in particular in discrete geometry and algebraic combinatorics.
As a geometric object, the associahedron provides a polytopal realization of the simplicial complex of the dissections of a convex polygon.
As such, it comes within the scope of S.~Fomin and A.~Zelevinsky's theory of finite type cluster algebras and their famous classification according to the celebrated Cartan-Killing classification~\cite{FominZelevinsky-ClusterAlgebrasI, FominZelevinsky-ClusterAlgebrasII}.
Indeed, the associahedron is a particular case of the generalized associahedra constructed by F.~Chapoton, S.~Fomin, and A.~Zelevinsky~\cite{ChapotonFominZelevinsky} to obtain polytopal realizations of the finite type cluster complexes~\cite{FominZelevinsky-ClusterAlgebrasII}.

Although the algebraic framework does not generalize, the combinatorial dynamics defining cluster algebras extend beyond crystallographic root systems and their Weyl groups, to all finite root systems and their Coxeter groups.
This generalization appears in the work of N.~Reading and his collaborations with D.~Speyer on Coxeter sortable elements and Cambrian lattices and fans~\cite{Reading-latticeCongruences, Reading-CambrianLattices, Reading-sortableElements, ReadingSpeyer}.
Using this framework, C.~Hohlweg, C.~Lange, and H.~Thomas constructed different realizations of the generalized associahedra~\cite{HohlwegLangeThomas} extending ideas from~\cite{HohlwegLange}.
In~\cite{Stella}, S.~Stella generalizes the approach of~\cite{ChapotonFominZelevinsky} and shows that the resulting realization of the cluster complex coincides with that of~\cite{HohlwegLangeThomas}.

\medskip

In~\cite{PilaudPocchiola}, the first author and M.~Pocchiola studied pseudoline arrangements on sorting networks as a combinatorial framework for (generalizations of) triangulations of convex polygons.
Together with F.~Santos, he defined the brick polytope of a sorting network~\cite{PilaudSantos-multitriangulations,PilaudSantos-brickPolytope}.
This construction provides polytopal realizations of the simplicial complexes associated to a certain class of sorting networks.
In particular, it turns out that all type~$A$ associahedra of C.~Hohlweg and C.~Lange~\cite{HohlwegLange} appear as brick polytopes for well-chosen sorting networks.

Independently in~\cite{Stump}, the second author connected (multi-)triangulations of a convex polygon with subword complexes of type~$A$ (see also a generalization developed in~\cite{SerranoStump}).
Such subword complexes were introduced by A.~Knutson and E.~Miller in the context of Gr\"obner geometry of Schubert varieties~\cite{KnutsonMiller-GroebnerGeometry}.
The definition was then extended to all Coxeter groups in~\cite{KnutsonMiller-subwordComplex}.
They proved that these simplicial complexes are either topological balls or spheres, and raised the question of realizing spherical subword complexes as boundary complexes of convex polytopes.

Recently, C.~Ceballos, J.-P.~Labb\'e, and the second author~\cite{CeballosLabbeStump} extended the subword complex interpretations of~\cite{Stump} and~\cite{PilaudPocchiola} for triangulations and multitriangulations of convex polygons to general finite types.
They proved in particular that all finite type cluster complexes are isomorphic to certain well-chosen subword complexes.

\medskip

In this paper, we define the brick polytope of any subword complex for a finite Coxeter group, and study the class of those subword complexes that are realized by their brick polytope.
We provide a simple combinatorial characterization of this class, and deduce that it contains all finite type cluster complexes.
For the latter, we show that the brick polytopes are indeed translates of the generalized associahedra of~\cite{HohlwegLangeThomas}.
We then extend many relevant combinatorial properties of cluster complexes and generalized associahedra to the general situation. In particular, we relate the normal fan of the brick polytope to the Coxeter fan, and the graph of the brick polytope to a quotient of the Hasse diagram of the weak order.

\medskip

This broader approach to finite type cluster complexes bypasses many difficulties encountered in the previously known approaches.
In particular, we can derive directly that the brick polytope indeed realizes the cluster complex, without relying on Cambrian lattices and fans.
Besides that, this approach with subword complexes can be used as a central hub sitting between various combinatorial constructions (cluster complexes, Coxeter sortable elements, and noncrossing partitions), explaining their combinatorial and geometric connections.
More concretely, it provides
\begin{itemize}
\item an explicit vertex description of generalized associahedra,
\item a previously unknown decomposition of generalized associahedra into Min\-kowski sums of Coxeter matroid polytopes,
\item a complementary geometric description of the Cambrian lattices and fans,
\item a simple way to understand the connection between Coxeter sortable elements and noncrossing partitions, as well as the connection between Cambrian lattices and noncrossing partition lattices,
\item an interpretation of the mutation matrix associated to a cluster in a cluster complex purely in terms of subword complexes.
\end{itemize}

Along the paper, we rephrase our main definitions in the classical situation of type~$A$ to show that they match the definitions in~\cite{PilaudSantos-brickPolytope}.
More details on the type~$A$ interpretation of our construction can be found therein.
The reader may as well skip this particular interpretation in type~$A$ for a more compact presentation.

\medskip

To close the introduction, we want to point out several subsequent constructions based on the results in previous versions of the present paper.
Namely,
\begin{itemize}
\item in~\cite{PilaudStump-ELlabeling}, the authors further study EL-labelings and canonical spanning trees for subword complexes, with applications to generation and order theoretic properties of the subword complex, and discuss their relation to alternative EL labelings of Cambrian lattices in~\cite{KallipolitiMuhle},
\item in~\cite{PilaudStump-barycenter}, the authors show that the vertex barycenter of generalized associahedra coincide with that of their corresponding permutahedra, using the vertex description of generalized associahedra presented here,
\item in~\cite{CeballosPilaud}, the first author and C.~Ceballos provide a combinatorial description of the denominator vectors based on the subword complex approach,
\item in his recent dissertation~\cite{Williams}, N.~Williams provides an amazing conjecture together with a huge amount of computational evidence that the present approach to cluster complexes and generalized associahedra as well yields a type-independent and explicit bijection to nonnesting partitions (see also~\cite{AST2010} for further background on this connection).
\end{itemize}

%%%%%%%%%%%%%%%%%%%%%%%%%%%%%%%%%%%%%%
%%%%%%%%%%%%%%%%%%%%%%%%%%%%%%%%%%%%%%

\section{Main results}
\label{sec:mainResults}

\enlargethispage{.1cm}
In this section, we summarize the main results of the paper for the convenience of the reader.
Precise definitions and detailed proofs appear in further sections.

%%%%%%%%%%%%%%%%%%%%%%%%%%%%%%%%%%%%%%

\subsection{Main definitions}

Consider a finite Coxeter system~$(W,S)$, with simple roots $\Delta \eqdef \set{\alpha_s}{s \in S}$, with simple coroots~$\Delta^\vee \eqdef \set{\alpha_s^\vee}{s \in S}$, and with fundamental weights $\nabla \eqdef \set{\omega_s}{s \in S}$.
The group~$W$ acts on the vector space~$V$ with basis~$\Delta$.
For a word~$\Q \eqdef \q_1 \cdots \q_\sizeQ$ on~$S$ and an~element~$\rho \in W$, A.~Knutson and E.~Miller~\cite{KnutsonMiller-subwordComplex} define the subword complex~$\subwordComplex(\Q,\rho)$ to be the simplicial complex of subwords of~$\Q$ whose complement contains a reduced expression of~$\rho$.
As explained in Section~\ref{subsec:subwordComplex}, we assume without loss of generality that~$\rho = w_\circ = \delta(\Q)$, where $w_\circ$ denotes the longest element of~$W$ and $\delta(\Q)$ denotes the Demazure product of~$\Q$.
The subword complex~$\subwordComplex(\Q) \eqdef \subwordComplex(\Q, w_\circ)$ is then a simplicial sphere~\cite{KnutsonMiller-subwordComplex}.

After recalling some background on Coxeter groups and subword complexes in Section~\ref{sec:coxeterGroupsSubwordComplex}, we study in Sections~\ref{sec:rootConfiguration} and~\ref{sec:brickPolytope} the root and the weight functions.
Associate to any facet~$I$ of the subword complex~$\subwordComplex(\Q)$ a root function~$\Root{I}{\cdot} : [\sizeQ] \to W(\Delta)$ and a weight function~$\Weight{I}{\cdot} : [\sizeQ] \to W(\nabla)$ defined by
$$\Root{I}{k} \eqdef \wordprod{\Q}{[k-1] \ssm I}(\alpha_{q_k}) \quad \text{and} \quad \Weight{I}{k} \eqdef \wordprod{\Q}{[k-1] \ssm I}(\omega_{q_k}),$$
where~$\wordprod{\Q}{X}$ denotes the product of the reflections~$q_x \in \Q$, for~$x \in X$, in the order given by~$\Q$.
The root function locally encodes the flip property in the subword complex: each facet adjacent to~$I$ in~$\subwordComplex(\Q)$ is obtained by exchanging an element~${i \in I}$ with the unique element~$j \notin I$ such that~$\Root{I}{j} \in \{ \pm \Root{I}{i} \}$.
After this exchange, the root function is updated by a simple application of~$s_{\Root{I}{i}}$.

We use the root function to define the \defn{root configuration} of~$I$ as the multiset
$$\Roots{I} \eqdef \bigmultiset{\Root{I}{i}}{i \in I}.$$
On the other hand, we use the weight function to define the \defn{brick vector} of~$I$ as
$$\brickVector(I) \eqdef \sum_{k \in [\sizeQ]} \Weight{I}{k},$$
and the \defn{brick polytope} of~$\Q$ as the convex hull of the brick vectors of all facets of the subword complex~$\subwordComplex(\Q)$,
$$\brickPolytope(\Q) \eqdef \conv \bigset{\brickVector(I)}{I \text{ facet of } \subwordComplex(\Q)}.$$
In the particular situation of type~$A$ Coxeter groups, these definitions match that of~\cite{PilaudSantos-brickPolytope} for brick polytopes of sorting networks.

%%%%%%%%%%%%%%%%%%%%%%%%%%%%%%%%%%%%%%

\subsection{Polytopal realizations of root independent subword complexes}

In this paper, we focus on \defn{root independent} subword complexes, for which the root configuration~$\Roots{I}$ of a facet~$I$ (or equivalently, of all facets) is linearly independent.
For these subword complexes, we obtain the main results of this paper in Section~\ref{sec:brickPolytope}.

\begin{theorem}
A root independent subword complex $\subwordComplex(\Q)$ is realized by the polar of its brick polytope~$\brickPolytope(\Q)$.
\end{theorem}

This theorem relies on elementary lemmas which relate the root configuration of a facet~$I$ of~$\subwordComplex(\Q)$ with the edges of~$\brickPolytope(\Q)$ incident to~$\brickVector(I)$, as follows.

\begin{proposition}
The cone of the brick polytope~$\brickPolytope(\Q)$ at the brick vector~$\brickVector(I)$ of a facet~$I$ of~$\subwordComplex(\Q)$ is generated by the negative of the root configuration~$\Roots{I}$, \ie
$$\cone \bigset{\brickVector(J) - \brickVector(I)}{J \text{ facet of } \subwordComplex(\Q)} = \cone \bigset{-\Root{I}{i}}{i \in I}.$$
\end{proposition}

%%%%%%%%%%%%%%%%%%%%%%%%%%%%%%%%%%%%%%

\subsection{Further properties}

Further combinatorial and geometric properties of our construction are developed in Section~\ref{sec:properties}.
We first define a map~$\projectionMap$ from the Coxeter group~$W$ to the facets of the subword complex~$\subwordComplex(\Q)$ which associates to an element~$w \in W$ the unique facet~$\projectionMap(w)$ of~$\subwordComplex(\Q)$ whose root configuration~$\Roots{\projectionMap(w)}$ is a subset of~$w(\Phi^+)$.
In other words, $\projectionMap(w)$ is the facet of~$\subwordComplex(\Q)$ whose brick vector~$\brickVector(\projectionMap(w))$ maximizes the functional $x \mapsto \dotprod{w(q)}{x}$, where~$q$ is any point in the interior of the fundamental chamber~$\fundamentalChamber$ of~$W$.
The map $\projectionMap$ thus enables us to describe the normal fan of the brick polytope.
It is obtained by gluing the chambers of the Coxeter fan according to the fibers of~$\projectionMap$ as follows.

\begin{proposition}
The normal cone of~$\brickVector(I)$ in~$\brickPolytope(\Q)$ is the union of the chambers~$w(\fundamentalChamber)$ of the Coxeter fan of~$W$ given by the elements~$w \in W$ with~$\projectionMap(w) = I$.
\end{proposition}

We then relate the graph of the brick polytope~$\brickPolytope(\Q)$ to the Hasse diagram of the weak order on~$W$.
Let~$I$ and~$J$ be two adjacent facets of~$\subwordComplex(\Q)$ with~$I \ssm i = J \ssm j$.
We say that the flip from~$I$ to~$J$ is \defn{increasing} if~$i < j$.
We call \defn{increasing flip order} the transitive closure of the graph of increasing flips.

\begin{proposition}
A facet~$I$ is covered by a facet~$J$ in increasing flip order if and only if there exist $w_I \in \projectionMap^{-1}(I)$~and~$w_J \in \projectionMap^{-1}(J)$ such that~$w_I$ is covered by~$w_J$ in weak~order.
\end{proposition}

We study some properties of the fibers of the map~$\projectionMap$ with respect to the weak order: we show that the fibers are closed by intervals and we characterize the fibers which admit a meet or a join.
Surprisingly, we provide an example of root independent subword complexes for which the increasing flip order is not a lattice (contrarily to the Cambrian lattices which occur as increasing flip orders of specific subword complexes).
More combinatorial properties of increasing flip graphs and orders for arbitrary subword complexes, with applications to generation and order theoretic properties of subword complexes, can be found in~\cite{PilaudStump-ELlabeling}.

\medskip
Finally, to close our study of brick polytopes of root independent subword complexes, we provide a Minkowski sum decomposition of these realizations into Coxeter matroid polytopes~\cite{BorovikGelfandWhite2}.

\begin{proposition}
The brick polytope~$\brickPolytope(\Q)$ is the Minkowski sum of the polytopes
$$\brickPolytope(\Q,k) \eqdef \conv \bigset{\Weight{I}{k}}{I \text{ facet of } \subwordComplex(\Q)}$$
over all positions~${k \in [\sizeQ]}$.
Moreover, each summand~$\brickPolytope(\Q,k)$ is a Coxeter matroid polytope as defined in~\cite{BorovikGelfandWhite2}.
\end{proposition}

%%%%%%%%%%%%%%%%%%%%%%%%%%%%%%%%%%%%%%

\subsection{Generalized associahedra}

Section~\ref{sec:generalizedAssociahedra} is devoted to the main motivation and example of brick polytopes.
Fix a Coxeter element~$c$ of~$W$, and a reduced expression~$\sq{c}$ of~$c$.
We denote by~$\cwo{c}$ the $\sq{c}$-sorting word of~$w_\circ$.
As recently proven by~\cite{CeballosLabbeStump}, the subword complex~$\clustercomplex$ is then isomorphic to the $c$-cluster complex of~$W$ as defined by N.~Reading in~\cite{Reading-coxeterSortable}.
Since the subword complex $\subwordComplex(\cw{c})$ is root independent, we can apply our brick polytope construction to obtain polytopal realizations of the cluster complex.

\begin{theorem}
For any reduced expression~$\sq{c}$ of any Coxeter element~$c$ of~$W$, the polar of the brick polytope~$\brickPolytope(\cw{c})$ realizes the subword complex~$\clustercomplex$, and thus the $c$-cluster complex of type~$W$.
\end{theorem}

We prove that brick polytopes are in fact translates of the known realizations of the generalized associahedra~\cite{ChapotonFominZelevinsky, HohlwegLangeThomas, Stella}.
We thus obtain the vertex description of these realizations, and a Minkowski sum decomposition into Coxeter matroid polytopes.
Furthermore, we obtain independent elementary proofs of the following two results, which are the heart of the construction of~\cite{HohlwegLangeThomas}.
We say that an element $w \in W$ is a \defn{singleton} if $\projectionMap^{-1}(\projectionMap(w)) = \{w\}$.
We denote here by~$q \eqdef \sum_s \omega_s$ the sum of all weights of~$\nabla$.

\begin{proposition}
\label{prop:mainAssociahedron}
Up to translation by a vector~$\translation$, the brick polytope~$\brickPolytope(\cw{c})$ is obtained from the balanced $W$-permutahedron~$\Perm$ removing all facets which do not intersect the set $\bigset{w(q)}{w \in W \text{singleton} }.$
\end{proposition}

\begin{proposition}
The following properties are equivalent for an element~$w \in W$:
\begin{enumerate}[(i)]
\item The element $w$ is a singleton.
\item The root configuration~$\Roots{\projectionMap(w)} \eqdef \set{\Root{\projectionMap(w)}{i}}{i \in \projectionMap(w)}$ equals~$w(\Delta)$.
\item The weight configuration~$\Weights{\projectionMap(w)} \eqdef \set{\Weight{\projectionMap(w)}{i}}{i \in \projectionMap(w)}$ equals~$w(\nabla)$.
\item The vertices~$\brickVector(\projectionMap(w))$ of the brick polytope~$\brickPolytope(\cw{c})$ and~$w(q)$ of the balanced $W$-permutahedron~$\Perm$ coincide up to~$\translation$, \ie~${\brickVector(\projectionMap(w)) = \translation + w(q)}$.
\item There exist reduced expressions~$\w$ of~$w$ and~$\sq{c}$ of~$c$ such that~$\w$ is a prefix of~$\cwo{c}$.
\item The complement of~$\projectionMap(w)$ in~$\cw{c}$ is~$\cwo{c}$.
\end{enumerate}
\end{proposition}

With some additional efforts, our vertex description of the brick polytope can be exploited to prove that the vertex barycenter of the generalized associahedron ${\brickPolytope(\cw{c}) - \translation}$ coincides with that of the $W$-permutahedron~$\Perm$.
This property was conjectured in~\cite{HohlwegLangeThomas} and proven in type~$A$ and~$B$ in~\cite{HohlwegLortieRaymond}.
The technical details needed for this proof will appear in a forthcoming paper~\cite{PilaudStump-barycenter}.

\medskip
We then revisit the connection between these polytopal realizations of the cluster complex and the Cambrian lattices and fans defined by N.~Reading and D.~Speyer in~\cite{Reading-CambrianLattices, Reading-sortableElements, ReadingSpeyer}.
We first describe in detail the bijective connections between the facets of the subword complex~$\clustercomplex$ and different Coxeter-Catalan families studied in~\cite{Reading-sortableElements}: $c$-clusters, $c$-sortable elements, and $c$-noncrossing partitions and subspaces.
In particular, combining the bijection of~\cite{CeballosLabbeStump} from the subword complex~$\clustercomplex$ to the $c$-cluster complex with the bijection of~\cite{Reading-sortableElements} from $c$-sortable elements in~$W$ to the cluster complex, we describe a bijection~$\facetSortable{c}$ from facets of~$\clustercomplex$ to $c$-sortable elements.
We then prove that the set~$\Skips{c}{w}$ defined in~\cite{ReadingSpeyer} coincides with the root configuration~$\Roots{\facetSortable{c}^{-1}(w)}$, and that the map $\facetSortable{c}^{-1}$ is the restriction of~$\projectionMap$ to $c$-sortable elements.
This observation yields the following results.

\begin{proposition}
If~$v,w \in W$ and~$v$ is $c$-sortable, then $\projectionMap(w) = \projectionMap(v)$ if and only if~$v$ is the maximal $c$-sortable element below~$w$ in weak order.
\end{proposition}

\begin{corollary}
\label{coro:mainCambrianLattice}
For the word~$\cw{c}$, the bijection~$\facetSortable{c}$ is an isomorphism between the Cambrian lattice on $c$-sortable elements in~$W$ and the increasing flip order on the facets of~$\clustercomplex$.
\end{corollary}

\begin{corollary}
\label{coro:mainCambrianFan}
The normal fan of the brick polytope~$\brickPolytope(\cw{c})$ coincides with the $c$-Cambrian fan.
\end{corollary}

Although Corollary~\ref{coro:mainCambrianLattice} follows from~\cite{IgusaSchiffler,IngallsThomas} for simply laced types and Corollary~\ref{coro:mainCambrianFan} could be deduced from the connection between the brick polytope~$\brickPolytope(\cw{c})$ and the $c$-associahedron of~\cite{HohlwegLangeThomas}, the subword complex approach provides new independent proofs of these results bypassing most technicalities of the previous approaches.

\medskip
Finally, to complete the interpretation of cluster complexes as subword complexes, we provide the following description of the mutation matrix associated to a cluster in the $c$-cluster complex, or equivalently, to a facet of the subword complex~$\clustercomplex$.

\begin{theorem}
Let~$I$ be a facet of the cluster complex~$\clustercomplex$.
The exchange matrix~$\mutmatrixold{I}$ indexed by positions in~$I$ is given by
\begin{align*}
  \mutmatrixold{I}_{uv} = 
  \begin{cases}
    -\bigdotprod{\alpha_{q_u}^\vee}{\wordprod{\Q}{[u,v] \ssm I}(\alpha_{q_v})} & \text{ if } u < v \\[1pt]
    \phantom{-}\bigdotprod{\alpha_{q_v}}{\wordprod{\Q}{[v,u] \ssm I}(\alpha_{q_u}^\vee)} & \text{ if } u > v \\[1pt]
    \phantom{-}0 & \text{ if } u = v
  \end{cases}
\end{align*}
where~$\Q \eqdef \q_1 \cdots \q_\sizeQ = \cw{c}$.
\end{theorem}

%%%%%%%%%%%%%%%%%%%%%%%%%%%%%%%%%%%%%%
%%%%%%%%%%%%%%%%%%%%%%%%%%%%%%%%%%%%%%

\section{Background on Coxeter groups and subword complexes}
\label{sec:coxeterGroupsSubwordComplex}

%%%%%%%%%%%%%%%%%%%%%%%%%%%%%%%%%%%%%%

\subsection{Finite Coxeter groups}
\label{subsec:coxeterGroups}

We recall here classical notions on finite Coxeter groups.
See~\cite{Humphreys, Humphreys1978} for more details.

\subsubsection{Coxeter systems}
Let~$(V,\dotprod{\cdot}{\cdot})$ be an $n$-dimensional Euclidean vector space.
For any vector~$v \in V\ssm 0$, we denote by~$s_v$ the reflection interchanging $v$ and $-v$ while fixing the orthogonal hyperplane pointwise.
Remember that~$ws_v = s_{w(v)}w$ for any non-zero vector~$v \in V$ and any orthogonal transformation~$w$ of~$V$.

We consider a \defn{finite Coxeter group}~$W$ acting on~$V$, that is, a finite group generated by reflections.
The set of all \defn{reflections} in $W$ is denoted by $R$.
The \defn{Coxeter arrangement} of~$W$ is the collection of all reflecting hyperplanes.
Its complement in~$V$ is a union of open polyhedral cones.
Their closures are called \defn{chambers}.
The \defn{Coxeter fan} is the polyhedral fan formed by the chambers together with all their faces.
This fan is \defn{complete} (its cones cover~$V$) and \defn{simplicial} (all cones are simplicial), and we can assume without loss of generality that it is \defn{essential} (the intersection of all chambers is reduced to the origin).

We fix an arbitrary chamber $\fundamentalChamber$ which we call the \defn{fundamental chamber}.
The \defn{simple reflections} of~$W$ are the $n$ reflections orthogonal to the facet defining hyperplanes of~$\fundamentalChamber$.
The set~$S \subseteq R$ of simple reflections generates~$W$.
In particular, $R = \set{wsw^{-1}}{w \in W, s \in S}$.
The pair~$(W,S)$ forms a \defn{Coxeter system}.

For inductive procedures, set~$W_{\langle s \rangle}$ for~$s \in S$ to be the \defn{parabolic subgroup} of~$W$ generated by~$S \ssm s$.
It is a Coxeter group with Coxeter system $(W_{\langle s \rangle},S \ssm s)$.

\subsubsection{Roots and weights}
\label{sec:rootsandweights}
For simple reflections~$s,t \in S$, denote by~$m_{st}$ the order of the product~${st \in W}$.
We fix a \defn{generalized Cartan matrix} for~$(W,S)$, \ie a matrix~$(a_{st})_{s,t \in S}$ such that $a_{ss} = 2$, $a_{st} \le 0$, $a_{st}a_{ts} = 4 \cos^2(\frac{\pi}{m_{st}})$ and ${a_{st} = 0 \Leftrightarrow a_{ts} = 0}$ for all~$s \ne t \in S$.
We can associate to each simple reflection~$s$ a \defn{simple root}~$\alpha_s \in V$, orthogonal to the reflecting hyperplane of~$s$ and pointing toward the half-space containing~$\fundamentalChamber$, in such a way that $s(\alpha_t) = \alpha_t - a_{st}\alpha_s$ for all~$s,t \in S$.
The set of all simple roots is denoted by~$\Delta \eqdef \set{\alpha_s}{s \in S}$.
The orbit $\Phi \eqdef \set{w(\alpha_s)}{w \in W, s \in S}$ of~$\Delta$ under~$W$ is a \defn{root system} for~$W$.
It is invariant under the action of~$W$ and contains precisely two opposite roots orthogonal to each reflecting hyperplane of~$W$.

The set $\Delta$ of simple roots forms a linear basis of~$V$ (since we assumed $W$ to act essentially on~$V$).
The root system~$\Phi$ is the disjoint union of the \defn{positive roots}~$\Phi^+ \eqdef \Phi \cap \R_{\ge 0}[\Delta]$ (non-negative linear combinations of the simple roots) and the \defn{negative roots} $\Phi^- \eqdef -\Phi^+$.
In other words, the positive roots are the roots whose scalar product with any vector of the interior of the fundamental chamber~$\fundamentalChamber$ is positive, and the simple roots form the basis of the cone generated by~$\Phi^+$.
Each reflection hyperplane is orthogonal to one positive and one negative root.
For a reflection $s \in R$, we set $\alpha_s$ to be the unique positive root orthogonal to the reflection hyperplane of~$s$, \ie such that $s = s_{\alpha_s}$.
For $\alpha \in \Delta$, we also denote by $\Phi_{\langle \alpha \rangle} \eqdef \Phi \cap \vect (\Delta \ssm \alpha)$
the \defn{parabolic subroot system} of $\Phi$ generated by $\Delta \ssm\{ \alpha \}$ with corresponding parabolic subgroup~$W_{\langle s_\alpha \rangle}$.

We denote by~$\alpha_s^\vee \eqdef 2 \alpha_s / \dotprod{\alpha_s}{\alpha_s}$ the \defn{coroot} corresponding to~$\alpha_s \in \Delta$, and we let~$\Delta^\vee \eqdef \set{\alpha^\vee_s}{s \in S}$ denote the coroot basis.
The vectors of its dual basis~$\nabla \eqdef \set{\omega_s}{s \in S}$ are called \defn{fundamental weights}.
In other words, the weights of~$(W,S)$ are defined by~${\dotprod{\alpha_s}{\omega_t}=\delta_{s=t}\dotprod{\alpha_s}{\alpha_s}/2}$ for all~$s,t \in S$.
Thus, the Coxeter group~$W$ acts on the weight basis by~$s(\omega_t) = \omega_t-\delta_{s=t}\alpha_s$.
Note that the transposed Cartan matrix transforms the weight basis into the root basis, \ie ${\alpha_s = \sum_{t \in S} a_{ts}w_t}$ for any~$s \in S$.
Geometrically, the weight~$\omega_s$ gives the direction of the ray of the fundamental chamber~$\fundamentalChamber$ not contained in the reflecting hyperplane~of~$s$.

The Coxeter group~$W$ is said to be \defn{crystallographic} if it stabilizes a lattice of~$V$.
This can only happen if all entries of the Cartan matrix are integers.
Reciprocally, if all entries of the Cartan matrix are integers, then the lattice generated by the simple roots~$\Delta$ is fixed by the Coxeter group~$W$.

\subsubsection{Coxeter permutahedra}
The \defn{$W$-permutahedron}~$\Perm[q][W]$ is the convex hull of the orbit under~$W$ of a point~$q$ in the interior of the fundamental chamber~$\fundamentalChamber$ of~$W$.
Although its geometry depends on the choice of the basepoint~$q$, its combinatorial structure does not.
Its normal fan is the \defn{Coxeter fan} of~$W$.
The facet of~$\Perm[q][W]$ orthogonal to~$w(\omega_s)$ is defined by the inequality $\dotprod{w(\omega_s)}{x} \le \dotprod{\omega_s}{q}$ and is supported by the hyperplane~$w(q + \vect(\Delta \ssm \alpha_s))$.
If~$q = \sum_{s \in S} \omega_s$, then $\Perm[q][W]$ is (a translate of) the Minkowski sum of all positive roots (each considered as a one-dimensional polytope).
We then call it the \defn{balanced $W$-permutahedron}, and denote it by~$\Perm$.
We refer to~\cite{Hohlweg} for further properties of $W$-permutahedra.

\enlargethispage{.1cm}
\subsubsection{Words on~$S$}
The \defn{length}~$\length(w)$ of an element~$w \in W$ is the length of the smallest expression of~$w$ as a product of the generators in~$S$.
An expression~${w = w_1 \cdots w_\ell}$ with~$w_1,\dots,w_\ell \in S$ is called \defn{reduced} if $\ell = \length(w)$.
Geometrically, the length of~$w$ is the cardinality of the \defn{inversion set} of~$w$, defined as the set~$\inv(w) \eqdef \Phi^+ \, \cap \, w(\Phi^-)$ of positive roots sent to negative roots by~$w^{-1}$.
Indeed, the inversion set of~$w$ can be written as $\inv(w) = \big\{\alpha_{w_1}, w_1(\alpha_{w_2}), \dots, w_1w_2 \cdots w_{p-1}(\alpha_{w_\ell})\big\}$ for any reduced expression~$w = w_1 \cdots w_\ell$ of~$w$.
Note that for any~$w \in W$ and~$s \in S$, we have~$\length(ws) = \length(w) + 1$ if~$w(\alpha_s) \in \Phi^+$ and~$\length(ws) = \length(w) - 1$ if~$w(\alpha_s) \in \Phi^-$.
Let~$w_\circ$ denote the unique \defn{longest element} in~$W$.
It sends all positive roots to negative ones.

The \defn{(right) weak order} on~$W$ is the partial order~$\le$ defined by~$u \le w$ if there exists~$v \in W$ such that~$uv = w$ and~$\length(u)+\length(v)=\length(w)$.
Equivalently, the weak order corresponds to the inclusion order on inversion sets: $u \le w$ if and only if ${\inv(u) \subseteq \inv(w)}$.
It defines a lattice structure on the elements of~$W$ with minimal element being the identity~$e \in W$ and with maximal element being~$w_\circ$.
Its Hasse diagram is the graph of the $W$-permutahedron, oriented by a linear function from~$e$~to~$w_\circ$.
The elements of~$W$ which cover a given element~$w \in W$ in weak order are precisely the products~$ws$ for which~$w(\alpha_s) \in \Phi^+$.

We denote by~$S^*$ the set of words on the alphabet~$S$.
To avoid confusion, we denote with a roman letter~$\sq{s}$ the letter of the alphabet~$S$ corresponding to the single reflection~$s \in S$.
Similarly, we use a roman letter like~$\w$ to denote a word of~$S^*$, and an italic letter like~$w$ to denote its corresponding group element in~$W$.
For example, we write~$\w \eqdef \w_1 \cdots \w_\ell$ meaning that the word~$\w \in S^*$ is formed by the letters~$\w_1,\dots,\w_\ell$, while we write~$w \eqdef w_1 \cdots w_\ell$ meaning that the element~$w \in W$ is the product of the simple reflections~$w_1,\dots,w_\ell$.
The \defn{Demazure product} on the Coxeter system~$(W,S)$ is the function~$\DemazureProduct : S^* \to W$ defined inductively~by
$$\DemazureProduct(\sq{\varepsilon}) = e \quad\text{and}\quad \DemazureProduct(\Q\sq{s}) = \begin{cases} ws & \text{if } \length(ws) = \length(w) + 1, \\ w & \text{if } \length(ws) = \length(w) - 1, \end{cases}$$
where $w = \DemazureProduct(\Q)$ denotes the Demazure product of $\Q$.
As its name suggests, $\DemazureProduct(\Q)$ corresponds to the product of the letters of~$\Q$ in the Demazure algebra of~$(W,S)$.

\subsubsection{Examples}
We complete this section with classical examples of Coxeter groups.

\begin{example}[Type~$I_2(m)$ --- Dihedral groups]
The \defn{dihedral group} of isometries of a regular $m$-gon is a Coxeter group denoted~$I_2(m)$.
See \fref{fig:dihedral} for illustrations.

\begin{figure}[h]
  \capstart
  \centerline{\includegraphics[width=.9\textwidth]{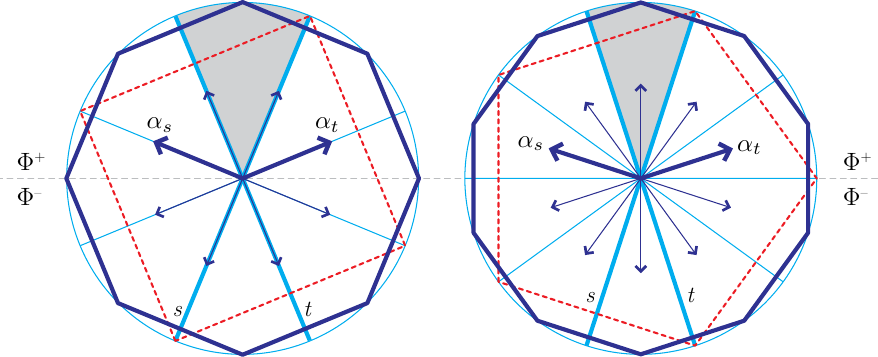}}
  \caption{The dihedral groups~$I_2(4)$ and~$I_2(5)$ and their permutahedra.
  The fundamental chamber is the topmost one (colored), the simple roots are those just above the horizontal axis (in bold), and the positive roots are all roots above the horizontal axis.}
  \label{fig:dihedral}
\end{figure}
\end{example}

\begin{figure}[p]
  \capstart
  \centerline{\includegraphics[scale=.63]{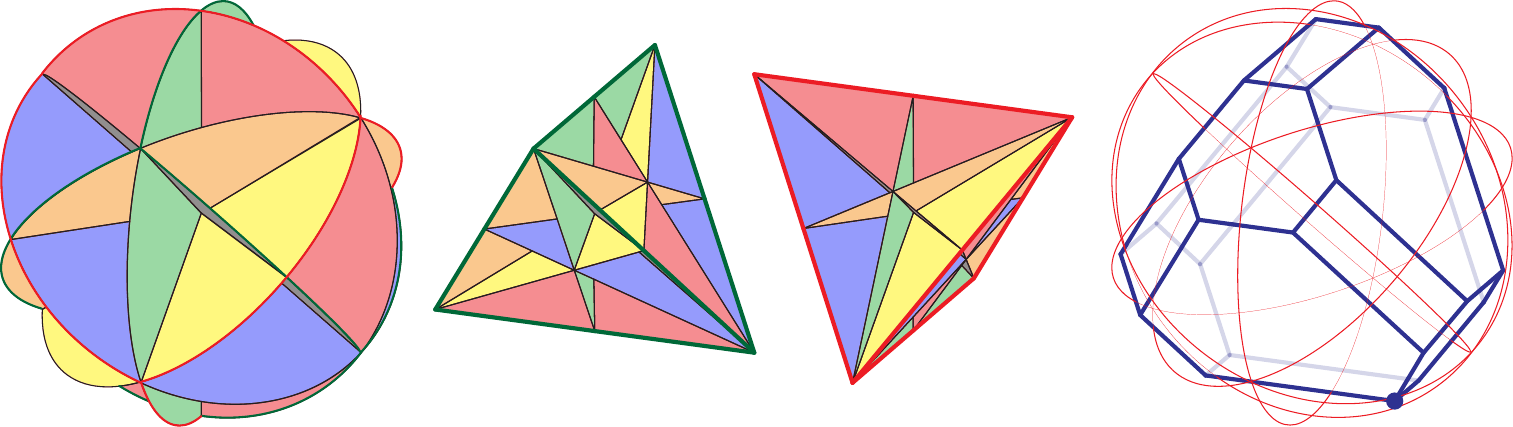}}
  \caption{The $A_3$-arrangement, intersected with the ball (left) and with the corresponding two polar regular tetrahedra (middle).
  The $A_3$-permutahedron is a \defn{truncated octahedron} (right).}
  \label{fig:typeAarrangement}
  \vspace{.2cm}
\end{figure}
\begin{figure}[p]
  \capstart
  \centerline{\includegraphics[scale=.63]{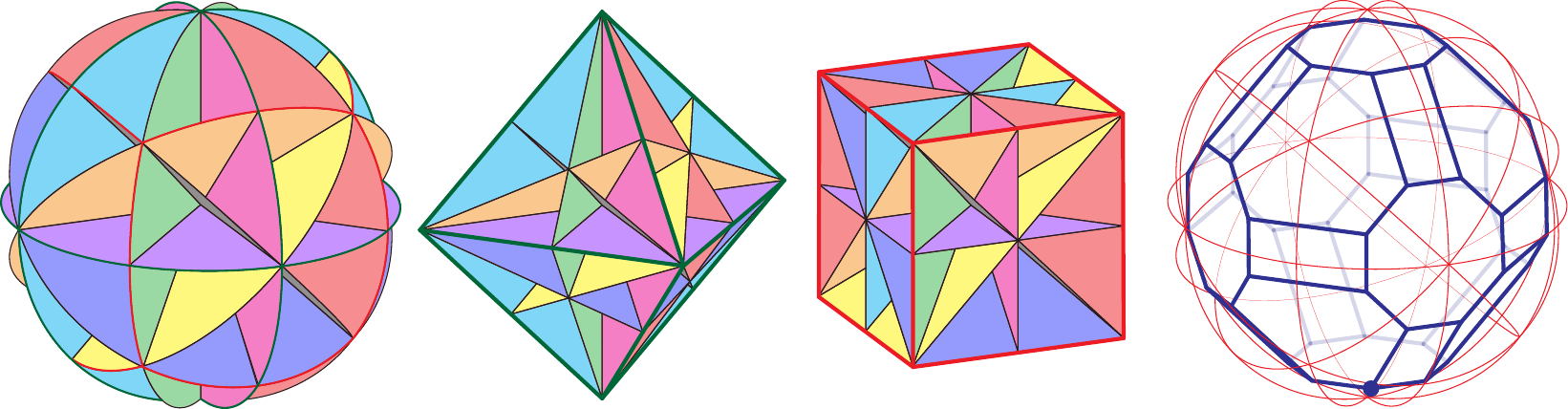}}
  \caption{The $B_3$-arrangement, intersected with the ball (left) and with the regular cube and the regular octahedron (middle).
  The $B_3$-permutahedron is a \defn{great rhombicuboctahedron} (right).}
  \label{fig:typeBarrangement}
  \vspace{.2cm}
\end{figure}
\begin{figure}[p]
  \capstart
  \centerline{\includegraphics[scale=.63]{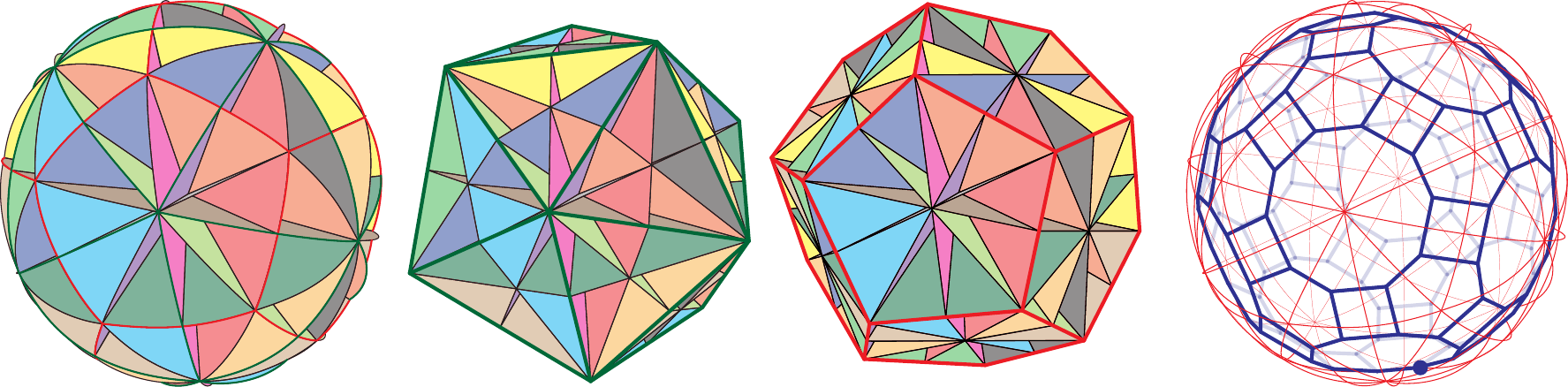}}
  \caption{The $H_3$-arrangement, intersected with the ball (left) and with the regular icosahedron and the regular dodecahedron (middle). The $H_3$-permutahedron is a \defn{great rhombicosidodecahe}\-\defn{dron} (right).}
  \label{fig:typeHarrangement}
\end{figure}

\vspace{-.3cm}

\begin{example}[Type~$A$ --- Symmetric groups]
\label{exm:typeASymmetryGroups}
The symmetric group~$\fS_{n+1}$, acting on the hyperplane $\one^\perp \eqdef \set{x \in \R^{n+1}}{\dotprod{\one}{x} = 0}$ by permutation of the coordinates, is the reflection group of \defn{type}~$A_n$.
It is the group of isometries of the standard $n$-dimensional regular simplex $\conv \{e_1,\dots,e_{n+1}\}$.
Its reflections are the transpositions of~$\fS_{n+1}$, and we can choose the adjacent transpositions~${\tau_p \eqdef (p,p+1)}$ to be the simple reflections.
A root system is given by~$\Phi \eqdef \set{e_p-e_q}{p \neq q \in [n+1]}$.
Its set of simple roots is ${\Delta \eqdef \set{e_{p+1}-e_p}{p \in [n]}}$, and its set of fundamental weights is~$\nabla \eqdef \set{\sum_{q > p} e_q}{p \in [n]}$.
Note that, even if they do not live in the hyperplane~$\one^\perp$, we have chosen these weights to match usual conventions.
For example, $\sum \nabla = \sum_{q \in [n+1]} (q-1)e_q$ and the balanced $A_n$-permutahedron is the classical permutahedron obtained as the convex hull of all permutations of~$\{0,\dots,n\}$, regarded as vectors in~$\R^{n+1}$.
Moreover, it enables us to match the presentation of the type~$A$ brick polytope~\cite{PilaudSantos-brickPolytope}.
The interested reader might prefer to project the weights to the hyperplane~$\one^\perp$, and adapt the examples in type~$A$ accordingly.
\fref{fig:typeAarrangement} presents the $A_3$-arrangement and an $A_3$-permutahe\-dron, where we marked the bottommost vertex with a dot to distinguished which chamber will be considered to be the fundamental chamber in further pictures.
\end{example}

\begin{example}[Type~$B_n$ --- Hyperoctahedral groups]
Consider the symmetry group of~$\fS_n$ acting on~$\R^n$ by permutation of the coordinates together with the group~$(\Z_2)^n$ acting on~$\R^n$ by sign change.
The semidirect product~$B_n \eqdef \fS_n \rtimes (\Z_2)^n$ given by this action is a reflection group.
It is the isometry group of the \mbox{$n$-dimen}\-sional regular cross-polytope~$\conv \{\pm e_1,\dots,\pm e_n\}$ and of its polar $n$-dimensional regular cube~$[-1,1]^n$.
Its reflections are the transpositions of~$\fS_n$ and the changes of one sign.
A root system is given by~$\Phi \eqdef \set{\pm e_p \pm e_q}{p<q \in [n]} \cup \set{\pm e_p}{p \in [n]}$, and we can choose~$\Delta \eqdef \{e_1\} \cup \set{e_{p+1}-e_p}{p \in [n-1]}$ for the set of simple roots.
\fref{fig:typeBarrangement} shows the $B_3$-arrangement and a $B_3$-permutahedron, where we marked the bottommost vertex with a dot to distinguished the fundamental~chamber.
\end{example}

\begin{example}[Type~$H_3$ --- Icosahedral group]
The isometry group of the regular icosahedron (and of its polar regular dodecahedron) is a reflection group.
It is isomorphic to the direct product of~$\Z_2$ by the alternating group~$\fA_5$ of even permutations of~$\{1,\dots,5\}$.
\fref{fig:typeHarrangement} presents the $H_3$-arrangement and a $H_3$-permutahedron, where the fundamental chamber is marked with a dot.
\end{example}

%%%%%%%%%%%%%%%%%%%%%%%%%%%%%%%%%%%%%%

\subsection{Subword complexes}
\label{subsec:subwordComplex}

Let~$(W,S)$ be a Coxeter system, let~${\Q \eqdef \q_1 \cdots \q_\sizeQ \in S^*}$ and let~$\rho \in W$.
In \cite{KnutsonMiller-subwordComplex, KnutsonMiller-GroebnerGeometry}, A.~Knutson and E.~Miller define the \defn{subword complex}~$\subwordComplex(\Q,\rho)$ as the pure simplicial complex of subwords of~$\Q$ whose complements contain a reduced expression of~$\rho$.
The vertices of this simplicial complex are labeled by (positions of) the letters in the word~$\Q$.
Note that two positions are different even if the letters of~$\Q$ at these positions coincide.
We denote by~$[\sizeQ] \eqdef \{1,\dots,\sizeQ\}$ the set of positions in~$\Q$.
The facets of the subword complex~$\subwordComplex(\Q,\rho)$ are the complements of the reduced expressions of~$\rho$ in the word~$\Q$.

Using that the subword complex~$\subwordComplex(\Q,\rho)$ is vertex-decomposable~\cite[Theorem~2.5]{KnutsonMiller-subwordComplex}, it is proven in~\cite[Corollary~3.8]{KnutsonMiller-subwordComplex} that it is a topological sphere if~$\rho = \DemazureProduct(\Q)$, and a ball otherwise.
In this paper we only consider spherical subword complexes.
By appending to the word~$\Q$ any reduced expression of~$\DemazureProduct(\Q)^{-1}w_\circ$, we obtain a word~$\Q'$ whose Demazure product is~$\DemazureProduct(\Q') = w_\circ$ and such that the spherical subword complexes~$\subwordComplex(\Q,\DemazureProduct(\Q))$ and~$\subwordComplex(\Q',w_\circ)$ are isomorphic (see~\cite[Theorem~3.7]{CeballosLabbeStump} for details).
Thus, we assume that~$\rho = \DemazureProduct(\Q) = w_\circ$ to simplify the presentation without loss of generality, and we write~$\subwordComplex(\Q)$ instead of~$\subwordComplex(\Q, w_\circ)$ to simplify notations.

For any facet~$I$ of~$\subwordComplex(\Q)$ and any element~$i \in I$, there is a unique facet~$J$ of~$\subwordComplex(\Q)$ and a unique element~$j \in J$ such that~$I \ssm i = J \ssm j$.
We say that~$I$~and~$J$ are \defn{adjacent} facets, and that~$J$ is obtained from~$I$ by~\defn{flipping}~$i$.

\begin{example}[A toy example]
\label{exm:recurrent1}
To illustrate definitions and results throughout the paper, we will follow all details of the example of the word~$\Qexm \eqdef \sq{\tau}_2\sq{\tau}_3\sq{\tau}_1\sq{\tau}_3\sq{\tau}_2\sq{\tau}_1\sq{\tau}_2\sq{\tau}_3\sq{\tau}_1$ whose letters are adjacent transpositions of~$\fS_4$.
The facets of the subword complex~$\subwordComplex(\Qexm)$ are $\{2,3,5\}$, $\{2,3,9\}$, $\{2,5,6\}$, $\{2,6,7\}$, $\{2,7,9\}$, $\{3,4,5\}$, $\{3,4,9\}$, $\{4,5,6\}$, $\{4,6,7\}$, and $\{4,7,9\}$.
This subword complex is thus spherical.
\end{example}

Before looking at various previously considered spherical subword complexes for the symmetric group, we recall two basic and very helpful isomorphisms exhibited in~\cite{CeballosLabbeStump}.
We will later use these isomorphisms in Section~\ref{sec:generalizedAssociahedra} when studying cluster complexes as subword complexes.

\begin{lemma}[{\cite[Proposition~3.8]{CeballosLabbeStump}}]
\label{lem:commutationlemma}
If two words $\Q$ and $\Q'$ coincide up to commutation of consecutive commuting letters, then the subword complexes~$\subwordComplex(\Q)$ and~$\subwordComplex(\Q')$ are isomorphic.
\end{lemma}

For the second property, define the \defn{rotated word} of the word $\Q \eqdef \q_1 \q_2 \cdots \q_\sizeQ$ to be the word~$\rotate{\Q} \eqdef \q_2 \cdots \q_\sizeQ \psi(q_1)$, where $\psi$ denotes the automorphism on the simple reflections~$S$ given by~$\psi(s) \eqdef w_\circ^{-1} s w_\circ$.
Define the \defn{rotation operator} to be the shift $i \mapsto i-1$ where~$0$ and~$\sizeQ$ are identified.

\begin{lemma}[{\cite[Proposition~3.9]{CeballosLabbeStump}}]
\label{lem:rotationlemma}
The rotation operator induces an isomorphism between the subword complexes $\subwordComplex(\Q)$ and $\subwordComplex(\rotate{\Q})$.
\end{lemma}

These two lemmas provide a way to think about spherical subword complexes as
\begin{itemize}
\item being attached to an acyclic directed graph with vertices being the letters of $\Q$, and with oriented edges being obtained as the transitive reduction of the oriented edges between noncommuting letters in $\Q$, and as
\item ``living on a M\"obius strip'' in the sense that $\subwordComplex(\Q)$ can be thought of being attached to the biinfinite word $\cdots \Q \, \psi(\Q) \, \Q \, \psi(\Q) \, \cdots$, and where a combinatorial model for this complex is obtained by looking at any window of size~$\sizeQ$ in this biinfinite word.
\end{itemize}
We refer the reader to~\cite{PilaudPocchiola, CeballosLabbeStump} for a precise presentation of these interpretations.
We now survey various examples of type~$A$ and~$B$ subword complexes previously considered in the literature.

\begin{typeA}[Sorting networks]
\label{rem:typeASortingNetworks}
Let~$W \eqdef \fS_{n+1}$ act on~$\one^\perp$ and $S \eqdef \set{\tau_p}{p \in [n]}$, where~$\tau_p$ denotes the adjacent transposition~$(p,p+1)$.
We can represent the word~$\Q \eqdef \q_1\q_2 \cdots \q_\sizeQ \in S^*$ by a \defn{sorting network}~$\cN_{\Q}$ as illustrated in \fref{fig:network}~(left).
The network~$\cN_\Q$ is formed by~$n+1$ horizontal lines (its \defn{levels}, labeled from bottom to top) together with~$\sizeQ$ vertical segments (its \defn{commutators}, labeled from left to right) corresponding to the letters of~$\Q$.
If $q_k = \tau_p$, the $k$\ordinal{} commutator of~$\cN_{\Q}$ lies between the $p$\ordinal{} and $(p+1)$\ordinal{} levels of~$\cN_{\Q}$.

\begin{figure}[ht]
  \capstart
  \centerline{\includegraphics[width=\textwidth]{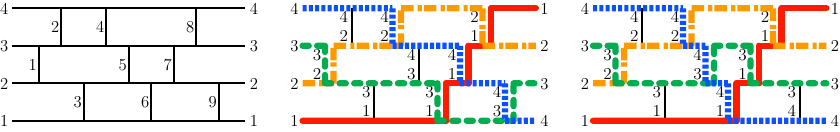}}
  \caption{The sorting network~$\cN_{\Qexm}$ corresponding to the word $\Qexm$ given in Example~\ref{exm:recurrent1} (left) and the pseudoline arrangements supported by the network~$\cN_{\Qexm}$ corresponding to the adjacent facets~$\{2,3,5\}$ (middle) and~$\{2,3,9\}$ (right) of~$\subwordComplex(\Qexm)$.}
  \label{fig:network}
\end{figure}

A \defn{pseudoline} supported by~$\cN_\Q$ is an abscissa monotone path on the network~$\cN_\Q$.
A commutator of~$\cN_\Q$ is a \defn{crossing} between two pseudolines if it is traversed by both pseudolines, and a \defn{contact} if its endpoints are contained one in each pseudoline.
A \defn{pseudoline arrangement}~$\Lambda$ (with contacts) is a set of~$n+1$ pseudolines on~$\cN_\Q$, any two of which have precisely one crossing, possibly some contacts, and no other intersection.
As a consequence of the definition, the pseudoline of~$\Lambda$ which starts at level~$p$ ends at level~$n-p+2$, and is called the $p$\ordinal{} pseudoline of~$\Lambda$.
As illustrated in \fref{fig:network} (middle and right), a facet~$I$ of~$\subwordComplex(\Q)$ is represented by a pseudoline arrangement~$\Lambda_I$ supported by~$\cN_{\Q}$.
Its contacts (resp.~crossings) are the commutators of~$\cN_\Q$ corresponding to the letters of~$I$ (resp.~of the complement of~$I$).

Let~$I$ and~$J$ be two adjacent facets of~$\subwordComplex(\Q)$, with~$I \ssm i = J \ssm j$.
Then~$j$ is the position of the crossing between the two pseudolines of~$\Lambda_I$ which are in contact at position~$i$, and the pseudoline arrangement~$\Lambda_J$ is obtained from the pseudoline arrangement~$\Lambda_I$ by exchanging the contact at~$i$ with the crossing at~$j$.
The flip between the pseudoline arrangements associated to two adjacent facets of~$\subwordComplex(\Qexm)$ is illustrated in \fref{fig:network} (middle and right).

A \defn{brick} of~$\cN_{\Q}$ is a connected component of its complement, bounded on the right by a commutator of~$\cN_{\Q}$.
For $k \in [\sizeQ]$, the $k$\ordinal{} brick is that immediately to the left of the $k$\ordinal{} commutator of~$\cN_{\Q}$.
\end{typeA}

\begin{typeA}[Combinatorial model for families of geometric graphs]
\label{rem:combinatorialModelsGeometricGraphs}
As pointed out in~\cite{Stump, PilaudPocchiola}, type~$A$ subword complexes can be used to provide a combinatorial model for relevant families of geometric graphs.
It relies on the interpretation of these geometric graphs in the line space of the plane.
We do not present this interpretation and refer to~\cite{PilaudPocchiola} for details.
As motivation, we just recall the following four families of geometric graphs which can be interpreted in terms of sorting networks (see \fref{fig:geometricGraphs} for illustrations).

\begin{figure}[b]
  \capstart
  \centerline{\includegraphics[width=1\textwidth]{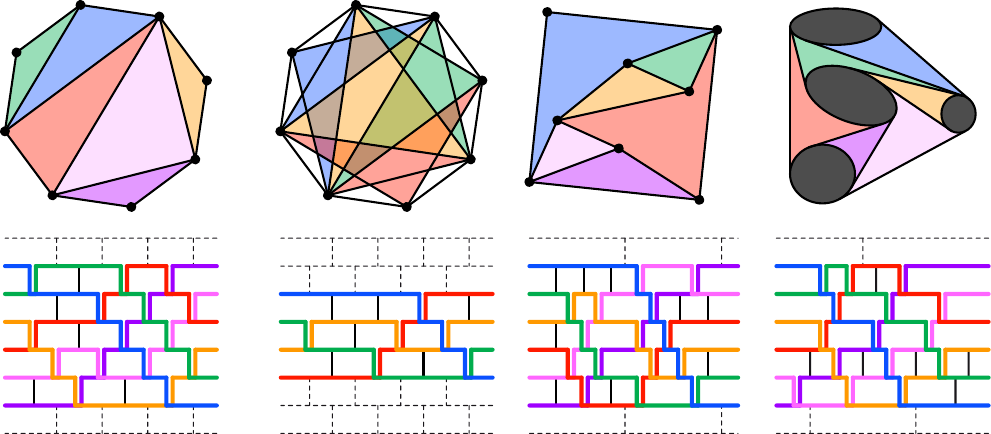}}
  \caption{Sorting networks interpretations of certain geometric graphs: a triangulation of the convex octagon, a $2$-triangulation of the convex octagon, a pseudotriangulation of a point set, and a pseudotriangulation of a set of disjoint convex bodies.}
  \label{fig:geometricGraphs}
\end{figure}

\begin{enumerate}[(i)]
\item Let~$\sq{c}$ be a reduced expression of a Coxeter element~$c$ and $\cwo{c}$ denotes the $\sq{c}$-sorting word of~$w_\circ$ (see Section~\ref{sec:generalizedAssociahedra} for definitions).
Then the subword complex~$\clustercomplex$ is isomorphic to the simplicial complex of crossing-free sets of internal diagonals of a convex $(n+3)$-gon.
Its facets correspond to triangulations of the $(n+3)$-gon, its ridges correspond to flips between them, and its vertices correspond to internal diagonals of the $(n+3)$-gon.
It is the boundary complex of the polar of the \defn{associahedron}.
See \eg~\cite{Lee, Loday, HohlwegLange, CeballosSantosZiegler, TamariFestschrift, LangePilaud}.

\item The subword complex~$\subwordComplex(\sq{c}^k\cwo{c})$ is isomorphic to the simplicial complex of $(k+1)$-crossing free sets of $k$-relevant diagonals of the $(n+2k+1)$-gon, whose facets are $k$-triangulations of the $(n+2k+1)$-gon~\cite{PilaudSantos-multitriangulations, PilaudPocchiola, Stump}.
It is not known whether this simplicial complex can be realized as the boundary complex of a convex polytope.

\item Consider the simplicial complex of pointed crossing-free sets of internal edges of a set~$P$ of~$n+3$ points in general position in the Euclidean plane.
Its facets are pseudotriangulations of~$P$~\cite{PocchiolaVegter, RoteSantosStreinu-survey}, and its ridges are flips between them.
It is isomorphic to a type~$A_n$ subword complex~$\subwordComplex(\Q)$, where~$\Q$ is obtained from the dual pseudoline arrangement of~$P$ by removing its first and last levels~\cite{PilaudPocchiola}.
This simplicial complex is known to be the boundary complex of the polar of the \defn{pseudotriangulation polytope} of~\cite{RoteSantosStreinu-polytope}.

\item For odd~$n$, consider the simplicial complex of crossing-free sets of internal free bitangents of a set~$X$ of~$(n+3)/2$ disjoint convex bodies of the Euclidean plane.
Its facets are pseudotriangulations of~$X$~\cite{PocchiolaVegter}, and its ridges are flips between them.
It is isomorphic to a type~$A_n$ subword complex~$\subwordComplex(\Q)$, where~$\Q$ is obtained from the dual double pseudoline arrangement of~$X$ by removing its first and last levels~\cite{PilaudPocchiola}.
\end{enumerate}
\end{typeA}

\begin{typeA}[Symmetric sorting networks]
Let~$W \eqdef B_n$ and let $S \eqdef \{\chi\} \cup \set{\tau_p}{p \in [n-1]}$, where $\chi$~denotes the operator which changes the sign of the first coordinate and $\tau_p$ denotes the adjacent transposition~$(p,p+1)$.
If the word~$\Q \in S^*$ has $x$~occurrences of~$\chi$, then it can be interpreted as a horizontally symmetric network~$\cN_\Q$ with~$2n+2$ levels and~$2m-x$ commutators: replace each appearance of~$\chi$ by a commutator between the first levels above and below the axis, and each appearance of~$\tau_p$ by a symmetric pair of commutators between the $p$\ordinal{} and $(p+1)$\ordinal{} levels both above and below the axis.
The reduced expressions of~$w_\circ$ in~$\Q$ correspond to the horizontally symmetric pseudoline arrangements supported by~$\cN_\Q$.
Since the horizontal symmetry in the dual space corresponds to the central symmetry in the primal space, the centrally symmetric versions of the geometric objects discussed in the classical situation~\ref{rem:combinatorialModelsGeometricGraphs} can be interpreted combinatorially by well-chosen type~$B$ subword complexes.
\end{typeA}

%%%%%%%%%%%%%%%%%%%%%%%%%%%%%%%%%%%%%%
%%%%%%%%%%%%%%%%%%%%%%%%%%%%%%%%%%%%%%

\section{Root configurations}
\label{sec:rootConfiguration}

%%%%%%%%%%%%%%%%%%%%%%%%%%%%%%%%%%%%%%

\subsection{Definition}

In this section, we describe fundamental properties of the root function and the root configuration of a facet of a subword complex.
The root function was already defined and studied in the work of C.~Ceballos, J.-P.~Labb\'e and the second author~\cite{CeballosLabbeStump}.
Although some statements presented below can already be found in their work, we provide brief proofs for the convenience of the reader.

\begin{definition}
\label{def:roots}
To a facet~$I$ of~$\subwordComplex(\Q)$ and a position~$k \in [\sizeQ]$, associate the~root
$${\Root{I}{k} \eqdef \wordprod{\Q}{[k-1] \ssm I}(\alpha_{q_k})},$$
where~$\wordprod{\Q}{X}$ denotes the product of the reflections~$q_x \in \Q$, for~$x \in X$, in the order given by~$\Q$.
The \defn{root configuration} of the facet~$I$ is the multiset
$$\Roots{I} \eqdef \bigmultiset{\Root{I}{i}}{i \in I}$$
of all roots associated to the elements of~$I$.
\end{definition}

\begin{example}
\label{exm:recurrent2}
In the subword complex~$\subwordComplex(\Qexm)$ with $\Qexm = \sq{\tau}_2\sq{\tau}_3\sq{\tau}_1\sq{\tau}_3\sq{\tau}_2\sq{\tau}_1\sq{\tau}_2\sq{\tau}_3\sq{\tau}_1$ of Example~\ref{exm:recurrent1}, we have for example $\Root{\{2,3,9\}}{2} = \tau_2(e_4-e_3) = e_4-e_2$ and $\Root{\{2,3,9\}}{7} = \tau_2\tau_3\tau_2\tau_1(e_3-e_2) = e_3-e_1$.
Moreover,
$$\Roots{\{2,3,9\}} = \{e_4-e_2, e_3-e_1, e_3-e_4\}.$$
\end{example}

\begin{typeA}[Incident pseudolines]
\label{rem:typeARoots}
Keep the notations of the classical situation~\ref{rem:typeASortingNetworks}.
For any~$k \in [\sizeQ]$, we have $\Root{I}{k} = e_t-e_b$, where $t$~and~$b$ are such that the $t$\ordinal{} and $b$\ordinal{} pseudolines of~$\Lambda_I$ arrive respectively on top and bottom of the $k$\ordinal{} commutator of~$\cN_{\Q}$.
For example, in the subword complex of Example~\ref{exm:recurrent1}, the root $\Root{\{2,3,9\}}{7} = e_3-e_1$ can be read in the $7$\ordinal{} commutator of \fref{fig:network} (right).

In~\cite{PilaudSantos-brickPolytope}, the root configuration~$\Roots{I}$ is studied as the incidence configuration of the contact graph~$\Lambda_I\contact$ of the pseudoline arrangement~$\Lambda_I$.
This oriented graph has one node for each pseudoline of~$\Lambda_I$ and one arc for each contact of~$\Lambda_I$, oriented from the pseudoline passing above to the pseudoline passing below.
We avoid this notion here since it is not needed for the construction of this paper, and only provides a graphical description in type~$A$ which cannot be extended to general finite types.
\end{typeA}

%%%%%%%%%%%%%%%%%%%%%%%%%%%%%%%%%%%%%%

\subsection{Roots and flips}

Throughout this paper, we show that the combinatorial and geometric properties of the root configuration~$\Roots{I}$ encode many relevant properties of the facet~$I$.
We first note that, for a given facet~$I$ of~$\subwordComplex(\Q)$, the map~${\Root{I}{\cdot}:[\sizeQ] \to \Phi}$ can be used to understand the flips in~$I$.
We refer to \fref{fig:network} for an illustration of the properties of the next lemma in type~$A$.
This lemma can also be found in~\cite[Section~3.1]{CeballosLabbeStump}.

\begin{lemma}
\label{lem:roots&flips}
Let~$I$ be any facet of the subword complex~$\subwordComplex(\Q)$.
\begin{enumerate}[(1)]
\item
\label{lem:roots&flips:enum:inversions}
The map~$\Root{I}{\cdot} : k \mapsto \Root{I}{k}$ is a bijection between the complement of~$I$ and~$\Phi^+$.

\item
\label{lem:roots&flips:enum:flip}
If $I$~and~$J$ are two adjacent facets of~$\subwordComplex(\Q)$ with~$I \ssm i = J \ssm j$, the position~$j$ is the unique position in the complement of~$I$ for which~$\Root{I}{j} \in \{\pm\Root{I}{i}\}$.
Moreover, $\Root{I}{j} = \Root{I}{i} \in \Phi^+$ if $i < j$, while $\Root{I}{j} = -\Root{I}{i} \in \Phi^-$ if $j < i$.

\item
\label{lem:roots&flips:enum:update}
In the situation of (\ref{lem:roots&flips:enum:flip}), the map~$\Root{J}{\cdot}$ is obtained from the map~$\Root{I}{\cdot}$ by:
$$\Root{J}{k} = \begin{cases} s_{\Root{I}{i}}(\Root{I}{k}) & \text{if } \min(i,j) < k \le \max(i,j), \\ \Root{I}{k} & \text{otherwise}. \end{cases}$$
\end{enumerate}
\end{lemma}

\begin{proof}
Choose a positive root~$\beta \in \Phi^+$.
Since $\wordprod{Q}{[\sizeQ] \ssm I}$ is a reduced expression for~$w_\circ$, we know that~$(\wordprod{Q}{[\sizeQ] \ssm I})^{-1}(\beta) = w_\circ(\beta) \in \Phi^-$.
Consequently, there exists ${k \in [\sizeQ] \ssm I}$ such that ${(\wordprod{Q}{[k-1] \ssm I})^{-1}(\beta) \in \Phi^+}$ while $q_k(\wordprod{Q}{[k-1] \ssm I})^{-1}(\beta) \in \Phi^-$.
Therefore, we have that $\alpha_{q_k} = (\wordprod{Q}{[k-1] \ssm I})^{-1}(\beta)$, and thus $\beta = \Root{I}{k}$.
This proves that~$\Root{I}{\cdot}$ is surjective, and even bijective from the complement of~$I$ to~$\Phi^+$, since these two sets have the same cardinality.

For~(\ref{lem:roots&flips:enum:update}), assume that~$i < j$.
Since
$$\Root{I}{k} = \wordprod{Q}{[k-1] \ssm I}(\alpha_{q_k}) = w_\circ\big(\wordprod{Q}{[k,\sizeQ] \ssm I}]\big)^{-1}(\alpha_{q_k}),$$
the roots $\Root{I}{k}$~and $\Root{J}{k}$ are equal as soon as $I$~and~$J$ coincide either on the set ${\{1,\dots,k-1\}}$ or on the set $\{k,\dots,\sizeQ\}$.
This proves~(\ref{lem:roots&flips:enum:update}) when $k \le i$ or $j < k$.
In the case where $i < k \le j$, we have
\begin{align*}
  \Root{J}{k}
    &= \wordprod{Q}{[i-1] \ssm I} \cdot q_i\cdot\wordprod{Q}{[i,k-1] \ssm I}(\alpha_{q_k}) \\
    &= s_{\wordprod{Q}{[i-1] \ssm I}(\alpha_{q_i})} \cdot \wordprod{Q}{[i-1] \ssm I} \cdot \wordprod{Q}{[i,k-1] \ssm I}(\alpha_{q_k}) \\
    &= s_{\Root{I}{i}}(\Root{I}{k}),
\end{align*}
where the second equality comes from the commutation rule $ws_v = s_{w(v)}w$.
The proof is similar~if~$j < i$.

Finally, (\ref{lem:roots&flips:enum:flip}) is a direct consequence of (\ref{lem:roots&flips:enum:inversions})~and~(\ref{lem:roots&flips:enum:update}): when we remove~$i$ from~$I$, we have to add a position~$j \notin I$ such that the root function keeps mapping the complement of~$(I \ssm i) \cup j$ to~$\Phi^+$.
\end{proof}

In view of Lemma~\ref{lem:roots&flips}, we say that $\Root{I}{i} = -\Root{J}{j} \in V$ is the \defn{direction} of the flip from facet~$I$ to facet~$J$ of~$\subwordComplex(\Q)$.

%%%%%%%%%%%%%%%%%%%%%%%%%%%%%%%%%%%%%%

\subsection{Reconstructing facets}

We now observe that a facet~$I \in \subwordComplex(\Q)$ can be reconstructed from its root configuration~$\Roots{I}$, using a sweeping procedure.

\begin{lemma}
\label{le:rootcharacterization}
The multiset~$\Roots{I}$ characterizes~$I$ among all facets of~$\subwordComplex(\Q)$.
\end{lemma}

\begin{proof}
Let $I$~and~$J$ be two distinct facets of~$\subwordComplex(\Q)$.
Assume without loss of generality that the first element~$i$ of the symmetric difference~$I \symdif J$ is in~$I \ssm J$.
Since~$\Root{I}{i}$ only depends on~$I \cap [i-1]$, the roots $\Root{I}{k}$~and~$\Root{J}{k}$ coincide for all~$1 \le k \le i$.
Let $\alpha \eqdef \wordprod{Q}{[i-1] \ssm I}(\alpha_{q_i}) = \wordprod{Q}{[i-1] \ssm J}(\alpha_{q_i}) \in \Phi^+$.
According to Lemma~\ref{lem:roots&flips}(\ref{lem:roots&flips:enum:flip}), all occurrences of~$\alpha$ in~$\Roots{J}$ appear before~$i$.
Consequently, $\alpha$~appears at least once more in~$\Roots{I}$ than in~$\Roots{J}$.
\end{proof}

\begin{remark}
\label{rem:reconstructFromRoots}
Lemma~\ref{le:rootcharacterization} says that we can reconstruct the facet~$I$ from the root configuration~$\Roots{I}$ scanning the word~$\Q$ from left to right as follows.
We start from position~$0$ in~$\Q$ and we define~$I_0$ to be the empty word and $\sfR_0$ to be the multiset~$\Roots{I}$.
At step~$k$, we construct a new subword~$I_k$ from~$I_{k-1}$ and a new multiset $\sfR_k$~from~$\sfR_{k-1}$ as follows:
\begin{enumerate}[$\bullet$]
\item if $\beta_{k-1} \eqdef \wordprod{Q}{[k-1] \ssm I_{k-1}}(\alpha_{q_k}) \in \sfR_{k-1}$, then $I_k \eqdef I_{k-1} \cup k$ and $\sfR_k \eqdef \sfR_{k-1} \ssm \beta_{k-1}$,
\item otherwise,~$I_k \eqdef I_{k-1}$ and~$\sfR_k \eqdef \sfR_{k-1}$.
\end{enumerate}
The facet~$I$ is the subword~$I_\sizeQ$ we obtained at the end of this procedure.
Observe that we can even reconstruct~$I$ knowing only its positive roots~${\Roots{I} \cap \Phi^+}$.
Indeed, in the previous sweeping procedure, the positions~$k$ for which~$\wordprod{Q}{[k-1] \ssm I_{k-1}}(\alpha_{q_k})$ is a negative root are forced to belong to the facet~$I$ according to Lemma~\ref{lem:roots&flips}(\ref{lem:roots&flips:enum:inversions}).
Similarly, we can reconstruct the facet~$I$ knowing only its negative roots~$\Roots{I} \cap \Phi^-$, sweeping the word~$\Q$ from right to left.
\end{remark}

%%%%%%%%%%%%%%%%%%%%%%%%%%%%%%%%%%%%%%

\subsection{Restriction to parabolic subgroups}
\label{sec:parabolicrestriction}

The root function~$\Root{I}{\cdot}$ is also useful to restrict subword complexes to parabolic subgroups of~$W$.

\begin{proposition}
\label{prop:restriction}
Let~$\subwordComplex(\Q)$ be a subword complex for a Coxeter system $(W,S)$ acting on~$V$, and let $V' \subseteq V$ be a subspace of $V$.
The simplicial complex given by all facets of $\subwordComplex(\Q)$ reachable from an initial facet of~$\subwordComplex(\Q)$ by flips whose directions are contained in $V'$ is isomorphic to a subword complex $\subwordComplex(\Q')$ for the restriction of $(W,S)$ to $V'$.
\end{proposition}

\begin{proof}
We consider the restriction~$(W',S')$ of the Coxeter system~$(W,S)$ to the subspace~$V'$.
We construct the word~$\Q'$ and the facet~$I'$ of $\subwordComplex(\Q')$ corresponding to the chosen initial facet $I$ of $\subwordComplex(\Q)$.
For this, let $X \eqdef \{ x_1,\ldots,x_p\}$ be the set of positions~$k \in [m]$ such that $\Root{I}{k} \in V'$.
The word~$\Q'$ has~$p$ letters corresponding to the positions in~$X$, and the facet~$I'$ contains precisely the positions~$k \in [p]$ such that the position~$x_k$ is in~$I$.
To construct the word~$\Q'$, we scan~$\Q$ from left to right as follows.
We initialize~$\Q'$ to the empty word, and for each $1 \leq k \leq p$, we add a letter $\q'_k \in S'$ to~$\Q'$ in such a way that $\Root{I'}{k} = \Root{I}{x_k}$.
To see that such a letter exists, we distinguish two cases.
Assume first that~$\Root{I}{x_k}$ is a positive root.
Let~$\mathcal{I}$ be the inversion set of~$w \eqdef \wordprod{\Q}{[x_k-1] \ssm I}$ and~$\mathcal{I}' = \mathcal{I} \cap V'$ be the inversion set of~$w' \eqdef \wordprod{\Q'}{[k-1] \ssm I'}$.
Then the set~$\mathcal{I}' \cup \{\Root{I}{x_k}\}$ is again an inversion set (as the intersection of~$V'$ with the inversion set~$\mathcal{I} \cup \{\Root{I}{x_k}\}$ of~$w q_{x_k}$) which contains the inversion set~$\mathcal{I'}$ of~$w'$ together with a unique additional root.
Therefore, the corresponding element of~$W'$ can be written as~$w'q'_k$ for some simple reflection~$q'_k \in S'$.
Assume now that~$\Root{I}{x_k}$ is a negative root.
Then~$x_k \in I$, so that we can flip it with a position~$x_{k'} < x_k$, and we can then argue on the resulting facet.
By the procedure described above, we eventually obtain the subword complex $\subwordComplex(\Q')$ and its facet $I'$ corresponding to the facet $I$.
Finally observe that sequences of flips in $\subwordComplex(\Q)$ starting at the facet $I$, and whose directions are contained in $V'$, correspond bijectively to sequences of flips in $\subwordComplex(\Q')$ starting at the facet $I'$.
In particular, let $J$ and $J'$ be two facets reached from $I$ and from $I'$, respectively, by such a sequence.
We then have that the root configuration of $J'$ is exactly the root configuration of $J$ intersected with $V'$.
This completes the proof.
\end{proof}

\begin{example}
\begin{figure}[b]
  \capstart
  \centerline{\includegraphics[width=\textwidth]{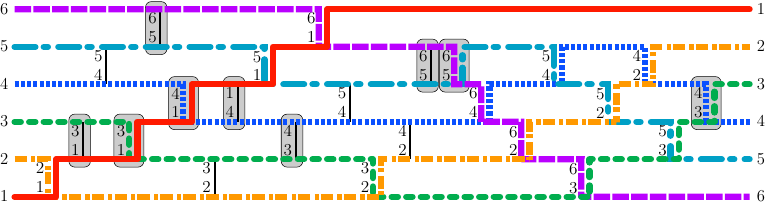}}
  \vspace{.3cm}
  \centerline{$\Downarrow$ \quad restriction to the space~$V' = \vect \langle e_3-e_1, e_4-e_3, e_6-e_5 \rangle$ \quad $\Downarrow$}
  \vspace{.3cm}
  \centerline{\includegraphics[width=\textwidth]{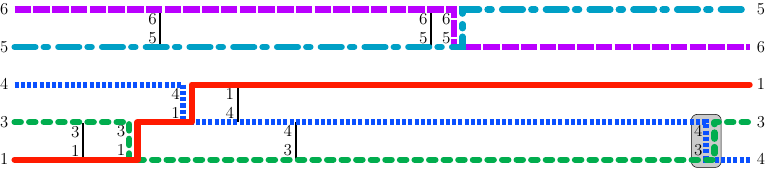}}
  \caption{Restricting subword complexes.}
  \label{fig:restriction}
\end{figure}

To illustrate different possible situations happening in this restriction, we consider the subword complex~$\subwordComplex(\Q)$ for the word
\[
\Q \eqdef \tau_1 \tau_2 \tau_4 \tau_2 \tau_5 \tau_3 \tau_1 \tau_3 \tau_4 \tau_2 \tau_5 \tau_3 \tau_1 \tau_2 \tau_4 \tau_4 \tau_3 \tau_2 \tau_4 \tau_1 \tau_3 \tau_4 \tau_2 \tau_3
\]
on the Coxeter group~$A_5 = \fS_6$ generated by~$S = \{\tau_1, \dots, \tau_5\}$.
\fref{fig:restriction}\,(top) shows the sorting network representing the subword complex~$\subwordComplex(\Q)$ and the pseudoline arrangement representing the facet~$I \eqdef \{2,3,5,7,8,10,12,14,15\}$ of~$\subwordComplex(\Q)$.
Let~$V'$ be the subspace of~$V$ spanned by the roots $e_3-e_1$, $e_4-e_3$ and~${e_6-e_5}$.
Let~$X = \{x_1,\dots,x_8\} = \{2,4,5,6,8,10,15,16\}$ denote the set of positions~$k \in [24]$ for which~$\Root{I}{k} \in V'$.
These positions are circled in \fref{fig:restriction}\,(top).

We can now directly read off the subword complex~$\subwordComplex(\Q')$ corresponding to the restriction of~$\subwordComplex(\Q)$ to all facets reachable from~$I$ by flips with directions in~$V'$.
Namely, the restriction of~$(W,S)$ to~$V'$ is the Coxeter system~$(W',S')$ where~$W'$ is generated by $S' = \{\tau'_1,\tau'_2,\tau'_3\} = \{ (1\;3), (3\;4), (5\;6)\}$, and thus of type~$A_2 \times A_1$.
Moreover, $\Q' = \tau'_1 \tau'_1 \tau'_3 \tau'_2 \tau'_2 \tau'_1 \tau'_3 \tau'_3 \tau'_1$ corresponds to the roots at positions in~$X$.
Finally, the facet~$I$ of~$\subwordComplex(\Q)$ corresponds to the facet~$I' = \{1,3,5,6,7\}$ of~$\subwordComplex(\Q')$.
\fref{fig:restriction}\,(bottom) shows the sorting network representing the restricted subword complex~$\subwordComplex(\Q')$ and the pseudoline arrangement representing the facet~$I'$ of~$\subwordComplex(\Q')$.
\end{example}

%%%%%%%%%%%%%%%%%%%%%%%%%%%%%%%%%%%%%%

\subsection{Root independent subword complexes}

The geometry of the root configuration~$\Roots{I}$ encodes many combinatorial properties of~$I$.
In fact, the facet~$I$ is relevant for the brick polytope only when its root configuration is pointed.
This is in particular the case when~$\Roots{I}$ forms a basis of~$V$.
From Lemma~\ref{lem:roots&flips} and the connectedness of the flip graph, we derive that this property depends only on~$\Q$, not on the particular facet~$I$ of~$\subwordComplex(\Q)$.

\begin{lemma}
Either all the root configurations~$\Roots{I}$ for facets~$I$ of~$\subwordComplex(\Q)$ are simultaneously linear bases of the vector space~$V$, or none of them is.
\end{lemma}

In this paper, we only consider spherical subword complexes $\subwordComplex(\Q)$ for which the root configuration~$\Roots{I}$ of a facet~$I$ (or equivalently, of all facets) is linearly independent.
We say that $\subwordComplex(\Q)$ is \defn{root independent}.
Proposition~\ref{prop:restriction} ensures that if the root configuration is linearly independent but does not span~$V$, then we can always find an isomorphic subword complex for a parabolic subgroup of~$W$ for which the root configuration is indeed a basis.
We will thus assume that the root configurations of root independent subword complexes are indeed bases of~$V$.

\begin{typeA}[Contact tree]
\label{rem:typeAContactTree}
In type~$A$, the root configuration~$\Roots{I}$ forms a basis of~$\one^\perp$ if and only if the contact graph of~$\Lambda_I$ is a tree.
It is independent if and only if the contact graph of~$\Lambda_I$ is a forest.
For example, the subword complex $\subwordComplex(\Qexm)$ of Example~\ref{exm:recurrent1} is root independent.
\end{typeA}

\begin{example}[Duplicated word]
\label{exm:duplicatedRoots}
Let~$\w_\circ \eqdef \w_1 \cdots \w_N$ be a reduced expression of the longest element~$w_\circ$ of~$W$ (thus~$N = \length(w_\circ) = |\Phi^+|$).
For~$k \in [N]$, we define a root~$\alpha_k \eqdef w_1 \cdots w_{k-1}(\alpha_{w_k})$.
Let~$P$ be a set of $n$~positions in~$\w_\circ$ such that the corresponding roots~$\set{\alpha_p}{p \in P}$ form a linear basis of~$V$.
Define a new word $\Qdup \in S^*$ obtained by duplicating the letters of~$\w_\circ$ at positions in~$P$.
For~$k \in [N]$, let~${k^* \eqdef k + |P \cap [k-1]|}$ be the new position in~$\Qdup$ of the $k$\ordinal{} letter of~$\w_\circ$.
Thus, we have $[N+n] = \set{k^*}{k \in [N]} \sqcup \set{p^* + 1}{p \in P}$.
The facets of~$\subwordComplex(\Qdup)$ are precisely the sets~$I_\varepsilon \eqdef \set{p^* + \varepsilon_p}{p \in P}$ where~${\varepsilon \eqdef (\varepsilon_1,\dots,\varepsilon_n) \in \{0,1\}^P}$.
Thus, the subword complex~$\subwordComplex(\Qdup)$ is the boundary complex of the $n$-dimensional cross polytope.
Moreover, the roots of a facet~$I_\varepsilon$ of~$\subwordComplex(\Qdup)$ are given by~$\Root{I_\varepsilon}{k^*} = \alpha_k$ for~$k \in [N]$ and~$\Root{I_\varepsilon}{p^*+1} = (-1)^{\varepsilon_p}\alpha_p$ for~${p \in P}$.
Thus, the root configuration of~$I_\varepsilon$ is given by~$\Roots{I_\varepsilon} = \set{(-1)^{\varepsilon_p}\alpha_p}{p \in P}$, and the subword complex $\subwordComplex(\Qdup)$ is root independent.
To illustrate the results of this paper, we discuss further the properties of duplicated words in Examples \ref{exm:duplicatedWeights}, \ref{exm:duplicatedProjectionMap}, \ref{exm:duplicatedNormalFan}, \ref{exm:duplicatedLattice}, and~\ref{exm:duplicatedMinkowski}.
\end{example}

%%%%%%%%%%%%%%%%%%%%%%%%%%%%%%%%%%%%%%

\subsection{Root configuration and linear functionals}

\enlargethispage{-.1cm}
We now consider the geometry of the root configuration~$\Roots{I}$ of a facet~$I$ of~$\subwordComplex(\Q)$ with respect to a linear functional~$f:V\to\R$.
We say that the flip of an element~$i$ in~$I$ is \defn{$f$-preserving} if~$f(\Root{I}{i}) = 0$.
We denote by
$$\subwordComplex_f(\Q) \eqdef \bigset{I \text{ facet of } \subwordComplex(\Q)}{\forall i \in I, f(\Root{I}{i}) \ge 0}$$
the set of facets of~$\subwordComplex(\Q)$ whose root configuration is contained in the closed positive halfspace defined by~$f$.

\begin{proposition}
\label{prop:preservingFlips}
If~$\subwordComplex(\Q)$ is root independent, the set~$\subwordComplex_f(\Q)$ forms a connected component of the graph of $f$-preserving flips on~$\subwordComplex(\Q)$.
\end{proposition}

\begin{proof}
Assume that $I$~and~$J$ are two adjacent facets of~$\subwordComplex(\Q)$ related by an \mbox{$f$-pre}\-serving flip.
Write~$I \ssm i = J \ssm j$ with ${f(\Root{I}{i}) = 0}$.
By Lemma~\ref{lem:roots&flips}(\ref{lem:roots&flips:enum:update}), we have $\Root{J}{k} - \Root{I}{k} \in \R\cdot\Root{I}{i}$ and thus $f(\Root{J}{k}) = f(\Root{I}{k})$ for all~$k \in [\sizeQ]$.
Thus,~$\subwordComplex_f(\Q)$ is closed under $f$-preserving flips.

It remains to prove that~$\subwordComplex_f(\Q)$ is connected by $f$-preserving flips.
Fix a facet~$I$ of~$\subwordComplex_f(\Q)$, and define the set of positions~$P(I,f) \eqdef \set{i \in I}{f(\Root{I}{i}) > 0}$.
According to the previous paragraph, this set is constant on each connected component of~$\subwordComplex_f(\Q)$.
We will prove below that the set~$P_f \eqdef P(I,f)$ is in fact independent of the facet~$I \in \subwordComplex_f(\Q)$.
Therefore, the graph of $f$-preserving flips on~$\subwordComplex_f(\Q)$ can be seen as the graph of flips on the subword complex~$\subwordComplex(\Q_{[\sizeQ] \ssm P_f})$ of the word obtained from~$\Q$ by erasing the letters at the positions given by~$P_f$.
The latter is connected since $\subwordComplex(\Q_{[\sizeQ] \ssm P_f})$ is a ball or a sphere.

To prove that~$P(I,f) \eqdef \set{i \in I}{f(\Root{I}{i}) > 0}$ is independent of the particular facet~$I \in \subwordComplex_f(\Q)$, we use an inductive argument on the minimal cardinality~$\mu$ of a set~$P(I,f)$ for the facets~$I \in \subwordComplex_f(\Q)$.
Clearly, if~$\mu = 0$, then all flips are $f$-preserving and $\subwordComplex_f(\Q) = \subwordComplex(\Q)$ is connected.

Assume now that there is a facet~$I \in \subwordComplex_f(\Q)$ such that $P(I,f) = \{p\}$.
Assume moreover that~$\Root{I}{p} \in \Phi^-$.
Performing $f$-preserving flips, we can obtain a facet~$\bar I$ such that~$P(\bar I, f) = P(I,f) = \{p\}$, and~$\Root{\bar I}{i} \in \Phi^-$ for each~$i \in \bar I$ with~${f(\Root{\bar I}{i}) = 0}$.
For this facet~$\bar I$, observe that
$$j \in \bar I \quad \iff \quad \wordprod{Q}{[j-1] \ssm \bar I}(\alpha_{q_j}) \in f^{-1}(\R_{\ge 0}) \cap \Phi^-.$$
In other words, the facet~$\bar I$, and thus the set~$P(I,f)$, can be reconstructed scanning the word~$\Q$ from left to right.
If~$\Root{I}{p} \in \Phi^+$, we argue similarly scanning the word~$\Q$ from right to left.

Assume now that~$P(I,f) = P' \cup \{p\}$, with~$P' \ne \varnothing$.
Consider an auxiliary linear function~$f'$ vanishing on $\Roots{I} \cap f^{-1}(0)$ and on~$\Root{I}{p}$, but still positive on the other roots of~$P(I,f)$.
Note that this function exists because $\Q$ is root independent which means that $\Roots{I}$ is linearly independent.
The facet~$I$ belongs to~$\subwordComplex_{f'}(\Q)$, so that we can reconstruct the positions~$P' = P(I,f')$ by induction hypothesis.
Finally, since~$I \ssm P'$ is also a facet of~$\subwordComplex(\Q_{[\sizeQ] \ssm P'})$, we can recover the last position~$p$ of~$P(I,f)$ with a similar argument as before.
\end{proof}

\begin{corollary}
\label{coro:preservingFlips}
If~$\subwordComplex(\Q)$ is root independent, there exists a face~$P_f$ of~$\subwordComplex(\Q)$ such that~$\subwordComplex_f(\Q)$ is the set of all facets containing~$P_f$.
\end{corollary}

\begin{proof}
Define again the set of positions~$P_f \eqdef \set{k \in I}{f(\Root{I}{k}) > 0}$ for a given facet~$I$ of~$\subwordComplex_f(\Q)$.
As mentioned in the previous proof, the set~$P_f$ is independent of the choice of~$I \in \subwordComplex_f(\Q)$, since the set~$\subwordComplex_f(\Q)$ is connected by~$f$-preserving flips which preserve the value~$f(\Root{I}{k})$ for all~$k \in [\sizeQ]$.
Furthermore, any position~$i \in I \ssm P_f$ is an $f$-preserving flip.
Consequently, $\subwordComplex_f(\Q)$ is precisely the set of facets of~$\subwordComplex(\Q)$ which contains~$P_f$.
\end{proof}

%%%%%%%%%%%%%%%%%%%%%%%%%%%%%%%%%%%%%%
%%%%%%%%%%%%%%%%%%%%%%%%%%%%%%%%%%%%%%

\section{Generalized brick polytopes}
\label{sec:brickPolytope}

%%%%%%%%%%%%%%%%%%%%%%%%%%%%%%%%%%%%%%

\subsection{Definition and basic properties}

We begin this section with generalizing the definition of brick polytopes of~\cite{PilaudSantos-brickPolytope} to any spherical subword complex~$\subwordComplex(\Q)$, for $\Q \eqdef \q_1 \q_2 \cdots \q_\sizeQ \in S^*$ with~$\DemazureProduct(\Q) = w_\circ$.

\begin{definition}
\label{def:weights}
To a facet~$I$ of~$\subwordComplex(\Q)$ and a position~$k \in [\sizeQ]$, associate~the~weight
$$\Weight{I}{k} \eqdef \wordprod{Q}{[k-1] \ssm I}(\omega_{q_k}),$$
where, as before, $\wordprod{\Q}{X}$ denotes the product of the reflections~$q_x \in \Q$, for~$x \in X$, in the order given by~$\Q$.
The \defn{brick vector} of the facet~$I$ is defined as
$$\brickVector(I) \eqdef \sum_{k \in [\sizeQ]} \Weight{I}{k}.$$
The \defn{brick polytope}~$\brickPolytope(\Q)$ of the word~$\Q$ is the convex hull of all the brick vectors,
$$\brickPolytope(\Q) \eqdef \conv \bigset{\brickVector(I)}{I \text{ facet of } \subwordComplex(\Q)}.$$
\end{definition}

\begin{example}
\label{exm:recurrent3}
In the subword complex~$\subwordComplex(\Qexm)$ of Example~\ref{exm:recurrent1}, we have for ex\-ample ${\Weight{\{2,3,9\}}{2} = \tau_2(e_4) = e_4}$ and ${\Weight{\{2,3,9\}}{7} = \tau_2\tau_3\tau_2\tau_1(e_3+e_4) = e_2+e_3}$.
The brick vector of~$\{2,3,9\}$ is~$\brickVector(\{2,3,9\}) = e_1 + 6e_2 + 5e_3 + 6e_4 = (1,6,5,6)$.
The brick polytope~$\brickPolytope(\Qexm)$ is a pentagonal prism, represented in \fref{fig:brickPolytope}.

\begin{figure}[h]
  \capstart
  \centerline{\includegraphics[scale=.75]{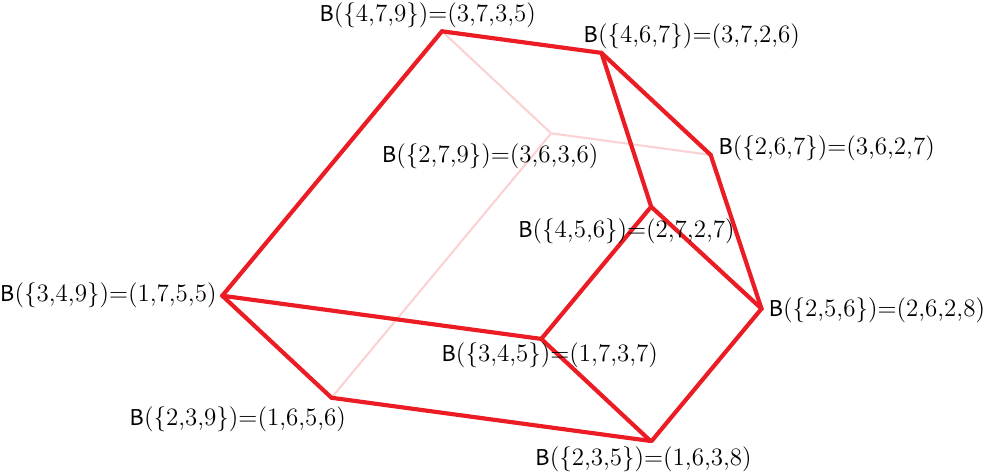}}
  \caption{The brick polytope~$\brickPolytope(\Qexm)$ of the word~$\Qexm$ of Example~\ref{exm:recurrent1}.}
  \label{fig:brickPolytope}
\end{figure}
\end{example}

\begin{typeA}[Counting bricks]
\label{rem:typeACountingBricks}
In type~$A$, we can read this definition on the sorting network interpretation of the classical situation~\ref{rem:typeASortingNetworks}.
For a facet~$I$ of~$\subwordComplex(\Q)$ and a position~${k \in [\sizeQ]}$, the weight~$\Weight{I}{k}$ is the characteristic vector of the pseudolines of~$\Lambda_I$ which pass above the $k$\ordinal{} brick of~$\cN_{\Q}$.
For any~$p \in [n+1]$, the $p$\ordinal{} coordinate of the brick vector~$\brickVector(I)$ is the number of bricks of~$\cN_{\Q}$ below the $p$\ordinal{} pseudoline of~$\Lambda_I$.
See \fref{fig:network} for an illustration.
This was the original definition used in type~$A$ in \cite{PilaudSantos-brickPolytope}, which explains the name \defn{brick polytope}.
\end{typeA}

\begin{example}[Duplicated word]
\label{exm:duplicatedWeights}
Consider the situation of Example~\ref{exm:duplicatedRoots}: the word~$\Qdup$ is obtained from the reduced expression~$\w_\circ \eqdef \w_1 \cdots \w_N$ of~$w_\circ$ by duplicating each letter located at a position in the chosen set~$P$.
For~$k \in [N]$, we define a weight~$\omega_k \eqdef w_1 \cdots w_{k-1}(\omega_{w_k})$.
Then the weights of a facet~$I_\varepsilon$ of~$\subwordComplex(\Qdup)$ are given by~$\Weight{I_\varepsilon}{k^*} = \omega_k$ for~$k \in [N]$ and~${\Weight{I_\varepsilon}{p^*+1} = \omega_p + \varepsilon_p \alpha_p}$ for~${p \in P}$.
Consequently, the brick vector of the facet~$I_\varepsilon$ is given by~$\brickVector(I_\varepsilon) = \Theta + \sum_{p \in P} \varepsilon_p \alpha_p$, where~$\Theta \eqdef \sum_{k \in [N]} \omega_k + \sum_{p \in P} \omega_p$.
It follows that the brick polytope~$\brickPolytope(\Qdup)$ is an $n$-dimensional parallelepiped, whose edges are directed by the basis~$\set{\alpha_p}{p \in P}$.
\end{example}

Although we define brick polytopes for any word~$\Q$, we only present in this paper their geometric and combinatorial properties when $\subwordComplex(\Q)$ is root independent.
We will see that this condition is sufficient for the brick polytope to realize its subword complex.
Reciprocally, the brick polytope can realize its subword complex only when the root configurations of the facets are linearly independent, and we can then assume that the subword complex $\subwordComplex(\Q)$ is root independent by Proposition~\ref{prop:restriction}.
Possible extensions of the results of this paper to all spherical subword complexes, root independent or not, are discussed in Section~\ref{sec:rootdependent}.

\medskip
The root function~$\Root{\cdot}{\cdot}$ and the weight function~$\Weight{\cdot}{\cdot}$ have similar definitions and are indeed closely related.
The following lemma underlines some of their similarities.

\begin{lemma}
\label{lem:weights&flips}
Let~$I$ be a facet of the subword complex~$\subwordComplex(\Q)$.
\begin{enumerate}[(1)]
\item
\label{lem:weights&flips:enum:roots&weights}
If~$\q_k = \q_{k+1}$, we have $\Weight{I}{k+1} = \begin{cases} \Weight{I}{k} & \text{if } k \in I, \\ \Weight{I}{k} - \Root{I}{k} & \text{if } k \notin I. \end{cases}$

\item
\label{lem:weights&flips:enum:update}
If $I$~and~$J$ are two adjacent facets of~$\subwordComplex(\Q)$, with~$I \ssm i = J \ssm j$, then~$\Weight{J}{\cdot}$ is obtained from~$\Weight{I}{\cdot}$~by:
$$\Weight{J}{k} = \begin{cases} s_{\Root{I}{i}}(\Weight{I}{k}) & \text{if } \min(i,j) < k \le \max(i,j), \\ \Weight{I}{k} & \text{otherwise}. \end{cases}$$

\item
\label{lem:weights&flips:enum:scalarProducts}
For~$j \notin I$, we have $\dotprod{\Root{I}{j}\,}{\,\Weight{I}{k}}$ is non-negative if~$j \ge k$,  and non-positive if~$j < k$.
\end{enumerate}
\end{lemma}

\begin{proof}
(\ref{lem:weights&flips:enum:roots&weights}) is an immediate consequence of the definitions of~$\Root{I}{k}$ and~$\Weight{I}{k}$.
The proof of~(\ref{lem:weights&flips:enum:update}) is similar to the proof of Lemma~\ref{lem:roots&flips}(\ref{lem:weights&flips:enum:update}).
We now prove~(\ref{lem:weights&flips:enum:scalarProducts}).
If~$j \ge k$, then $\wordprod{Q}{[k,j-1] \ssm I}(\alpha_{q_j}) \in \Phi^+$ since~$j \notin I$.
Thus,
$${\dotprod{\Root{I}{j}\,}{\,\Weight{I}{k}} = \dotprod{\wordprod{Q}{[k,j-1] \ssm I}(\alpha_{q_j})\,}{\,\omega_{q_k}} \ge 0}.$$
Similarly, if~$j > k$, then~$\big(\wordprod{Q}{[j,k-1] \ssm I}\big)^{-1}(\alpha_{q_j}) \in \Phi^-$ since~$j \notin I$.
Consequently,
\[
\dotprod{\Root{I}{j}\,}{\,\Weight{I}{k}} = \bigdotprod{\big(\wordprod{Q}{[j,k-1] \ssm I}\big)^{-1}(\alpha_{q_j})\,}{\,\omega_{q_k}} \le 0. \qedhere
\]
\end{proof}

%%%%%%%%%%%%%%%%%%%%%%%%%%%%%%%%%%%%%%

\subsection{Brick polytopes of root independent subword complexes}

The end of this section contains the main results of this paper, proving that brick polytopes provide polytopal realizations of the root independent subword complexes.

\begin{lemma}
\label{lem:bricks&flips}
If $I$~and~$J$ are two adjacent facets of~$\subwordComplex(\Q)$, with~$I \ssm i = J \ssm j$, then the difference of the brick vectors~$\brickVector(I) - \brickVector(J)$ is a positive multiple of~$\Root{I}{i}$.
\end{lemma}

\begin{proof}
Since~$I$ and~$J$ play symmetric roles, we can assume that~$i < j$.
Lemma~\ref{lem:weights&flips} implies that
$$\brickVector(I)-\brickVector(J) = \sum_{k \in [\sizeQ]} \Weight{I}{k}-\Weight{J}{k} = \bigg(\sum_{i < k \le j} \dotprod{\Root{I}{j}^\vee\,}{\,\Weight{I}{k}}\bigg)\cdot\Root{I}{i},$$
where the last sum is positive since all its summands are non-negative and the last one equals to~$1$.
\end{proof}

Let~$f: V \to \R$ be a linear functional.
We denote by~$\brickPolytope_f(\Q)$ the face of~$\brickPolytope(\Q)$ which maximizes~$f$.
Remember that~$\subwordComplex_f(\Q)$ denotes the set of facets of~$\subwordComplex(\Q)$ whose root configuration is included in the closed positive halfspace defined by~$f$.

\begin{lemma}
\label{lem:faces}
The faces of the brick polytope~$\brickPolytope(\Q)$ are faces of the subword complex~$\subwordComplex(\Q)$.
For any facet~$I$ of~$\subwordComplex(\Q)$, we have $\brickVector(I) \in \brickPolytope_f(\Q) \iff I \in \subwordComplex_f(\Q)$.
\end{lemma}

\begin{proof}
Consider a facet~$I$ of~$\subwordComplex(\Q)$ such that~$\brickVector(I) \in \brickPolytope_f(\Q)$, and let~$i \in I$.
Let~$J$ be the facet obtained from~$I$ by flipping~$i$, and~$\lambda > 0$ such that ${\Root{I}{i} = \lambda(\brickVector(I)-\brickVector(J))}$ (by Lemma~\ref{lem:bricks&flips}).
Then~${f(\Root{I}{i}) = \lambda(f(\brickVector(I))-f(\brickVector(J))) \ge 0}$ since~$\brickVector(I)$ maximizes~$f$.
Thus, $I \in \subwordComplex_f(\Q)$.
We obtained that the set of facets~$I$ of~$\subwordComplex(\Q)$ such that~${\brickVector(I) \in \brickPolytope_f(\Q)}$ is contained in~$\subwordComplex_f(\Q)$.
Since the former is closed while the latter is connected under $f$-preserving flips (Proposition~\ref{prop:preservingFlips}), they coincide.
\end{proof}

\begin{proposition}
\label{prop:cones}
For any facet~$I$ of~$\subwordComplex(\Q)$, the cone of the brick polytope~$\brickPolytope(\Q)$ at the brick vector~$\brickVector(I)$ coincides with the cone generated by the negative of the root configuration~$\Roots{I}$ of~$I$.
We set
$$\rootCone(I) \eqdef \cone \bigset{\brickVector(J)-\brickVector(I)}{J \text{ facet of } \subwordComplex(\Q)} = \cone \bigset{-\Root{I}{i}}{i \in I}.$$
\end{proposition}

\begin{proof}
Let~$\rootCone_\sfB(I)$ be the cone of~$\brickPolytope(\Q)$ at~$\brickVector(I)$ and let~$\rootCone_\sfR(I)$ be the cone generated by~$-\Roots{I}$.
The inclusion~${\rootCone_\sfR(I) \subseteq \rootCone_\sfB(I)}$ is an immediate consequence of Lemma~\ref{lem:bricks&flips}.
For the other direction, we need to prove that any face of~$\rootCone_\sfR(I)$ is also a face of $\rootCone_\sfB(I)$.
That is to say that the brick vector~$\brickVector(I)$ maximizes any linear functional~$f$ among the brick vectors of all facets of~$\subwordComplex(\Q)$ as soon as it maximizes~$f$ among the brick vectors of the facets adjacent to~$I$.
This is ensured by Lemma~\ref{lem:faces}.
\end{proof}

\begin{theorem}
\label{theo:realization}
If $\subwordComplex(\Q)$ is root independent, it is realized by the polar of the brick polytope~$\brickPolytope(\Q)$.
\end{theorem}

\begin{proof}
Proposition~\ref{prop:cones} ensures that the brick vector of each facet of~$\subwordComplex(\Q)$ is a vertex of~$\brickPolytope(\Q)$ and that this polytope is full-dimensional and simple.
Since the faces of the brick polytope are faces of the subword complex by Lemma~\ref{lem:faces}, it immediately implies that the boundary complex of the brick polytope~$\brickPolytope(\Q)$ is isomorphic to the subword complex~$\subwordComplex(\Q)$.
\end{proof}

\begin{remark}
If $(W,S)$ is crystallographic (\ie stabilizes a lattice~$\Lambda$), then the brick polytope~$\brickPolytope(\Q)$ is furthermore a lattice polytope (\ie its vertices are elements of~$\Lambda$).
Indeed, the brick vector of any facet is a sum of weights of~$W$.
\end{remark}

%%%%%%%%%%%%%%%%%%%%%%%%%%%%%%%%%%%%%%
%%%%%%%%%%%%%%%%%%%%%%%%%%%%%%%%%%%%%%

\section{Further combinatorial and geometric properties}
\label{sec:properties}

In this section, we consider further relevant properties of the brick polytope~$\brickPolytope(\Q)$.
We first define a surjective map~$\projectionMap$ from the Coxeter group~$W$ to the facets of the subword complex~$\subwordComplex(\Q)$.
Using this map, we connect the normal fan of~$\brickPolytope(\Q)$ to the Coxeter fan, and the graph of~$\brickPolytope(\Q)$ to the Hasse diagram of the weak order.
We then study some properties of the fibers of~$\projectionMap$ with respect to the weak order.
Finally, we decompose~$\brickPolytope(\Q)$ into a Minkowski sum of $W$-matroid polytopes.
Again, although most of the results presented in this section seem to have natural extensions to all spherical subword complexes (see Section~\ref{sec:rootdependent}), our current presentation only considers the root independent situation.

%%%%%%%%%%%%%%%%%%%%%%%%%%%%%%%%%%%%%%

\subsection{Surjective map}
\label{subsec:surjectiveMap}

A set~$U \subseteq \Phi$ is \defn{separable} if there is a hyperplane which separates~$U$ from its complement~${\Phi \ssm U}$.
For finite Coxeter systems, separable sets are precisely the sets~$w(\Phi^+)$, for~${w \in W}$.
According to Proposition~\ref{prop:preservingFlips}, there is a unique facet~$I$ of~$\subwordComplex(\Q)$ whose root configuration~$\Roots{I}$ is contained in a given separable set.
This defines a map~$\projectionMap$ from~$W$ to the facets of~$\subwordComplex(\Q)$ which associates to~$w \in W$ the unique facet~$\projectionMap(w)$ of~$\subwordComplex(\Q)$ such that~${\Roots{\projectionMap(w)} \subseteq w(\Phi^+)}$.
Furthermore, since we assumed that any root configuration forms a basis (and thus is pointed), this map~$\projectionMap$ is surjective.
However, note that this map is not injective.
In terms of brick polytopes, $\projectionMap(w)$ is the unique facet of~$\subwordComplex(\Q)$ whose brick vector maximizes the linear functional~$x \mapsto \dotprod{w(q)}{x}$ over all vertices of the brick polytope~$\brickPolytope(\Q)$, where~$q$ is any point in the interior of the fundamental chamber~$\fundamentalChamber$ of~$W$.

\begin{example}
\label{exm:recurrent4}

\begin{figure}[p]
  \capstart
  \centerline{\includegraphics[scale=.54]{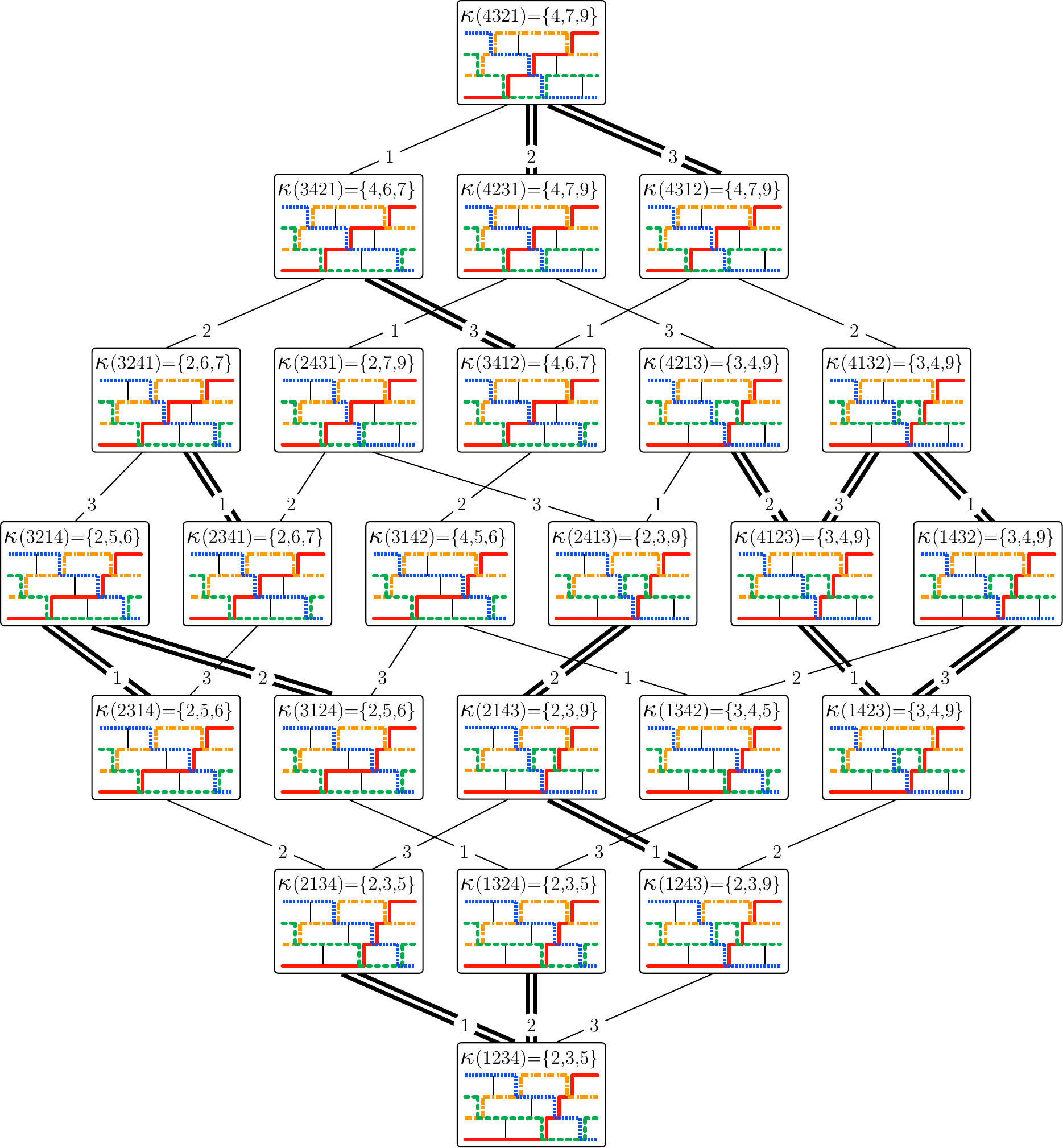}}
  \caption{The map~$\projectionMap$ from~$\fS_4$ to the facets of the subword complex~$\subwordComplex(\Qexm)$ of Example~\ref{exm:recurrent1}.}
  \label{fig:lattice}
\end{figure}

The map $\projectionMap$ for the word~$\Qexm$ of Example~\ref{exm:recurrent1} is presented in \fref{fig:lattice}.
For each permutation~$w \in \fS_4$, we have written the facet~$\projectionMap(w)$ of~$\subwordComplex(\Qexm)$, and we have represented the pseudoline arrangement~$\Lambda_{\projectionMap(w)}$ supported by~$\cN_{\Qexm}$.
The underlying diagram is the Hasse diagram of the weak order on~$\fS_4$: two permutations are related by an edge if they differ by the transposition of two adjacent letters.
To visualize the fibers of~$\projectionMap$, we draw a light edge when the images of the permutations under~$\projectionMap$ are distinct, and a double strong edge if the images are the same.
\end{example}

\begin{example}[Duplicated word]
\label{exm:duplicatedProjectionMap}
Consider the word~$\Qdup$ defined in Example~\ref{exm:duplicatedRoots}.
Two elements~$w,w'$ of~$W$ have the same image under the map~$\projectionMap$ if and only if the symmetric difference of~$w(\Phi^+)$ and~$w'(\Phi^+)$ is disjoint from the set~$\set{\pm\alpha_p}{p \in P}$, or equivalently the symmetric difference of the inversion sets of~$w$ and~$w'$ is disjoint from the set~$\set{\alpha_p}{p \in P}$.
\end{example}

\begin{lemma}
\label{lem:flipProjectiveMap}
Consider an element~$w \in W$ and a simple root~$\alpha \in \Delta$.
Then~$\projectionMap(ws_\alpha)$ is obtained from~$\projectionMap(w)$ as follows.
\begin{enumerate}[(i)]
\item If~$w(\alpha) \notin \Roots{\projectionMap(w)}$, then~$\projectionMap(ws_\alpha)$ equals~$\projectionMap(w)$.
\item Otherwise, $\projectionMap(ws_\alpha)$~is obtained from~$\projectionMap(w)$ by flipping the unique~$i \in \projectionMap(w)$ such that~${w(\alpha) = \Root{\projectionMap(w)}{i}}$.
\end{enumerate}
\end{lemma}

\begin{proof}
If~$w(\alpha) \notin \Roots{\projectionMap(w)}$, then $\Roots{\projectionMap(w)} \subseteq ws_\alpha(\Phi^+) = w(\Phi^+) \symdif \{\pm w(\alpha)\}$ (where~$\symdif$ denotes the symmetric difference).
Thus, $\projectionMap(ws_\alpha) = \projectionMap(w)$ by uniqueness in the definition of the map~$\projectionMap$.

Assume now that~$w(\alpha) = \Root{\projectionMap(w)}{i}$ for some~$i \in \projectionMap(w)$, and let~$J$ denote the facet of~$\subwordComplex(\Q)$ obtained from~$\projectionMap(w)$ by flipping~$i$, and $j$ be the position in~$J$ but not in~$\projectionMap(w)$.
According to Lemma~\ref{lem:roots&flips}(\ref{lem:weights&flips:enum:update}),
$$\Root{J}{k} = \begin{cases} s_{w(\alpha)}(\Root{\projectionMap(w)}{k}) = ws_\alpha w^{-1}(\Root{\projectionMap(w)}{k}) & \text{if } \min(i,j) < k \le \max(i,j), \\ \Root{\projectionMap(w)}{k} & \text{otherwise.} \end{cases}$$
If~$\min(i,j) < k \le \max(i,j)$, we have~$\Root{J}{k} \in ws_\alpha(\Phi^+)$ since~${w^{-1}(\Root{\projectionMap(w)}{k}) \in \Phi^+}$.
Otherwise, we have ${w^{-1}(\Root{\projectionMap(w)}{k}) \in \Phi^+ \ssm \alpha}$ so that $s_\alpha w^{-1}(\Root{\projectionMap(w)}{k}) \in \Phi^+$, and thus~${\Root{J}{k} = \Root{\projectionMap(w)}{k} \in ws_\alpha(\Phi^+)}$.
We obtain that $\Roots{J} \subseteq ws_\alpha(\Phi^+)$, and consequently~$J = \projectionMap(ws_\alpha)$ by uniqueness in the definition of the map~$\projectionMap$.
\end{proof}

%%%%%%%%%%%%%%%%%%%%%%%%%%%%%%%%%%%%%%

\subsection{Normal fan description}
\label{subsec:normalFan}

Remember that the \defn{normal cone} of a face~$F$ of a polytope~${\Pi \subset V}$ is the cone of all vectors~$v \in V$ such that the linear function~${x \mapsto \dotprod{x}{v}}$ on~$\Pi$ is maximized by all points in~$F$.
The \defn{normal fan} of~$\Pi$ is the complete polyhedral fan formed by the normal cones of all faces of~$\Pi$ (see~\cite[Lecture~7]{Ziegler}).
For example, the normal fan of the $W$-permutahedron is the Coxeter fan of~$W$.
The following proposition relates the normal fan of the brick polytope with the Coxeter fan.

\begin{proposition}
\label{prop:normalCone}
Let~$I$ be a facet of~$\subwordComplex(\Q)$.
The normal cone~$\normalCone(I)$ of~$\brickVector(I)$ in~$\brickPolytope(\Q)$ is the union of the chambers~$w(\fundamentalChamber)$ of the Coxeter fan of~$W$ given by the elements~${w \in W}$ with~$\projectionMap(w)=I$.
\end{proposition}

\begin{proof}
For a given cone~$X \subset V$, we denote by~$X\polar \eqdef \set{v \in V}{\forall x \in X, \; \dotprod{x}{v} \ge 0}$ its polar cone.
We have
$$
\normalCone(I) = \rootCone(I)\polar = \cone \big(-\Roots{I}\big)\polar = \bigg(\bigcap_{w \in \projectionMap^{-1}(I)} \!\!\!\!\!w(\fundamentalChamber\polar)\bigg)\polar = \!\!\!\bigcup_{w \in \projectionMap^{-1}(I)} \!\!\!\!\!w(\fundamentalChamber\polar)\polar = \!\!\! \bigcup_{w \in \projectionMap^{-1}(I)} \!\!\!\!w(\fundamentalChamber),
$$
because~$\fundamentalChamber\polar$ is the cone generated by~$\Phi^+$.
\end{proof}

\begin{corollary}
\label{coro:refinedNormalFans}
The Coxeter fan refines the normal fan of the brick polytope.
\end{corollary}

\begin{example}
\label{exm:recurrent5}
\fref{fig:normalFan} illustrates Corollary~\ref{coro:refinedNormalFans} for the subword complex~$\subwordComplex(\Qexm)$ of Example~\ref{exm:recurrent1}.
The normal fan of the brick polytope~$\brickPolytope(\Qexm)$ (middle) is obtained by coarsening the chambers of the Coxeter fan (left) according to the fibers of~$\projectionMap$ (represented on the permutahedron on the right).

\begin{figure}
  \capstart
  \centerline{\includegraphics[width=\textwidth]{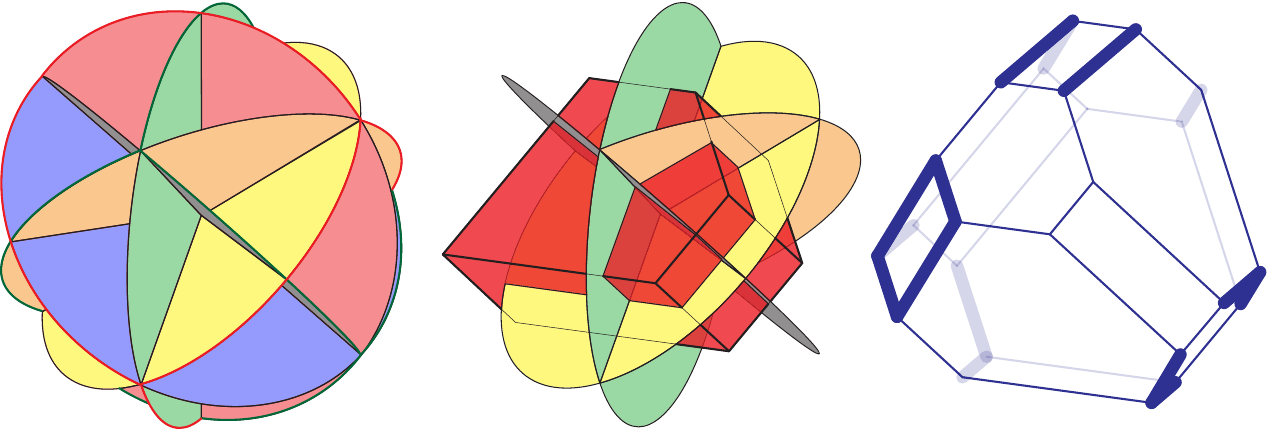}}
  \caption{The Coxeter fan (left) refines the normal fan of the brick polytope~$\brickPolytope(\Qexm)$ (middle) according to the fibers of~$\projectionMap$ (right).}
  \label{fig:normalFan}
\end{figure}
\end{example}

\begin{example}[Duplicated word]
\label{exm:duplicatedNormalFan}
Consider the word~$\Qdup$ defined in Example~\ref{exm:duplicatedRoots}.
The normal fan of its brick polytope~$\brickPolytope(\Qdup)$ is the fan formed by the hyperplanes orthogonal to the roots~$\set{\alpha_p}{p \in P}$.
The vectors of the dual basis of~$\set{\alpha_p}{p \in P}$ are normal vectors for the facets of the brick polytope~$\brickPolytope(\Qdup)$.
Observe that the normal fan of the brick polytope~$\brickPolytope(\Q)$ is not in general defined by a subarrangement of the Coxeter arrangement (see \eg \fref{fig:normalFan}).
\end{example}

\begin{typeA}
In type~$A$, the root configuration~$\Roots{I}$ is the incidence configuration of the contact graph~$\Lambda_I\contact$ of~$\Lambda_I$ \cite{PilaudSantos-brickPolytope}.
The normal cone~$\normalCone(I)$ has one inequality~$v_i \le v_j$ for each arc~$(i,j)$ of~$\Lambda_I\contact$.
Therefore, an element~${w \in W}$ lies in the fiber~$\projectionMap^{-1}(I)$ if and only if it is a linear extension of the transitive closure of~$\Lambda_I\contact$.
\end{typeA}

%%%%%%%%%%%%%%%%%%%%%%%%%%%%%%%%%%%%%%

\subsection{Increasing flips and the weak order}
\label{subsec:weakOrder}

When $I$~and~$J$ are adjacent facets of~$\subwordComplex(\Q)$ with~${I \ssm i = J \ssm j}$, we say that the flip from $I$~to~$J$ is \defn{increasing} if~${i < j}$.
By Lemma~\ref{lem:roots&flips}\eqref{lem:roots&flips:enum:flip}, this condition is equivalent to~$\Root{I}{i} \in \Phi^+$, and by Lemma~\ref{lem:bricks&flips} to ${f(\brickVector(I)) < f(\brickVector(J))}$, where $f$ is a linear function positive on~$\Phi^+$.
The graph of increasing flips is thus clearly acyclic, and we call \defn{increasing flip order} its transitive closure~$\prec$.
This order generalizes the Tamari order on triangulations.

\begin{remark}
\label{rem:increasing flip graph}
The increasing flip graph was already briefly discussed by A.~Knutson and E.~Miller in~\cite[Remark~4.5]{KnutsonMiller-subwordComplex}.
In particular, they observe that one can directly obtain the $h$-vector of $\subwordComplex(\Q)$ from its increasing flip graph.
The lexicographic order on the facets of $\subwordComplex(\Q)$ is the shelling order induced by the vertex-decomposability.
It is moreover a linear extension of the increasing flip graph.
Therefore, a facet $I$ of $\subwordComplex(\Q)$ contributes to the entry $h_k$ of the $h$-vector $\hvector\big(\subwordComplex(\Q)\big) \eqdef (h_0,\ldots,h_{d+1})$ if there are $k$ incoming edges into $I$ in the increasing flip graph, \ie if~$|\Roots{I} \cap \Phi^+| = k$.
\end{remark}

Our next statement connects the increasing flip order on~$\subwordComplex(\Q)$ to the weak order on~$W$.

\begin{proposition}
\label{prop:quotient}
A facet~$I$ is covered by a facet~$J$ in increasing flip order if and only if there exist $w_I \in \projectionMap^{-1}(I)$~and~$w_J \in \projectionMap^{-1}(J)$ such that~$w_I$ is covered by~$w_J$ in weak~order.
In other words, the increasing flip graph on~$\subwordComplex(\Q)$ is the quotient of the Hasse diagram of the weak order by the fibers of the map~$\projectionMap$.
\end{proposition}

\begin{proof}
Let~$w \in W$ and~$s \in S$ such that~$w < ws$ in weak order.
Since~$w(\alpha_s) \in \Phi^+$, Lemma~\ref{lem:flipProjectiveMap} ensures that either~$\projectionMap(ws) = \projectionMap(w)$ or~$\projectionMap(ws)$ is obtained from~$\projectionMap(w)$ by an increasing flip.

Reciprocally, consider two facets~$I$ and~$J$ of~$\subwordComplex(\Q)$ related by an increasing flip.
Write $I \ssm i = J \ssm j$, with~$i<j$.
Their normal cones~$\normalCone(I)$ and~$\normalCone(J)$ are adjacent along a codimension~$1$ cone which belongs to the hyperplane~$H$ orthogonal to~$\Root{I}{i} = \Root{J}{j}$.
Let~$\fundamentalChamber_I$ be a chamber of the Coxeter arrangement of~$W$ contained in~$\normalCone(I)$ and incident to~$H$, and let~$\fundamentalChamber_J \eqdef s_{\Root{I}{i}}(\fundamentalChamber_I)$ denote the symmetric chamber with respect to~$H$.
Let~$w_I$ and~$w_J$ be such that~$\fundamentalChamber_I = w_I(\fundamentalChamber)$ and~$\fundamentalChamber_J = w_J(\fundamentalChamber)$.
Then~$w_I \in \projectionMap^{-1}(I)$ and~$w_J \in \projectionMap^{-1}(J)$, and~$w_J$ covers~$w_I$ in weak order.
\end{proof}

It immediately follows from Proposition~\ref{prop:quotient} that $I \prec J$ in increasing flip order if there exist $w_I \in \projectionMap^{-1}(I)$~and~$w_J \in \projectionMap^{-1}(J)$ such that~$w_I < w_J$ in weak order.
We conjecture that the converse also holds in general.

\begin{example}
\label{exm:recurrent6}
\fref{fig:latticeQuotient} represents the Hasse diagram of the increasing flip order on the facets of the subword complex~$\subwordComplex(\Qexm)$ of Example~\ref{exm:recurrent1}.
It is the quotient of the Hasse diagram of the weak order by the fibers of the map~$\projectionMap$ of \fref{fig:lattice}.

\begin{figure}
  \capstart
  \centerline{\includegraphics[scale=.54]{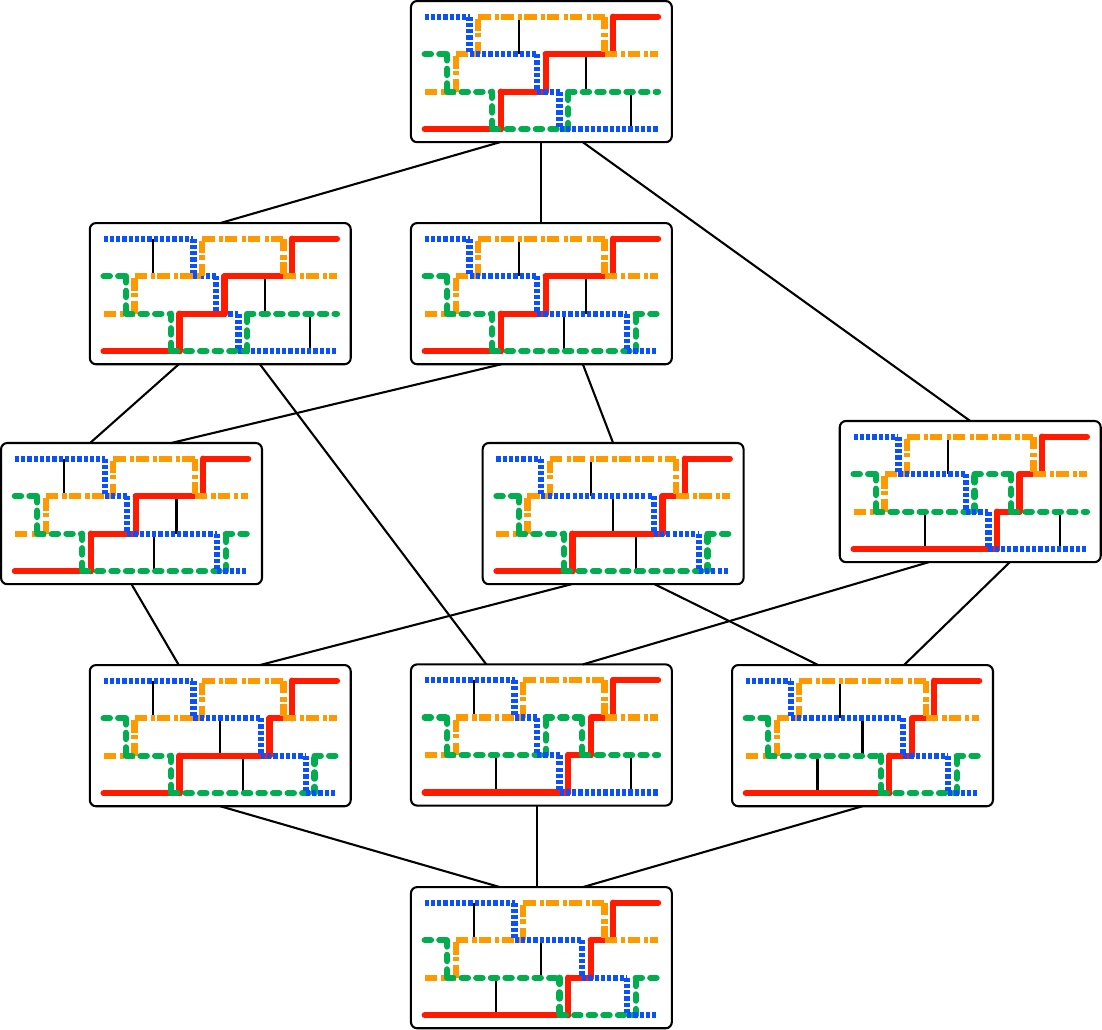}}
  \caption{The Hasse diagram of the increasing flip order on the facets of the subword complex~$\subwordComplex(\Qexm)$ of Example~\ref{exm:recurrent1} is the quotient of the Hasse diagram of the weak order on~$\fS_4$ by the fibers of the map~$\projectionMap$ represented in \fref{fig:lattice}.}
  \label{fig:latticeQuotient}
\end{figure}
\end{example}

\begin{example}[Duplicated word]
\label{exm:duplicatedLattice}
Consider the word~$\Qdup$ defined in Example~\ref{exm:duplicatedRoots}.
Then the quotient of the weak order by the fibers of~$\projectionMap$ is the boolean lattice.
\end{example}

\begin{remark}
\label{rem:latticeCounterExample}
In his work on Cambrian lattices \cite{Reading-latticeCongruences, Reading-CambrianLattices, Reading-sortableElements}, N.~Reading considers quotients of the weak order under lattice congruences, which fulfill sufficient conditions for the quotient to be a lattice.
In the context of brick polytopes, the fibers of~$\projectionMap$ induce different partitions of~$W$ which are not lattice congruences in general (the fibers are not necessarily intervals of the weak order, but they contain all intervals between two of their elements --- see Section~\ref{subsec:fibers}).
N.~Reading's construction raises the question if increasing flip orders are always lattices.
It turns out that this does not hold in general.
For the word~$\Q = \tau_1\tau_2\tau_3\tau_2\tau_1\tau_2\tau_3\tau_2\tau_1$ in the adjacent transpositions in $\fS_4$, the flip graph of the root independent subword complex~$\subwordComplex(\Q)$ is not a lattice.
Indeed, consider the four facets
$$I_1 = \{1,2,9\}, \quad I_2 = \{1,3,4\}, \quad J_1 = \{1,8,9\}, \quad J_2 = \{6,7,9\}$$
of $\subwordComplex(\Q)$.
For $i,j \in \{1,2\}$, we then have that $I_i < J_j$ in the increasing flip poset, but $J_1$ and $J_2$ are not comparable.
Thus, $I_1$ and $I_2$ do not have a unique join.
\end{remark}

Further combinatorial properties of increasing flip graphs and orders for any subword complexes, with applications to generation and order theoretic properties of subword complexes, can be found in~\cite{PilaudStump-ELlabeling}.

%%%%%%%%%%%%%%%%%%%%%%%%%%%%%%%%%%%%%%

\subsection{Fibers of \texorpdfstring{$\projectionMap$}{the projection} and the weak order}
\label{subsec:fibers}

In this section, we discuss properties of the fibers of the map~$\projectionMap$ with respect to the right weak order.
First, although these fibers are not always intervals (see Example~\ref{exm:recurrent4}), they are order-convex.

\begin{lemma}
If two elements~$\underline{w} \le \overline{w}$ of~$W$ lie in the same fiber of~$\projectionMap$, then the complete interval~$\set{w \in W}{\underline{w} \le w \le \overline{w}}$ lies in the same fiber of~$\projectionMap$.
\end{lemma}

\begin{proof}
As already observed in the proof of Proposition~\ref{prop:quotient}, $w < ws$ in weak order implies $\projectionMap(w) \le \projectionMap(ws)$ in increasing flip order, for any~$w \in W$.
Consequently, if~$\underline{w}$ and~$\overline{w}$ lie in the same fiber, then so does any element~$w$ on a chain between~$\underline{w}$ and~$\overline{w}$, and thus in the interval~$\set{w \in W}{\underline{w} \le w \le \overline{w}}$.
Another way to see this property will be given below by an alternative description of the map~$\projectionMap$.
\end{proof}

Next, we give an explicit description of the fibers of $\projectionMap$ in terms of inversion sets of elements of~$W$.
Remember from Section~\ref{subsec:coxeterGroups} that the inversion set of an element~$w \in W$ with reduced expression~$w = w_1w_2 \cdots w_p$ is the set
$$\inv(w) \eqdef \Phi^+ \, \cap \, w(\Phi^-) = \big\{\alpha_{w_1}, w_1(\alpha_{w_2}), \dots, w_1w_2 \cdots w_{p-1}(\alpha_{w_p})\big\} \subseteq \Phi^+,$$
and that the weak order coincides with the inclusion order on inversion sets: $w \leq w'$ if and only if $\inv(w) \subseteq \inv(w')$.

For finite Coxeter groups, the inversion sets of the elements of~$W$ are precisely the biclosed subsets of~$\Phi^+$.
A set~$U \subseteq \Phi^+$ is \defn{biclosed} if and only if both~$U$ and its complement~$\Phi^+ \ssm U$ are closed under non-negative linear combinations in~$\Phi^+$.
In other words, $U$ is biclosed if and only if~$U \cup -(\Phi^+ \ssm U)$ is separable.

For a facet~$I$ of the subword complex~$\subwordComplex(\Q)$, we define~$\RootsPlus{I} \eqdef \Roots{I} \, \cap \, \Phi^+$ and $\RootsMinus{I} \eqdef -(\Roots{I} \, \cap \, \Phi^-)$.
In other words,~$\RootsPlus{I}$ and~$\RootsMinus{I}$ are the subsets of~$\Phi^+$ such that~$\Roots{I} = \RootsPlus{I} \, \cup \, -\RootsMinus{I}$.
With these notations, the preimage of~$I$ under~$\projectionMap$ is given by all~$w \in W$ such that
$$\RootsMinus{I} \, \subseteq \, \inv(w) \, \subseteq \, \Phi^+ \ssm \RootsPlus{I}.$$

While the definition of the map $\projectionMap$ provides a direct geometric relation between separable sets and root configurations, this alternative description provides a relation to inversion sets, and thus to the weak order.
It moreover gives a criterion to characterize the fibers of~$\projectionMap$ which have a meet or a join.
For a facet~$I$ of~$\subwordComplex(\Q)$, define $\wedge(I)$ and $\vee(I)$ by
$$\wedge(I) \eqdef \!\!\!\! \bigcap_{w \in \projectionMap^{-1}(I)} \!\!\! \inv(w) \quad \text{and} \quad \vee(I) \eqdef \!\!\!\! \bigcup_{w \in \projectionMap^{-1}(I)} \!\!\! \inv(w).$$

\begin{remark}
\label{rem:computationPreimage}
The subsets~$\wedge(I)$ and~$\vee(I)$ can be computed from $\Roots{I}$ by the following iterative procedure.
Define two sequences~$(\wedge_p(I))_{p \in \N}$ and~$(\overline{\vee}_p(I))_{p \in \N}$ of collections of positive roots by
\begin{enumerate}[(i)]
\item $\wedge_0(I) \eqdef \RootsMinus{I}$ and~$\overline{\vee}_0(I) \eqdef \RootsPlus{I}$;
\item the set $\wedge_{p+1}(I)$ is the set of positive roots~$\alpha$ such that $\alpha = \lambda_1 \alpha_1 + \lambda_2 \alpha_2$ for some $\alpha_1, \alpha_2 \in \wedge_p(I)$ and some $\lambda_1, \lambda_2 \ge 0$, or such that $\lambda\alpha + \mu\beta \in \wedge_p(I)$ for some $\beta \in \overline{\vee}_p(I)$ and some $\lambda, \mu \ge 0$;
\item the set $\overline{\vee}_{p+1}(I)$ is the set of positive roots~$\beta$ such that $\beta = \mu_1 \beta_1 + \mu_2 \beta_2$ for some $\beta_1, \beta_2 \in \overline{\vee}_p(I)$ and some $\mu_1, \mu_2 \ge 0$, or such that $\lambda\alpha + \mu\beta \in \overline{\vee}_p(I)$ for some $\alpha \in \wedge_p(I)$ and some $\lambda, \mu \ge 0$.
\end{enumerate}
The sequences~$(\wedge_p(I))_{p \in \N}$ and~$(\overline{\vee}_p(I))_{p \in \N}$ are then stationary with limits~$\wedge(I)$ and~$\Phi^+ \ssm \vee(I)$.
\end{remark}

\begin{proposition}
\label{prop:fibers meet and join}
The fiber~$\projectionMap^{-1}(I)$ has a meet if and only if~$\wedge(I)$ is biclosed and a join if and only if~$\vee(I)$ is biclosed.
\end{proposition}

\begin{proof}
We first observe that,~$\projectionMap(w) = I$ if and only if~$\wedge(I) \subseteq \inv(w) \subseteq \vee(I)$.
If~$\wedge(I)$ is biclosed, it is the inversion set of an element~$\underline{w} \in W$.
Then $\underline{w}$ is in~$\projectionMap^{-1}(I)$ and it is smaller than all the elements of~$\projectionMap^{-1}(I)$ in weak order.
It is thus the meet of~$\projectionMap^{-1}(I)$.
Reciprocally, if the fiber~$\projectionMap^{-1}(I)$ has a meet, its inversion set is precisely~$\wedge(I)$, which is thus biclosed.
The proof is similar for the join.
\end{proof}

\begin{example}
\label{exm:recurrent7}
Consider the facet~$\{2,5,6\}$ of the subword complex~$\subwordComplex(\Qexm)$ of Example~\ref{exm:recurrent1}.
Its root configuration is~$\Roots{\{2,5,6\}} = \{e_4-e_2, e_4-e_1, e_1-e_3\}$, thus $\RootsPlus{\{2,5,6\}} = \{e_4-e_2, e_4-e_1\}$ and~$\RootsMinus{\{2,5,6\}} = \{e_3-e_1\}$.
Its fiber is $\projectionMap^{-1}(\{2,5,6\}) = \{2314, 3124, 3214\}$ (see \fref{fig:lattice}).
We have ${\wedge(\{2,5,6\}) = \{e_3-e_1\}}$ and $\vee(\{2,5,6\}) = \{e_2-e_1, e_3-e_2, e_3-e_1\}$.
The fiber~$\projectionMap^{-1}(\{2,5,6\})$ has no meet since $\wedge(\{2,5,6\})$ is not biclosed.
However, its join is~$3214$ since $\vee(\{2,5,6\}) = \inv(3214)$.
\end{example}

%%%%%%%%%%%%%%%%%%%%%%%%%%%%%%%%%%%%%%

\subsection{Minkowski sum decompositions into Coxeter matroid polytopes}
\label{subsec:minkowskiSum}

Remember that the Minkowski sum of two polytopes~$\Pi_1, \Pi_2 \subset V$ is the polytope
$$\Pi_1 + \Pi_2 \eqdef \bigset{v_1 + v_2}{v_1 \in \Pi_1, v_2 \in \Pi_2} \subset V.$$
A vertex~$v$ of~$\Pi_1 + \Pi_2$ maximizing a linear functional~$f : V \to \R$ on~$\Pi_1 + \Pi_2$ is a sum~$v_1 + v_2$ of vertices~$v_1 \in \Pi_1$ and~$v_2 \in \Pi_2$ maximizing~$f$ on~$\Pi_1$ and~$\Pi_2$ respectively.

\begin{proposition}
\label{prop:minkowskiSum}
The brick polytope~$\brickPolytope(\Q)$ is the Minkowski sum of the polytopes
$$\brickPolytope(\Q,k) \eqdef \conv \bigset{\Weight{I}{k}}{I \text{ facet of } \subwordComplex(\Q)}$$
over all positions~${k \in [\sizeQ]}$.
\end{proposition}

In other words, this proposition affirms that the sum and convex hull operators in the definition of the brick polytope commute:
$$\brickPolytope(\Q) \eqdef \conv_I \sum\nolimits_k \Weight{I}{k} \,=\, \sum\nolimits_k \conv_I\Weight{I}{k} \eqfed \sum\nolimits_k \brickPolytope(\Q,k),$$
where the index~$k$ of the sums ranges over the positions~$[\sizeQ]$ and the index~$I$ of the convex hulls ranges over the facets of the subword complex~$\subwordComplex(\Q)$.

\begin{proof}[Proof of Proposition~\ref{prop:minkowskiSum}]
The proof is identical to that of~\cite[Proposition~3.26]{PilaudSantos-brickPolytope}.
We repeat it here for the convenience of the reader.

First, by definition of the brick vector, the brick polytope~$\brickPolytope(\Q)$ is clearly a subset of the Minkowski sum~$\sum_k \brickPolytope(\Q,k)$.
To prove equality, we only need to show that any vertex of~$\sum_k \brickPolytope(\Q,k)$ is also a vertex of~$\brickPolytope(\Q)$.

Consider a linear functional~$f : V \to \R$.
For any two adjacent facets~$I,J$ of~$\subwordComplex(\Q)$ with~$I \ssm i = J \ssm j$, and any~$k \in [\sizeQ]$, we obtain that
\begin{itemize}
\item if $\Weight{I}{k} \notin \{\pm\Weight{I}{i}\}$, then $f(\Weight{I}{k}) = f(\Weight{J}{k})$;
\item otherwise, $f(\Weight{I}{k}) - f(\Weight{J}{k})$ has the same sign as~$f(\brickVector(I)) - f(\brickVector(J))$ by Lemma~\ref{lem:bricks&flips}.
\end{itemize}
Consequently, a facet~$I_f$ maximizes~$f(\brickVector(\cdot))$ among all facets of~$\subwordComplex(\Q)$ if and only if it maximizes~$f(\Weight{\cdot}{k})$ for any~$k \in [\sizeQ]$.

Now let~$v$ be any vertex of~$\sum_k \brickPolytope(\Q,k)$, and let~$f_v : V \to \R$ denote a linear functional maximized by~$v$ on~$\sum_k \brickPolytope(\Q,k)$.
Then~$v$ is the sum of the vertices in~$\brickPolytope(\Q,k)$ which maximize~$f_v$.
Thus,~$v = \sum_k \Weight{I_f}{k} = \brickVector(I_f)$ is a vertex of~$\brickPolytope(\Q)$.
\end{proof}

Moreover, we observe that each summand~$\brickPolytope(\Q,k)$ is a $W$-matroid polytope as defined in \cite{BorovikGelfandWhite2}.
A \defn{$W$-matroid polytope} is a polytope such that the group generated by the mirror symmetries of its edges is a subgroup of~$W$.
Equivalently, a $W$-matroid polytope is a polytope whose edges are directed by the roots of~$\Phi$, and whose vertex set lies on a sphere centered at the origin.
The edges of each summand~$\brickPolytope(\Q,k)$ are parallel to the roots since the edges of their Minkowski sum~$\brickPolytope(\Q)$ are, and the vertices of~$\brickPolytope(\Q,k)$ are equidistant from the origin since they are all images of~$w_{q_k}$ under the orthogonal group~$W$.
Thus, each summand~$\brickPolytope(\Q,k)$ is indeed a $W$-matroid polytope.

\begin{example}
\label{exm:recurrent9}
For the word~$\Qexm$ of Example~\ref{exm:recurrent1}, the Minkowski sum is formed by five points $\brickPolytope(\Qexm,k)$ for~$k \in \{1, 2, 3, 5, 8\}$, three segments $\brickPolytope(\Qexm,4) = \conv \{e_2,e_4\}$, $\brickPolytope(\Qexm,6) = \conv \{\one-e_1,\one-e_3\}$ and~$\brickPolytope(\Qexm,9) = \conv \{\one-e_3, \one-e_4\}$, and one triangle $\brickPolytope(\Qexm,7) = \conv \{e_1+e_2, e_2+e_3, e_2+e_4\}$.
\end{example}

\begin{example}[Duplicated word]
\label{exm:duplicatedMinkowski}
Consider the word~$\Qdup$ defined in Example~\ref{exm:duplicatedRoots}.
According to our description of the weight function~$\Weight{I_\varepsilon}{\cdot}$ given in Example~\ref{exm:duplicatedWeights}, we obtain that
\begin{itemize}
\item for~$k \in [N]$, the polytope~$\brickPolytope(\Qdup,k^*)$ is the single point~$\omega_k$;
\item for~$p \in P$, the polytope~$\brickPolytope(\Qdup,p^*+1)$ is the segment from~$\omega_p$ to~$\omega_p + \alpha_p$.
\end{itemize}
Thus, the polytope~$\brickPolytope(\Qdup)$ is the Minkowski sum of with the roots~$\set{\alpha_p}{p \in P}$ (considered as $1$-dimensional polytopes) with the point ${\Theta \eqdef \sum_{k \in [N]} \omega_k + \sum_{p \in P} \omega_p}$, that is, a parallelepiped whose edges are the roots~$\set{\alpha_p}{p \in P}$.
\end{example}

%%%%%%%%%%%%%%%%%%%%%%%%%%%%%%%%%%%%%%
%%%%%%%%%%%%%%%%%%%%%%%%%%%%%%%%%%%%%%

\section{Cluster complexes and generalized associahedra revisited}
\label{sec:generalizedAssociahedra}

The \defn{cluster complex} was defined in~\cite{FominZelevinsky-ClusterAlgebrasII} to encode the exchange graph of clusters in cluster algebras.
Its first polytopal realization was constructed by F.~Chapoton, S.~Fomin, and A.~Zelevinsky~\cite{ChapotonFominZelevinsky}.
A more general realization of $c$-cluster complexes for general finite Coxeter groups was later obtained by C.~Hohlweg, C.~Lange, and H.~Thomas~\cite{HohlwegLangeThomas}, based on the $c$-Cambrian fans of N.~Reading and D.~Speyer~\cite{Reading-CambrianLattices, Reading-sortableElements, ReadingSpeyer}.
Recently S.~Stella generalized in~\cite{Stella} the approach of~\cite{ChapotonFominZelevinsky} and showed that the resulting realizations of the associahedra coincide with that of~\cite{HohlwegLangeThomas}.

We explore in this section a completely different approach to construct generalized associahedra, based on subword complexes and brick polytopes.
In~\cite{CeballosLabbeStump}, C.~Ceballos, J.-P.~Labb\'e and the second author proved that the $c$-cluster complex is isomorphic to the subword complex~$\clustercomplex$, where~$\cwo{c}$ is the $\sq{c}$-sorting word for~$w_\circ \in W$.
In this section, we first study the brick polytope of these specific subword complexes, show that the polar of the brick polytope indeed realizes the cluster complex, and provide an explicit combinatorial description of its vertices, its normal vectors, and its facets.
Using these descriptions, we show that the brick polytope and the generalized associahedra in~\cite{HohlwegLangeThomas} only differ by an affine translation.
This opens new perspectives on the latter realization, and provides new proofs of several constructions therein, in particular their connections with the Cambrian lattices and fans studied in \cite{Reading-CambrianLattices, Reading-sortableElements, ReadingSpeyer, ReadingSpeyerInfinite}.
Finally, we show that one can directly compute the exchange matrix corresponding to a given facet of the $c$-cluster complex using the geometry of the subword complex.

\subsection{Cluster complexes as subword complexes and their inductive structure}
\label{subsec:CLS}

We first briefly recall some results from~\cite{CeballosLabbeStump} that we will use all along this section.
Consider a finite Coxeter system~$(W,S)$ and let~$c$ be a \defn{Coxeter element}, \ie the product of all simple reflections of~$S$ in a given order.
Following~\cite{Reading-sortableElements}, we choose an arbitrary reduced expression~$\sq{c}$ of~$c$ and let~$\sw{w}{c}$ denote the \defn{$\sq{c}$-sorting word} of~$w$, \ie the lexicographically first (as a sequence of positions) reduced subword of~$\sq{c}^\infty$ for~$w$.
In particular, let~$\cwo{c}$ denote the $\sq{c}$-sorting word of the longest element~$w_\circ \in W$.
Let~$N \eqdef \length(w_\circ) = |\Phi^+|$, and let~${m \eqdef n + N}$ be the length of~$\cw{c}$.
The following statement was proven in~\cite{CeballosLabbeStump}.

\begin{theorem}[\cite{CeballosLabbeStump}]
The subword complex $\clustercomplex$ is isomorphic to the $c$-cluster complex.
\end{theorem}

In Sections~\ref{sec:families} to \ref{sec:exchangematrix} we will need the following two properties of $\clustercomplex$ exhibited in~\cite{CeballosLabbeStump}.
These properties do not hold for general spherical subword complexes, and describe an important inductive structure of cluster complexes.
First, the rearrangement in Lemma~\ref{lem:commutationlemma} together with the rotation in Lemma~\ref{lem:rotationlemma} provides an isomorphism between the $c$-cluster complex and the $c'$-cluster complex for different Coxeter elements~$c$ and~$c'$.

\begin{lemma}[{\cite[Proposition~4.3]{CeballosLabbeStump}}]
\label{lem:isomorphismlemma}
Let~$s$ be an initial letter in~$c$ and set $c' \eqdef scs$ with a reduced expression~$\sq{c'}$.
The rotation operator in Lemma~\ref{lem:rotationlemma} provides an isomorphism between the subword complexes~$\clustercomplex$ and $\subwordComplex(\cw{c'})$.
\end{lemma}

The second lemma shows how a cluster complex contains cluster complexes for standard parabolic subgroups.

\begin{lemma}[{\cite[Lemma~5.2]{CeballosLabbeStump}}]
\label{lem:restrictionlemma}
Let $s$ be the initial letter in $\sq{c}$, and let~$I$ be a facet of $\clustercomplex$ for which $1 \in I$.
Then $\Root{I}{j} \in \Phi_{\langle \alpha_s \rangle}$ for all $j \in I$ with $j \neq 1$.
\end{lemma}

These two lemmas translate the following inductive structure of cluster complexes.
Let~$s$ be initial in~$c$, let~$\sq{c}$ be a word for~$c$ starting with~$s$, and let~$I$ be a facet of~$\clustercomplex$.
Then there are two cases:
\begin{enumerate}[(i)]
\item $1 \notin I$. We have by Lemma~\ref{lem:isomorphismlemma} that~$I$, shifted by one, is a facet of the subword complex for the word~$\cw{c'}$ with $c' \eqdef scs$.
\item $1 \in I$. Lemma~\ref{lem:restrictionlemma} ensures that $k \notin I$ for all $k$ such that $\Root{I}{k} \in \Phi \ssm \Phi_{\langle \alpha_s \rangle}$. By Proposition~\ref{prop:restriction}, we therefore obtain that $I$ can also be considered as a facet of the subword complex for the word $\cw{c'}$ with~$c' \eqdef sc$ (see also~\cite[Lemma~5.4]{CeballosLabbeStump}).
\end{enumerate}

The first $n$ positions of~$[\sizeQ]$ form a facet $I_e$ of the subword complex~$\clustercomplex$ since its complement~$\cwo{c}$ is a reduced expression for~$w_\circ$.
For convenience, we call this facet the \defn{initial facet}, and observe that it is the unique source in the increasing flip graph.
The notation~$I_e$ is due to the fact that $\kappa^{-1}(I_e) = \{e\}$ where $e \in W$ is the identity element.

%%%%%%%%%%%%%%%%%%%%%%%%%%%%%%%%%%%%%%

\subsection{Realizing cluster complexes}

We now construct geometric realizations of cluster complexes, based on their interpretation in terms of subword complexes.
The following theorem is an easy consequence of the results on brick polytopes developed in this paper.

\begin{theorem}
The polar of the brick polytope~$\brickPolytope(\cw{c})$ realizes the subword complex~$\clustercomplex$.
\end{theorem}

\begin{proof}
The root configuration~$\Roots{[n]}$ of the initial facet $I_e = [n]$ of the subword complex~$\clustercomplex$ coincides with the root basis~$\Delta$.
Thus, the subword complex $\subwordComplex(\cw{c})$ is root independent, and Theorem~\ref{theo:realization} ensures that it is isomorphic to the boundary complex of the polar of the brick polytope~$\brickPolytope(\cw{c})$.
\end{proof}

\begin{example}
\label{exm:associahedron1}
We illustrate the results of this section with the Coxeter element~$\cexm \eqdef \tau_1\tau_2\tau_3$ of~$A_3$.
The $\sq{\cexm}$-sorting word of~$w_\circ$ is ${\cwo{\cexm} = \sq{\tau}_1\sq{\tau}_2\sq{\tau}_3\sq{\tau}_1\sq{\tau}_2\sq{\tau}_1}$.
We thus consider the word~$\cw{\cexm} \eqdef \sq{\tau}_1\sq{\tau}_2\sq{\tau}_3\sq{\tau}_1\sq{\tau}_2\sq{\tau}_3\sq{\tau}_1\sq{\tau}_2\sq{\tau}_1$.
The facets of the subword complex~$\subwordComplex(\cw{\cexm})$ are~$\{1, 2, 3\}$, $\{2, 3, 4\}$, $\{3, 4, 5\}$, $\{4, 5, 6\}$, $\{5, 6, 7\}$, $\{6, 7, 8\}$, $\{6, 8, 9\}$, $\{4, 6, 9\}$, $\{3, 5, 7\}$, $\{2, 4, 9\}$, $\{1, 3, 7\}$, $\{1, 7, 8\}$, $\{1, 8, 9\}$, and~$\{1, 2, 9\}$.
Their brick vectors and the brick polytope~$\brickPolytope(\cw{\cexm})$ are represented in \fref{fig:associahedron1}.

\begin{figure}
  \capstart
  \centerline{\includegraphics[scale=.75]{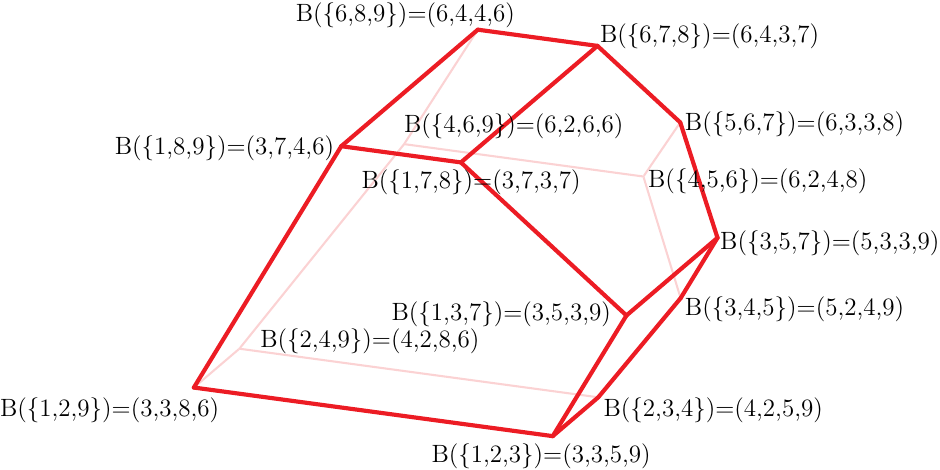}}
  \caption{The brick polytope~$\brickPolytope(\cw{\cexm})$, for the Coxeter element~$\cexm = \tau_1\tau_2\tau_3$, is J.-L.~Loday's $3$-dimensional associahedron.}
  \label{fig:associahedron1}
\end{figure}
\end{example}

%%%%%%%%%%%%%%%%%%%%%%%%%%%%%%%%%%%%%%

\subsection{Normal vectors and facets}

We define the \defn{weight configuration} of a facet~$I$ of~$\clustercomplex$ to be the set
$$\Weights{I} \eqdef \bigset{\Weight{I}{i}}{i \in I}$$
of all weights associated to the elements of~$I$.
Although not relevant for general words, the weight configuration is interesting for~$\cw{c}$ for the following reason.

\begin{proposition}
\label{prop:normalVectorsAssociahedra}
For any facet~$I$ of~$\clustercomplex$, and any position~$i \in I$, the weight $\Weight{I}{i}$ is a normal vector of the facet~$\cone  \set{\Root{I}{j}}{j \in I \ssm i}$ of the cone~$\rootCone(I)$.
Thus, the normal cone~$\normalCone(I)$ of~$\brickVector(I)$ in~$\brickPolytope(\cw{c})$ is generated by the weight configuration~$\Weights{I}$ of~$I$.
\end{proposition}

\begin{proof}
We derive this result from Lemmas~\ref{lem:isomorphismlemma} and~\ref{lem:restrictionlemma}.
If $1 \notin I$, we can apply Lemma~\ref{lem:isomorphismlemma} to rotate the first letter~$s$ in $\cw{c}$.
This rotation results in applying the reflection~$s$ both to the normal cone~$\rootCone(I)$ and to the weight configuration~$\Weights{I}$, and we are again in the same situation for the subword complex $\subwordComplex(\cw{c'})$ for the Coxeter element~${c' = scs}$.
We can thus assume that $1 \in I$.
Now using Lemma~\ref{lem:restrictionlemma}, the facet~$I$ contains, except for the first position~$1 \in I$, only letters~$j$ for which~$\Root{I}{j}$ belongs to~$\Phi_{\langle \alpha_{\sq{c}_1} \rangle}$.
Thus, $\Weight{I}{1}$ is orthogonal to all the roots of~$\Roots{I} \ssm \Root{I}{1}$.
\end{proof}

To obtain the inequality description of the brick polytope~$\brickPolytope(\cw{c})$, we study carefully certain facets of~$\clustercomplex$.
We need to introduce the following notations.
Let~$\w_1,\dots,\w_N$ denote the letters of the $\sq{c}$-sorting word of~$w_\circ$.
For~$\position \in [N]$, define ${\rho_\position \eqdef w_1 \cdots w_\position \in W}$ to be the product of the first~$\position$ letters of~$\cwo{c}$, and set by convention~$\rho_0 \eqdef e$.
Moreover, let
$$\alpha_\position \eqdef \rho_{\position-1}(\alpha_{w_\position}), \qquad \omega_\position \eqdef \rho_{\position-1}(\omega_{w_\position}), \qquad\text{and}\qquad \translation \eqdef \sum_{\position \in [N]} \omega_\position.$$
Remember that~$\Delta \eqdef \set{\alpha_s}{s \in S}$ denotes the set of simple roots of~$W$, while $\nabla \eqdef \set{\omega_s}{s \in S}$ denotes the set of fundamental weights of~$W$.

\begin{proposition}
\label{prop:structurePrefix}
For any~$\position \in \{0,\dots,N\}$, the root configuration, the weight configuration, and the brick vector of the facet~$\projectionMap(\rho_\position)$ of~$\clustercomplex$ are given by
$$\Roots{\projectionMap(\rho_\position)} = \rho_\position(\Delta), \quad \Weights{\projectionMap(\rho_\position)} = \rho_\position(\nabla) \quad\text{and}\quad \brickVector(\projectionMap(\rho_\position)) = \translation + \rho_\position \left( \sum_{s \in S} \omega_s \right).$$
\end{proposition}

\begin{proof}
For this proof, we need some additional definitions.
In the word~$\cw{c}$, we say that a position~$i$ is \defn{terminal} if there is no position~$j > i$ such that the~$i$\ordinal{} and $j$\ordinal{} letters of~$\cw{c}$ coincide.
Let~$\terminal$ denote the set of terminal positions.
For~$\position \in [N]$, let~$j_\position = \position+n+1$, and let~$i_\position$ denote the position in~$\cw{c}$ of the last occurrence of the $j_\position$\ordinal{} letter of~$\cw{c}$ before~$j_\position$.
Note that~${\position+1 \le i_\position \le \position+n}$, that~$i_\position + 1 \le i_{\position+1}$, and that all positions~$i$ with~$i_\position + 1 \le i < i_{\position+1}$ are terminal.

We first prove by induction on~$\position$ that $\projectionMap(\rho_\position) = (\terminal \cap [i_\position-1]) \cup \{i_\position, \dots, j_\position-1\}$ and that~$\Root{\projectionMap(\rho_\position)}{i_\position} = \Root{\projectionMap(\rho_\position)}{j_\position} = \alpha_{\position+1}$.
The result is clear for~$\position=0$, for which $\projectionMap(e) = [n]$, $i_0 = 1$ and~$j_0 = n+1$.
Assume that it holds for a given~$\position \in [N]$.
Since $\rho_{\position+1} = \rho_\position w_{(\position + 1)}$, we derive from Lemma~\ref{lem:flipProjectiveMap} that
\begin{align*}
\projectionMap(\rho_{\position+1}) & = (\projectionMap(\rho_\position) \ssm i_\position) \cup j_\position = (\terminal \cap [i_\position-1]) \cup \{i_\position+1, \dots, j_\position\} \\
& = (\terminal \cap [i_{\position+1}-1]) \cup \{i_{\position+1}, \dots, j_{\position+1}-1\}.
\end{align*}
By~$\{i_{\position+1}, \dots, j_{\position+1}-1\} \subseteq \projectionMap(\rho_{\position+1})$, we get~$\wordprod{\Q}{[i_{\position+1}-1] \ssm \projectionMap(\rho_{\position+1})} = \wordprod{\Q}{[j_{\position+1}-1] \ssm \projectionMap(\rho_{\position+1})}$, and thus~$\Root{\projectionMap(\rho_{\position+1})}{i_{\position+1}} = \Root{\projectionMap(\rho_{\position+1})}{j_{\position+1}}$.
Moreover, since~$\projectionMap(\rho_{\position+1}) \subseteq [j_\position]$, we get $\wordprod{\Q}{[j_\position] \ssm \projectionMap(\rho_{\position+1})} = \rho_{\position+1}$, and thus~$\Root{\projectionMap(\rho_{\position+1})}{j_{\position+1}} = \rho_{\position+1}(\alpha_{w_{(\position+2)}}) = \alpha_{\position+2}$.
This proves that the induction hypothesis holds for~$\position+1$.

From this description of the facet~$\projectionMap(\rho_\position)$, we can now derive its weight configuration~$\Weights{\rho_\position}$.
Consider a position~$i \in \projectionMap(\rho_\position)$, and let~$s$ denote the $i$\ordinal{} letter of~$\cw{c}$.
If~$i$ is terminal, then ${\Weight{\projectionMap(\rho_\position)}{i} = w_\circ(\omega_s) = \rho_\position(\omega_s)}$.
If~$i \in \{i_\position, \dots, j_\position-1\}$, then $\Weight{\projectionMap(\rho_\position)}{i} = \wordprod{\Q}{[i_\position-1] \ssm \projectionMap(\rho_{\position})} = \rho_\position(\omega_s)$ since ${\{i_\position, \dots, j_\position-1\} \subseteq \projectionMap(\rho_\position) \subseteq [j_\position-1]}$.
In both cases, we have~$\Weight{\projectionMap(\rho_\position)}{i} = \rho_\position(\omega_s)$ for all~${i \in \projectionMap(\rho_\position)}$.
Since each simple reflection~$s \in S$ appears in the word~$\cw{c}$ either at a terminal position before~$i_\position$, or at a position in~$\{i_\position, \dots, j_\position-1\}$, we obtain that~$\Weights{\projectionMap(\rho_\position)} = \rho_\position(\nabla)$.

We can now compute the brick vector~$\brickVector(\projectionMap(\rho_\position))$.
First, we have that
\[
\sum_{i \in \projectionMap(\rho_\position)} \Weight{\projectionMap(\rho_\position)}{i} = \rho_\position \left( \sum_{s \in S} \omega_s \right).
\]
Second, using that~$\projectionMap(\rho_\position) \ssm i_\position = \projectionMap(\rho_{\position+1}) \ssm j_\position$, that $\{i_\position, \dots, j_\position-1\} \subseteq \projectionMap(\rho_\position)$, and that the~$i_\position$\ordinal{} and~$j_\position$\ordinal{} letters of~$\cw{c}$ coincide, we observe that the reduced expression of~$w_\circ$ given by the complement of~$\projectionMap(\rho_\position)$ is always~$\cwo{c}$.
This ensures that
\[
\sum_{i \notin \projectionMap(\rho_\position)} \Weight{\projectionMap(\rho_\position)}{i} = \sum_{\position \in [N]} \omega_{\position} \eqfed \translation.
\]
Summing the two contributions, we derive the brick vector
\[
\brickVector(\projectionMap(\rho_\position)) \eqdef \sum_{i \in [\sizeQ]} \Weight{\projectionMap(\rho_\position)}{i} = \translation + \rho_\position \left( \sum_{s \in S} \omega_s \right).
\]

Finally, according to Lemma~\ref{lem:bricks&flips}, the root configuration~$\Roots{\projectionMap(\rho_\position)}$ is the set of roots orthogonal to the facet defining hyperplanes of the normal cone of~$\brickVector(\projectionMap(\rho_\position))$ in~$\brickPolytope(\cw{c})$.
Since this cone is generated by~$\Weights{\projectionMap(\rho_\position)} = \rho_\position(\nabla)$, we obtain that~$\Roots{\projectionMap(\rho_\position)} = \rho_\position(\Delta)$.
\end{proof}

\begin{corollary}
\label{coro:normalVectorsAssociahedra}
The normal vectors of the facets of the brick polytope~$\brickPolytope(\cw{c})$ are given by~$-\nabla \cup \set{\omega_\position}{\position \in [N]}$.
\end{corollary}

\begin{proof}
By Proposition~\ref{prop:normalVectorsAssociahedra}, the normal vectors of the facets of the brick polytope~$\brickPolytope(\cw{c})$ are the weights~$\Weight{I}{i}$ for~$I$ facet of~$\clustercomplex$ and~$i \in I$.
Moreover, the weight~$\Weight{I}{i}$ is independent of the facet~$I$ of~$\clustercomplex$ containing~$i$.
The result thus follows from Proposition~\ref{prop:structurePrefix} since the facets~$\projectionMap(\rho_\position)$ cover all positions of~$[m]$, and the union of their weight configurations is precisely~$-\nabla \cup \set{\omega_\position}{\position \in [N]}$.
\end{proof}

Let~$q \eqdef \sum_{s \in S} \omega_s$, and let~$\Perm \eqdef \conv \set{w(q)}{w \in W}$ be the balanced \mbox{$W$-per}\-mutahedron.

\begin{theorem}
\label{theo:removingFacets}
The brick polytope~$\brickPolytope(\cw{c})$ is obtained from the balanced permutahedron~$\Perm$ removing all facets which do not intersect~$\set{\rho_\position(q)}{0 \le \position \le N\}}$, and then translating by the vector~$\translation \eqdef \sum_{\position \in [N]} \omega_\position$.
\end{theorem}

\begin{proof}
Let~$\brickPolytope'$ denote the translate of~$\brickPolytope(\cw{c})$ by the vector~$-\translation$.
By Proposition~\ref{prop:structurePrefix}, the vertices of~$\brickPolytope'$ and~$\Perm$ corresponding to the prefixes~$\rho_\position$ of~$\cwo{c}$ coincide, as well as their normal cones.
Thus, the facets of~$\brickPolytope'$ and~$\Perm$ containing these vertices are defined by the same inequalities.
The result thus follows from Corollary~\ref{coro:normalVectorsAssociahedra} which affirms that these inequalities exhaust all facets of~$\brickPolytope'$.
\end{proof}

\begin{corollary}
\label{coro:HohlwegLangeThomas}
Up to the translation by the vector~$\translation$, the brick polytope~$\brickPolytope(\cw{c})$ coincides with the \mbox{$c$-asso}\-ciahedron $\Asso$ in~\cite{HohlwegLangeThomas}.
\end{corollary}

\begin{remark}
By a recent work of S.~Stella~\cite{Stella}, the \mbox{$c$-asso}\-ciahedron $\Asso$ also coincides, up to an affine translation, with the first realization of the cluster complex by F.~Chapoton, S.~Fomin, and A.~Zelevinsky~\cite{ChapotonFominZelevinsky} and its generalization by S.~Stella~\cite{Stella}.
We thus obtained in the present paper the vertex description of both realizations (by Definition~\ref{def:weights}) and a natural Minkowski sum decomposition into $W$-matroid polytopes (by Proposition~\ref{prop:minkowskiSum}).

In a subsequent paper, we will use this new approach to generalized associahedra to show that the vertex barycenters of all $c$-associahedra $\Asso$ coincide with those of the corresponding $W$-permutahedron $\Perm[q][W]$~\cite{PilaudStump-barycenter}.
This generalizes results by C.~Hohlweg and J.~Lortie and A.~Raymond in the case of the symmetric group~\cite{HohlwegLortieRaymond} and answers several questions asked by C.~Hohlweg in~\cite{Hohlweg}.
\end{remark}

\begin{example}
\label{exm:associahedron2}
For the Coxeter element~$\cexm \eqdef \tau_1\tau_2\tau_3$ of Example~\ref{exm:associahedron1}, the translation vector is~$\translation = (3,2,3,6)$.
We have represented in \fref{fig:associahedron2} the balanced \mbox{$3$-di}\-mensional permutahedron~$\Perm[][A_3] \eqdef \conv \set{\sigma(0,1,2,3)}{\sigma \in \fS_4}$ together with the $3$-dimensional associahedron~${\brickPolytope(\cw{\cexm})-\translation}$.
The set~$\set{\rho_\position(q)}{0 \le \position \le 6}$ are all the vertices located on the leftmost path from the bottommost vertex to the topmost vertex in the graphs of these polytopes (they are marked with strong red dots).
The facets of the permutahedron~$\Perm[][A_3]$ containing these points, which give precisely the facets of the associahedron~${\brickPolytope(\cw{\cexm})-\translation}$, are shaded.

\begin{figure}
  \capstart
  \centerline{\includegraphics[scale=.75]{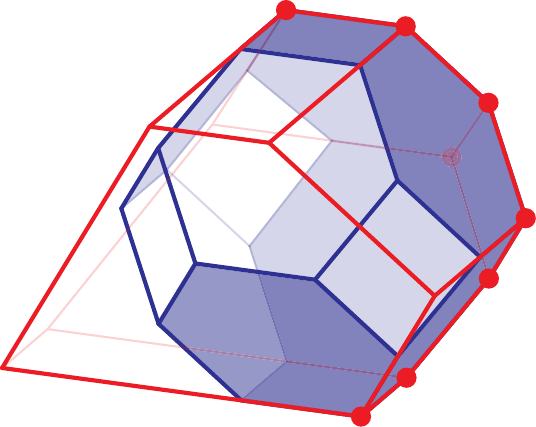}}
  \caption{The brick polytope~$\subwordComplex(\cw{\cexm})$ is obtained removing facets of the balanced permutahedron~$\Perm[][A_3]$. The remaining facets of the permutahedron are shaded.}
  \label{fig:associahedron2}
\end{figure}
\end{example}

In the end of this paper, we revisit various results contained in~\cite{HohlwegLangeThomas}.
The approach developed in this paper, based on subword complexes and brick polytopes, enables us to simplify certain proofs and connections.

First, we provide a new way to see that the $c$-singletons as defined in~\cite{HohlwegLangeThomas} can be used to characterize the vertices common to the brick polytope~$\brickPolytope(\cw{c})$ and to the balanced $W$-permutahedron~$\Perm$.
In the following list, immediate reformulations of the same characterization are indicated by a --- symbol.

\begin{proposition}
\label{prop:singletons}
The following properties are equivalent for an element~$w \in W$:
\begin{enumerate}[(i)]
\item
\label{prop:singletons:enum:singleton}
The element $w$ is a \defn{singleton} for the map~$\projectionMap$, \ie $\projectionMap^{-1}(\projectionMap(w)) = \{w\}$.

\item
\label{prop:singletons:enum:roots}
The root configuration of~$\projectionMap(w)$ is given by~$\Roots{\projectionMap(w)} = w(\Delta)$.

\item[---\;]
\label{prop:singletons:enum:simpleSystem}
The root configuration~$\Roots{\projectionMap(w)}$ of~$\projectionMap(w)$ is a \defn{simple system} for~$W$, that is, the set of extremal roots of a separable subset of~$\Phi$.

\item[---\;]
\label{prop:singletons:enum:rootCone}
The cone~$\rootCone(\projectionMap(w))$ of~$\brickPolytope(\cw{c})$ at~$\brickVector(\projectionMap(w))$ is the polar cone of the chamber~$w(\fundamentalChamber)$ of the Coxeter arrangement of~$W$, \ie ${\rootCone(\projectionMap(w)) = w(\fundamentalChamber)\polar = w(\fundamentalChamber\polar)}$.

\item
\label{prop:singletons:enum:weights}
The weight configuration~$\Weights{\projectionMap(w)}$ of~$\projectionMap(w)$ is given by~$\Weights{\projectionMap(w)} = w(\nabla)$.

\item[---\;]
\label{prop:singletons:enum:chamber}
The weight configuration~$\Weights{\projectionMap(w)}$ of~$\projectionMap(w)$ generates a chamber of the Coxeter arrangement of~$W$.

\item[---\;]
\label{prop:singletons:enum:normalCone}
The normal cone~$\normalCone(\projectionMap(w))$ of~$\brickVector(\projectionMap(w)) \in \brickPolytope(\cw{c})$ coincides with the chamber~$w(\fundamentalChamber)$ of the Coxeter arrangement of~$W$, \ie ${\normalCone(\projectionMap(w)) = w(\fundamentalChamber)}$.

\item
\label{prop:singletons:enum:vertices}
The vertices~$\brickVector(\projectionMap(w))$ of the brick polytope~$\brickPolytope(\cw{c})$ and~$w(q)$ of the balanced $W$-permutahedron~$\Perm$ coincide up to~$\translation$, \ie~${\brickVector(\projectionMap(w)) = \translation + w(q)}$.

\item
\label{prop:singletons:enum:prefix}
There exist reduced expressions~$\w$ of~$w$ and~$\sq{c}$ of~$c$ such that~$\w$ is a prefix of~$\cwo{c}$.

\item
\label{prop:singletons:enum:reducedExpression}
The complement of~$\projectionMap(w)$ in~$\cw{c}$ is~$\cwo{c}$.
\end{enumerate}
\end{proposition}

\begin{figure}[p]
  \capstart
  \centerline{\includegraphics[scale=.63]{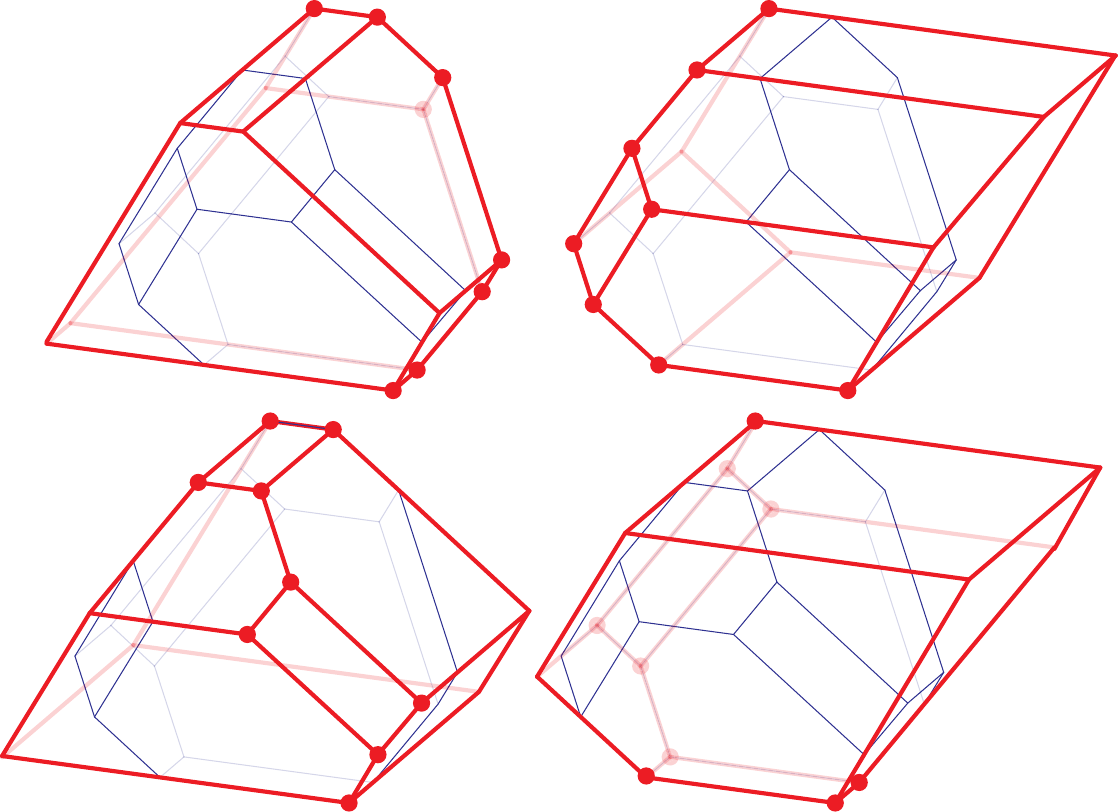}}
  \caption{The $c$-associahedron of type~$A_3$ for the Coxeter element~$c$ being $\tau_1\tau_2\tau_3$ (top left), $\tau_3\tau_2\tau_1$ (top right), $\tau_2\tau_1\tau_3$ (bottom left), and $\tau_3\tau_1\tau_2$ (bottom right).}
  \label{fig:typeAassociahedra}
\end{figure}
\begin{figure}[p]
  \capstart
  \centerline{\includegraphics[scale=.63]{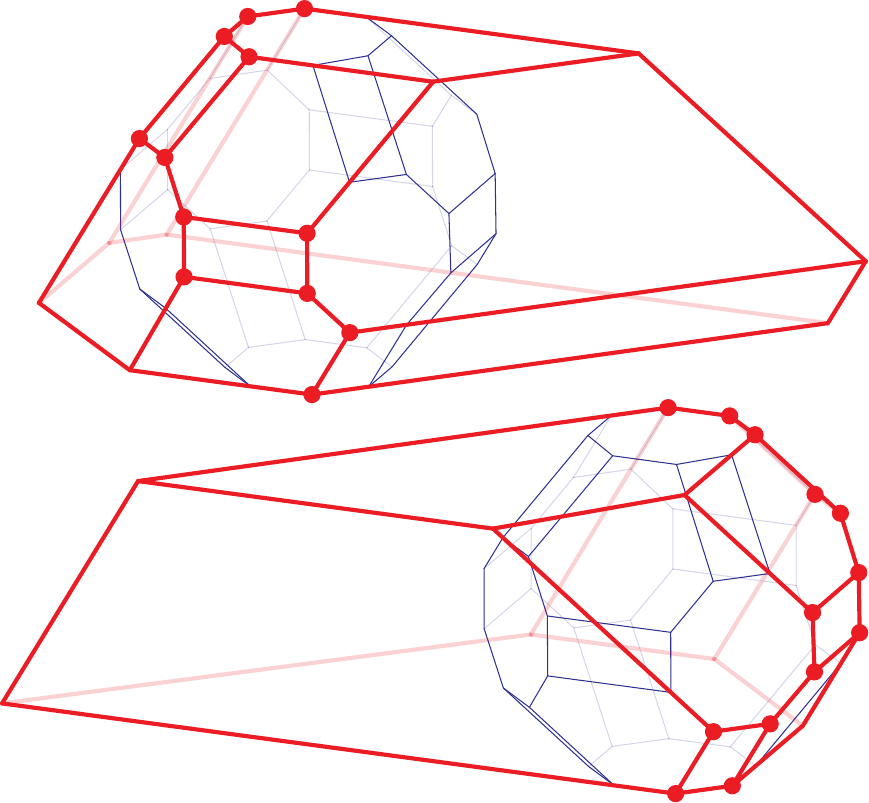}}
  \caption{The $c$-associahedron of type~$B_3$ for the Coxeter element~$c$ being $\tau_2\tau_1\chi$ (top) and $\tau_2\chi\tau_1$ (bottom).}
  \label{fig:typeBassociahedra}
\end{figure}
\begin{figure}
  \capstart
  \centerline{\includegraphics[scale=.63]{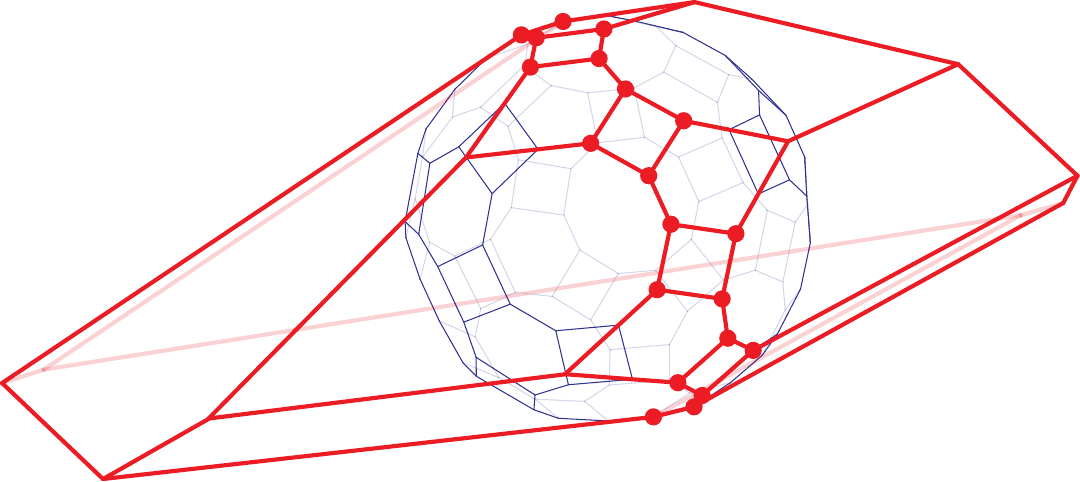}}
  \caption{An $H_3$-associahedron.}
  \label{fig:typeHassociahedra}
\end{figure}

\vspace{-.2cm}

\begin{proof}
Properties~(\ref{prop:singletons:enum:singleton}), (\ref{prop:singletons:enum:roots}), and (\ref{prop:singletons:enum:weights}), as well as their respective variants, are equivalent according to Lemma~\ref{lem:flipProjectiveMap}, Proposition~\ref{prop:normalCone}, and Proposition~\ref{prop:normalVectorsAssociahedra}.

If~$w$ violates~(\ref{prop:singletons:enum:prefix}), then at least one of the weights of~$w(\nabla)$ does not belong to~${-\nabla \cup \set{\omega_\position}{\position \in [N]}}$.
Proposition~\ref{prop:normalVectorsAssociahedra} then affirms that $w$~also violates~(\ref{prop:singletons:enum:normalCone}).
Thus, (\ref{prop:singletons:enum:normalCone}) implies~(\ref{prop:singletons:enum:prefix}).

We proved in Proposition~\ref{prop:structurePrefix} that any prefix of~$\cwo{c}$ satisfies~(\ref{prop:singletons:enum:vertices}).
It is now simple to see that the isomorphism between~$\clustercomplex$ and~$\subwordComplex(\cw{c'})$ given in Lemma~\ref{lem:isomorphismlemma} commutes with the root configuration~$\Roots{\cdot}$, the brick vector~$\brickVector(\cdot)$ and the map~$\projectionMap$.
In particular, the functions~$\Roots{\projectionMap(\cdot)}$ and~$\brickVector(\projectionMap(\cdot))$ are independent of the choice of the reduced expression~$\sq{c}$ of~$c$.
Consequently, Proposition~\ref{prop:structurePrefix} ensures that (\ref{prop:singletons:enum:prefix}) implies~(\ref{prop:singletons:enum:vertices}).

Similarly, we have seen in the proof of Proposition~\ref{prop:structurePrefix} that (\ref{prop:singletons:enum:reducedExpression}) holds when $w$ is a prefix of~$\cwo{c}$.
Up to commutations, we thus have that (\ref{prop:singletons:enum:prefix}) implies (\ref{prop:singletons:enum:reducedExpression}).

Since the permutahedron~$\Perm$ is a simple polytope, the only vertices which can survive when removing a set~$\cF$ of facets of~$\Perm$ are those which do not belong to any facet of~$\cF$.
Consequently, Theorem~\ref{theo:removingFacets} ensures that (\ref{prop:singletons:enum:vertices}) implies~(\ref{prop:singletons:enum:normalCone}).

Finally, if the complement of~$\projectionMap(w)$ in~$\cw{c}$ is precisely~$\cwo{c}$, then the weights of~$\projectionMap(w)$ all lie in~${-\nabla \cup \set{\omega_\position}{\position \in [N]}}$.
Thus, all facets incident to~$\projectionMap(w)$ when deleting the facets of the permutahedron, and (\ref{prop:singletons:enum:reducedExpression}) implies (\ref{prop:singletons:enum:normalCone}).
\end{proof}

\begin{remark}
\label{rem:changeBasePoint}
For completeness, we want to point out that all definitions and results of Sections~\ref{sec:rootConfiguration} and~\ref{sec:brickPolytope} easily generalize if we replace~$\Root{I}{k}$ by~$\lambda(q_k)\Root{I}{k}$ and~$\Weight{I}{k}$ by~$\lambda(q_k)\Weight{I}{k}$ for an arbitrary function~$\lambda: S \to \R_{>0}$.
The resulting polytope~$\brickPolytope_\lambda(\Q)$ has the same combinatorial properties for any~$\lambda$, but its geometry is different.
We decided to present our construction and results with~$\lambda \equiv 1$ to simplify the exposition.
In the special case of the subword complex~$\subwordComplex{\cw{c}}$, the resulting brick polytope~$\brickPolytope_\lambda(\cw{c})$ is a translate of the $c$-associahedron~$\Asso[x]$ in~\cite{HohlwegLangeThomas}, where $x = \sum_{s \in S} \lambda(s)\omega_s$.
This polytope is obtained from the permutahedron~$\Perm[x][W]$ removing all facets which do not intersect~$\set{\rho_\position(x)}{0 \le \position \le N}$.
\end{remark}

\begin{example}
\label{exm:associahedra}
Figures~\ref{fig:typeAassociahedra}, \ref{fig:typeBassociahedra}, and \ref{fig:typeHassociahedra} represent type~$A_3$, $B_3$, and~$H_3$ associahedra obtained by removing facets of the type~$A_3$, $B_3$, and~$H_3$ permutahedra of Figures~\ref{fig:typeAarrangement}, \ref{fig:typeBarrangement}, and~\ref{fig:typeHarrangement}.
They are all obtained from the corresponding permutahedron removing facets not containing $c$-singletons, which are marked with dots.
\end{example}

\begin{remark}
In regard of the results of this section, it is tempting to try to generalize the geometric description of the generalized associahedra to all brick polytopes.
We observe however that there exist root independent subword complexes $\subwordComplex(\Q)$ for which:
\begin{itemize}
\item The weight configuration~$\Weights{I} = \set{\Weight{I}{i}}{i \in I}$ of a facet~$I$ does not necessarily generate the normal cone of~$\brickVector(I)$ in~$\brickPolytope(\Q)$.

\item No element~$w \in W$ is a singleton for the map~$\projectionMap$.
In other words, no chamber of the Coxeter fan is the normal cone of a vertex of the brick polytope~$\brickPolytope(\Q)$.
In particular, the chambers~$\fundamentalChamber$ and~$w_\circ(\fundamentalChamber)$ are not necessarily normal cones of the brick polytope.

\item No translate of the brick polytope~$\brickPolytope(\Q)$ is obtained from a $W$-permuta\-hedron by deleting facets.
\end{itemize}

The word~$\Qexm = \sq{\tau}_2\sq{\tau}_3\sq{\tau}_1\sq{\tau}_3\sq{\tau}_2\sq{\tau}_1\sq{\tau}_2\sq{\tau}_3\sq{\tau}_1$ from Example~\ref{exm:recurrent1} provides an interesting example.
First, the normal cone of~$\brickVector({\{2,3,9\}})$ is not generated by the weight configuration~$\Weights{\{2,3,9\}}$ since~$\Weight{\{2,3,9\}}{2} = e_4  \not\perp e_3-e_4 = \Root{\{2,3,9\}}{9}$.
Second, neither~$\fundamentalChamber$ nor~$w_\circ(\fundamentalChamber)$ are normal cones of~$\brickPolytope(\Qexm)$, see Figures~\ref{fig:lattice} and~\ref{fig:normalFan}.
Finally, the brick polytope~$\brickPolytope(\Qexm)$ cannot be obtained from a permutahedron by removing facets.
Indeed, one can see in \fref{fig:brickPolytope} that the distance between the two pentagonal facets of~$\brickPolytope(\Qexm)$ (orthogonal to the vector~$(1,-3,1,1)$) is half the distance between the topmost and bottommost rectangular facets of~$\brickPolytope(\Qexm)$ (orthogonal to~$(-3,1,1,1)$), while the corresponding facets of any $A_3$-permutahedron are at equal distance.
The duplicated word~$\sq{\tau}_1\sq{\tau}_1\sq{\tau}_2\sq{\tau}_2\sq{\tau}_3\sq{\tau}_3\sq{\tau}_1\sq{\tau}_2\sq{\tau}_1$ in type~$A_3$ even gives an example where the map~$\projectionMap$ has no singleton since all fibers have~$2$~or~$6$ elements.
\end{remark}

%%%%%%%%%%%%%%%%%%%%%%%%%%%%%%%%%%%%%%

\subsection{Five Coxeter-Catalan families}
\label{sec:families}

In this section, we temporally interrupt our geometric study of the cluster complex to describe and relate the combinatorics of several Coxeter-Catalan families.
We will need these relations in the following section to connect brick polytopes to Cambrian lattices and fans.
We revisit here the bijective connections between
\begin{enumerate}[(i)]
\item the set~$\cluster c W$ of $c$-clusters,
\item the set~$\sortable c W$ of $c$-sortable elements,
\item the set~$\ncp c W$ of $c$-noncrossing partitions,
\item the set~$\ncs c W$ of $c$-noncrossing subspaces,
\end{enumerate}
and describe their connection to the set of facets of~$\clustercomplex$ which we denote for convenience in this section by~$\facet{c}{W}$.
The bijective connections are summarized in \fref{fig:bijections}.

\begin{figure}[ht]
  \capstart
  \begin{tikzpicture}[description/.style={fill=white, inner sep=3pt}]
    \matrix (m) [matrix of math nodes, row sep=5em, column sep=0em]
      { \cluster c W & \phantom{\facet{c}{W}} & \sortable c W & & \ncp c W & \phantom{\facet{c}{W}} & \ncs c W \\
      & & & \facet{c}{W} \\ };
    \path[->, font=\normalsize]
      (m-1-3) edge node [above] {$\sortableCluster{c}$} (m-1-1)
      (m-1-1) edge [bend left=20, transform canvas={yshift=.1cm}] node [above] {$\clusterNcp{c}$} (m-1-5)
      (m-1-3) edge node [above] {$\sortableNcp{c}$} (m-1-5)
%       (m-1-5) edge node [above] {$\ncpNcs$} (m-1-7)
      (m-2-4) edge [-right to] node [above] {$\facetCluster{c}$} (m-1-1)
      (m-1-1) edge [-right to, transform canvas={yshift=-.05cm,xshift=-.05cm}] (m-2-4)
      (m-2-4) edge node [right] {$\facetSortable{c}$} (m-1-3)
      (m-2-4) edge node [left] {$\facetNcp{c}$} (m-1-5)
      (m-1-5) edge [-right to, transform canvas={yshift=-.05cm,xshift=-.05cm}] (m-1-7)
      (m-1-7) edge [-right to] node [above] {$\ncpNcs$} (m-1-5);
%       (m-1-7) edge node [above] {$\ncsFacet{c}$} (m-2-4);
  \end{tikzpicture}
  \caption{Coxeter-Catalan families and their connections.}
  \label{fig:bijections}
\end{figure}
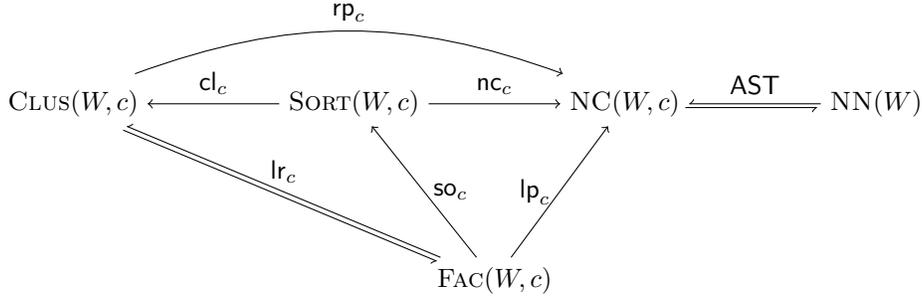

All connections between these Coxeter-Catalan families are based on their analogous inductive structure, described in Section~\ref{subsec:CLS}.
Although the proofs heavily depend on this inductive structure, most maps can also be described directly.

\subsubsection{Facets of the subword complex and clusters}
\label{sec:facetCluster}

Since this connection is the most straightforward, we start with the identification between~$c$-clusters and facets of~$\clustercomplex$ as given in~\cite{CeballosLabbeStump}.
To this end, let~$\Phi_{\ge -1} \eqdef -\Delta \cup \Phi^+$ be the set of \defn{almost positive roots}.
Denote by $\sw{w}{c} \eqdef \w_1 \cdots \w_{\length(w)}$ the \defn{$\sq{c}$-sorting word} of some element $w \in W$, \ie the lexicographically first reduced subword of~$\sq{c}^\infty$ for~$w$, and let~${\sizeQ \eqdef n + \length(w)}$ be the length of the word $\sq{c} \sw{w}{c}$.
Associate to each position~$i \in [\sizeQ]$ in this word an almost positive root~$\facetCluster{c}(w, i)$ by
$$\facetCluster{c}(w, i) \eqdef \begin{cases} -\alpha_{c_i} & \text{if } i \le n, \\ w_{1} \cdots w_{i-n-1}(\alpha_{w_{i-n}}) & \text{otherwise}. \end{cases}$$
In the case where~$w$ is the longest element~$w_\circ$, it is shown in~\cite[Theorem~2.2]{CeballosLabbeStump} that the map~$\facetCluster{c}(\cdot) \eqdef \facetCluster{c}(w_\circ, \cdot)$ induces a bijection~$\facetCluster{c} : \facet{c}{W} \longrightarrow \cluster{c}{W}$ between facets of~$\clustercomplex$ and $c$-clusters,
$$\facetCluster{c} : \begin{array}{ccc} \facet{c}{W} & \longrightarrow & \cluster{c}{W} \\ I & \longmapsto & \set{\facetCluster{c}(i)}{i \in I} \end{array}.$$
In other words, write once and for all the set of almost positive roots~$\Phi_{\ge -1}$ as the list $\facetCluster{c}(1), \dots, \facetCluster{c}(\sizeQ)$.
Then the cluster~$\facetCluster{c}(I)$ associated to a facet~$I$ in~$\facet{c}{W}$ is given by its sublist at positions in $I$, while the facet~$\facetCluster{c}^{-1}(X)$ associated to a $c$-cluster $X$ is the set of positions of $X$ within this list.
Note in particular that the elements in the $c$-cluster~$X$ are naturally ordered by their position in this list.
We here use this description of $c$-clusters as the definition of the set $\cluster c W$ and refer the reader to~\cite{Reading-coxeterSortable} and~\cite{CeballosLabbeStump} for further background.

\begin{example}
\label{exm:associahedron3}
Consider the Coxeter element~$\cexm \eqdef \tau_1\tau_2\tau_3$ of~$A_3$ already discussed in Example~\ref{exm:associahedron1} .
The $\sq{\cexm}$-sorting word for~$w_\circ$ is~$\cw{\cexm} = \sq{\tau}_1\sq{\tau}_2\sq{\tau}_3\sq{\tau}_1\sq{\tau}_2\sq{\tau}_1$.
Therefore, the list~$\facetCluster{\cexm}(1), \dots, \facetCluster{\cexm}(\sizeQ)$ is given by
$$e_1-e_2, \; e_2-e_3, \; e_3-e_4, \; e_2-e_1, \; e_3-e_1, \; e_4-e_1, \; e_3-e_2, \; e_4-e_2, \; e_4-e_3.$$
Consider now the facet~$\Fexm \eqdef \{4, 6, 9\} \in \facet{\cexm}{A_3}$.
Its associated cluster is given by~$\facetCluster{\cexm}(\Fexm) = \{e_2-e_1, e_4-e_1, e_4-e_3\}$, while its root configuration is given by~$\Roots{\Fexm} = \{e_3-e_2, e_1-e_3, e_3-e_4\}$.
The following proposition gives a direct relation between these two sets of roots.
\end{example}

The next proposition connects the root configuration~$\Roots{I}$ to the $c$-cluster~$\facetCluster{c}(I)$ associated to a facet~$I \in \facet{c}{W}$.
This connection enables us to compute~$\Roots{I}$ from~$\facetCluster{c}(I)$ and \viceversa.
We will need it later in the description of the bijection ${\clusterNcp{c}:\cluster{c}{W} \longrightarrow \ncp{c}{W}}$.

\begin{proposition}
\label{pr:cluster product}
Let~$I \eqdef \{i_1 < \dots < i_n\}$ be a facet in~$\facet{c}{W}$.
Consider its root configuration~$\Roots{I} = \{\alpha_1, \dots, \alpha_n\}$, where $\alpha_p \eqdef \Root{I}{i_p}$, and the associated $c$-cluster~$\facetCluster{c}(I) = \{\beta_1,\dots,\beta_n\}$, where~$\beta_p \eqdef \facetCluster{c}(i_p)$.
Moreover, let $j$ be such that~$i_j \le n$ and~$i_{j+1} > n$, or in other words, ${\{\beta_1,\ldots,\beta_j\} \subseteq -\Delta}$, ${\{\beta_{j+1},\ldots,\beta_n\} \subseteq \Phi^+_{\langle \beta_1,\ldots,\beta_j\rangle}}$.
For any~$p$ with~$j < p \leq n$, we have
\begin{align}
\alpha_p = - s_{\beta_n} \cdots s_{\beta_{p+1}}(\beta_p) \quad \text{and} \quad \beta_p = - s_{\alpha_n} \cdots s_{\alpha_{p+1}}(\alpha_p). \label{eq:alphas and betas}
\end{align}
Moreover, we have
\begin{align}
  s_{\alpha_{j+1}} \cdots s_{\alpha_{n}} = s_{\beta_n} \cdots s_{\beta_{j+1}} &= c', \label{eq:reordering}
\end{align}
where $c'$ is obtained from $c$ by removing all letters $s_{\beta_1},\ldots,s_{\beta_j}$.
Finally, the reflections $s_{\alpha_{j+1}},\ldots,s_{\alpha_n}$ can be reordered by transpositions of consecutive commuting reflections such that all positive roots in $\big\{\alpha_{j+1},\ldots,\alpha_n\big\}$ appear first in this product, followed by the negatives.
\end{proposition}

\begin{proof}
Starting with $\alpha_n = - \beta_n$, it is straightforward to check that both statements in \eqref{eq:alphas and betas} are equivalent.
We thus only show the second equality.
We write $\cw{c} \eqdef \q_1 \cdots \q_\sizeQ$, and observe that by definition,
\begin{align}
\beta_p &=  w_\circ q_\sizeQ q_{\sizeQ-1} \hspace{0.95cm} \cdots \hspace{1.28cm} q_{i_p}(\alpha_{q_{i_p}}), \notag \\
\alpha_p &= w_\circ q_\sizeQ q_{\sizeQ-1} \cdots \hat q_{i_n} \cdots \hat q_{i_{p+1}} \cdots q_{i_p+1}(\alpha_{q_{i_p}}), \tag{$*$} \label{eq:alphasinbetas}
\end{align}
where the hats on the~$\hat q_{i_n},\dots,\hat q_{i_{p+1}}$ mean that they are omitted in the expression for~$\alpha_p$.
Thus applying successively $s_{\beta_n},\dots,s_{\beta_{p+1}}$ to $\alpha_p$ inserts successively $q_{i_n},\dots,q_{i_{p+1}}$ into the righthand side of~\eqref{eq:alphasinbetas}, statement~\eqref{eq:alphas and betas} follows.
To prove~\eqref{eq:reordering}, observe that the first equality follows from~\eqref{eq:alphas and betas} with $p = j+1$.
The second equality is directly deduced from~\cite[Proposition~2.8]{CeballosLabbeStump} with $k=1$.
To see the last part of this proposition, observe that one can directly check that this statement holds if $W$ is a dihedral group.
In general, if two reflections $s_{\alpha_k}$ and $s_{\alpha_{k'}}$ in this product for which one root is positive while the other is negative, do not commute, the positions $k$ and $k'$ form a facet of the subword complex reduced to the parabolic subgroup of $W$ generated by $s_{\alpha_k}$ and $s_{\alpha_{k'}}$, and thus $s_{\alpha_k}$ must be positive, while $s_{\alpha_{k'}}$ must be negative.
Compare Proposition~\ref{prop:restriction} for the construction of the subword complex reduced to a parabolic subgroup.
\end{proof}

\subsubsection{Facets of the subword complex and sortable elements}
\label{subsubsec:facetSortable}

\enlargethispage{.4cm}
Recall the definition of $c$-sortable elements from~\cite{Reading-sortableElements}.
Let~$w$ be an element of~$W$.
Its $\sq{c}$-sorting word can be written as~$\sw{w}{c} = \sq{c}_{K_1}\sq{c}_{K_2}\cdots\sq{c}_{K_p}$, where~$\sq{c}_{K}$ denotes the subword of~$\sq{c}$ only taking the simple reflections in~$K \subseteq S$ into account.
The element~$w$ is called \defn{$c$-sortable} if its $\sq{c}$-sorting word~$\sw{w}{c} = \sq{c}_{K_1}\sq{c}_{K_2}\cdots\sq{c}_{K_p}$ is nested, \ie if ${K_1\supseteq K_2\supseteq\cdots\supseteq K_p}$.
Observe that the property of being $c$-sortable does not depend on the particular reduced expression~$\sq{c}$ but only on~$c$.
We denote by~$\sortable{c}{W}$ the set of $c$-sortable elements in~$W$.

N.~Reading defines the bijection $\sortableCluster{c} : \sortable{c}{W} \longrightarrow \cluster{c}{W}$ between $c$-sortable elements and $c$-clusters in~\cite[Theorem~8.1]{Reading-sortableElements} by
$$\sortableCluster{c}(w) \eqdef \bigset{\facetCluster{c}(i)}{w_i \text{ is the last occurrence of a letter in $\sq{c}\sw{w}{c}$}}.$$
He also provides the following inductive description of $c$-sortable elements.
Let $w \in W$, and let $s$ be \defn{initial} in $c$, \ie $\length(sc) < \length(c)$.
Then
$$w \in \sortable c W \iff
\begin{cases}
  sw \in \sortable{scs}{W} &\text{if } \length(sw) < \length(w), \\
  w \in \sortable{sc}{W_{\langle s \rangle}} &\text{if } \length(sw) > \length(w).
\end{cases}$$
The first case yields an induction on the length $\length(w)$, while the second case is an induction on the rank of $W$.
This inductive structure is analogous to that of the subword complex~$\clustercomplex$ presented in Section~\ref{subsec:CLS}, and thus yield a bijection
$$\facetSortable{c} \eqdef \sortableCluster{c}^{-1} \circ \facetCluster{c} : \facet{c}{W} \longrightarrow \sortable c W.$$
In this paper, we have already seen two ways of computing $\facetSortable{c}$:
\begin{enumerate}
\item The first uses the inductive description: the $\sq{c}$-sorting word $\w$ of $w = \facetSortable{c}(I)$ is obtained by looking at the first letter $s$ of $\cw{c}$. If $1 \notin I$, this letter is appended to $\w$, and the first letter is rotated. Otherwise, it is not appended to $\w$ and all positions~$k$ for which $\Root{I}{k} \in \Phi \ssm \Phi_{\langle \alpha_s \rangle}$ are deleted.
\item The second uses the iterative procedure described in Remark~\ref{rem:computationPreimage} and Proposition~\ref{prop:fibers meet and join}.
\end{enumerate}
Observe that the first follows the same lines as the description in~\cite[Remark~6.9]{Reading-sortableElements}, see also Section~\ref{sec:sortableNcp} below.

\begin{example}
\label{exm:associahedron4}
Consider the facet~$\Fexm \eqdef \{4, 6, 9\} \in \facet{\cexm}{A_3}$ already mentioned in Example~\ref{exm:associahedron3}.
Its root configuration is~${\Roots{\Fexm} = \{e_3-e_2, e_1-e_3, e_3-e_4\}}$ so that~${\RootsPlus{\Fexm} = \{e_3-e_2\}}$ and~${\RootsMinus{\Fexm} = \{e_3-e_1, e_4-e_3\}}$.
Consequently, we have ${\wedge(\Fexm) = \{e_3-e_1, e_4-e_3, e_4-e_1, e_2-e_1\}}$, which is the inversion set of the $\cexm$-sortable element~$\facetSortable{c}(\Fexm) = 2431 = \tau_1\tau_2\tau_3|\tau_2$.
\fref{fig:CambrianLattice} represents all facets of~$\facet{\cexm}{A_3}$ and their corresponding $\cexm$-sortable elements by the map~$\facetSortable{\cexm}$.
\end{example}

In the following, we use the inductive structure of $\subwordComplex(\cwo{c})$ to see that the root configuration was already present in~\cite{ReadingSpeyerInfinite}.
Therein, N.~Reading and D.~Speyer defined inductively a set~$\Skips{c}{w}$ associated to a $\sq{c}$-sortable element $w$,
$$\Skips{c}{w} \eqdef
\begin{cases}
  s\Skips{scs}{sw} &\text{if } \length(sw) < \length(w), \\
  \Skips{sc}{w} \cup \{\alpha_s\} &\text{if } \length(sw) > \length(w).
\end{cases}$$
This set is intimately related to the root configuration.
\begin{proposition}
\label{prop:thesetC}
For any facet $I \in \facet{c}{W}$, the root configuration $\Roots{I}$ and the set $\Skips{c}{\facetSortable{c}(I)}$ coincide.
\end{proposition}

\begin{proof}
This is a direct consequence of the inductive description of $c$-sortable elements and of the facets of $\clustercomplex$ presented above.
\end{proof}

Proposition~\ref{prop:thesetC} yields the following description of~$\Skips{c}{w}$.
Let~${w \in W}$ be $c$-sortable with $\sq{c}$-sorting word $\w \eqdef \w_1 \cdots \w_{\length(w)}$, let $s \in S$, and let~$k \in [\length(w)]$.
We say that~$w$ \defn{skips} $s$ in position $k+1$ if the leftmost instance of $s$ in $\sq{c}^\infty$ not used in $\w$ occurs after position~$k$.
A skip $s$ at position $k+1$ is moreover called \defn{forced} if $w_1 \cdots w_k(\alpha_s) \in \Phi^-$, and \defn{unforced} otherwise.

\begin{corollary}[{\cite[Proposition~5.1]{ReadingSpeyerInfinite}}]
\label{prop:skipunskip}
We have
\begin{align*}
  \Skips{c}{w} &= \bigset{w_1 \cdots w_k(\alpha_s)}{ w \text{ skips } s \text{ at position } k+1}
\end{align*}
and the forced and unforced skips are given by the intersection of $\Skips{c}{w}$ with $\Phi^-$ and with $\Phi^+$, respectively.
\end{corollary}

\begin{remark}
  Remark~\ref{rem:reconstructFromRoots} yields a shortcut for $\facetSortable{c}^{-1} = \facetCluster{c}^{-1} \circ \sortableCluster{c}: \sortable{c}{W} \longrightarrow \facet{c}{W}$ using Proposition~\ref{prop:thesetC} and its Corollary~\ref{prop:skipunskip}.
  Observe moreover that the distinction between forced and unforced skips naturally arises when we sweep the word~$\cw{c}$ from left to right as described in Remark~\ref{rem:reconstructFromRoots}.
\end{remark}

\subsubsection{Facets of the subword complex and noncrossing partitions}
\label{sec:sortableNcp}

We start with recalling the definition of noncrossing partitions as given for example in~\cite[Section~5]{Reading-sortableElements} (see also Lemma~5.2 therein for an inductive description).
Define the \defn{reflection length} $\reflength(w)$ for $w \in W$ to be the length of the shortest factorization of $w$ into reflections.
Observe that in particular, the reflection length of any Coxeter element $c$ is $n$.
An element $w \in W$ is a \defn{noncrossing partition} if~${\reflength(w) + \reflength(w^{-1}c) = n}$.
The set of all noncrossing partitions is denoted by $\ncp c W$.
The term \emph{partition} has historical reasons and comes from the connection to noncrossing set partitions in the classical types.
The $\sq{c}$-sorting word $\sw{w}{c} \eqdef \w_1 \dots \w_{\ell(w)}$ of some element $w \in W$ induces a total order on its inversion set
$${\inv(w) \eqdef \big\{ \alpha_{w_1}, w_1(\alpha_{w_2}), \dots, w_1w_2 \cdots w_{\ell(w)-1}(\alpha_{w_{\ell(w)}})\big\} \subseteq \Phi^+}.$$
The (ordered) set of \defn{cover reflections} of~$w$ is given by
$$\covers{w} \eqdef \bigset{t \in R}{tw = ws < w \text{ for some } s \in S} \subseteq \bigset{s_\beta}{\beta \in \inv(w)}.$$
In~\cite[Theorem~6.1]{Reading-sortableElements}, N.~Reading showed that the map
$$
  \sortableNcp{c} :
    \begin{array}{ccc}
      \sortable c W & \longrightarrow & \ncp c W \\
      w & \longmapsto & t_1 \cdots t_k
    \end{array}
$$
is a bijection, where $\{t_1, \dots, t_k\} = \covers{w}$ is the ordered set of cover reflections~of~$w$.

Using this result and the description of $\Skips{c}{w}$ in the previous section, we are now in the situation to provide a direct connection between facets in $\facet{c}{W}$ and noncrossing partitions.
A position~$i$ in a facet~$I \in \facet{c}{W}$ is an \defn{upper position} in~$I$ if~$\Root I i \in \Phi^-$.
Moreover, its corresponding almost positive root $\facetCluster{c}(i)$ is an \defn{upper root} in~$\facetCluster{c}(I)$.
Compare the original equivalent definition of upper roots before Corollary~8.6 in~\cite{ReadingSpeyer}.

\begin{theorem}
\label{thm:facetNcp}
  The bijection $\facetNcp{c} = \sortableNcp{c} \circ \facetSortable{c}$ is given by
  $$
    \facetNcp{c} :
      \begin{array}{ccc}
         \facet{c}{W} & \longrightarrow & \ncp c W \\
         I & \longmapsto & s_{\Root{I}{j_1}} \cdots s_{\Root{I}{j_k}},
      \end{array}
   $$
   where $\{ j_1 < \dots < j_k \}$ is the collection of upper positions in $I$.
\end{theorem}

\begin{proof}
The construction of $\facetSortable{c} : \facet{c}{W} \longrightarrow \sortable{c}{W}$ implies that the sorted set of cover reflections of $\facetSortable{c}(I)$ is given by $\Roots{I} \cap \Phi^-$.
This can as well be deduced from~\cite[Proposition~5.2]{ReadingSpeyerInfinite} together with Proposition~\ref{prop:thesetC}.
The statement then follows from the above description of the map $\sortableNcp{c} : \sortable{c}{W} \longrightarrow \ncp{c}{W}$.
\end{proof}

\begin{example}
\label{exm:associahedron5}
Consider the facet~$\Fexm \eqdef \{4, 6, 9\} \in \facet{\cexm}{A_3}$ mentioned in Examples~\ref{exm:associahedron3} and~\ref{exm:associahedron4}.
Its root configuration is ${\Roots{\Fexm} = \{e_3-e_2, e_1-e_3, e_3-e_4\}}$.
Consequently, its upper positions are~$6$ and~$9$ and its associated $c$-noncrossing partition is~$\facetNcp{\cexm}(\Fexm) = (13)(34) = 3241$.
\end{example}

Using Theorem~\ref{thm:facetNcp} together with the description of the $h$-vector in Remark~\ref{rem:increasing flip graph}, we immediately obtain the following corollary for the number of noncrossing partitions of a given reflection length.
\begin{corollary}
\label{coro:ncphvector}
  Let $\hvector\big(\clustercomplex\big) = (h_0,\ldots,h_n)$ be the $h$-vector of the subword complex $\clustercomplex$. Then
  \begin{align}
    \big| \set{w \in \ncp c W}{\reflength(w) = i} \big| = h_{n-i}. \label{eq:ncphvector}
  \end{align}
\end{corollary}

\begin{theorem}
\label{th:clusterNcp}
  The bijection $\clusterNcp{c} = \facetNcp{c} \circ \facetCluster{c}^{-1}$ is given by
  $$
    \clusterNcp{c} :
      \begin{array}{ccc}
         \cluster c W & \longrightarrow & \ncp c W \\
         \{ \beta_1,\ldots,\beta_n\} & \longmapsto & s_{\beta_{j_k}} \cdots s_{\beta_{j_1}},
      \end{array}
   $$
   where the product is taken over all upper roots from right to left (recall here that a $c$-cluster is naturally ordered, see Section~\ref{sec:facetCluster}).
\end{theorem}

\begin{proof}
  Observe that Proposition~\ref{pr:cluster product} immediately implies that the bijection $\clusterNcp{c}$ as given in Theorem~\ref{th:clusterNcp} is well-defined.
  Let $I = \{ i_1 < \cdots < i_n \}$ be the facet in $\facet{c}{W}$ such that $\facetCluster{c}(I)$ is given by $\{ \beta_1,\ldots,\beta_n\}$.
  The set of upper positions in $I$ is then given by $\{ j_1 < \cdots < j_k \} \subseteq I$.
  We have seen in Proposition~\ref{pr:cluster product}, that the reflections corresponding to upper positions appear as a suffix of the product $s_{\alpha_{j+1}} \cdots s_{\alpha_{n}}$ on the left-hand side of Proposition~\ref{pr:cluster product}\eqref{eq:reordering}.
  Together with Proposition~\ref{pr:cluster product}\,\eqref{eq:alphas and betas}, we thus obtain that
  $$s_{\beta_{i_k}} \cdots s_{\beta_{i_1}} = s_{\facetCluster{c}(i_k)} \cdots s_{\facetCluster{c}(i_1)} = s_{\Root{I}{i_1}} \cdots s_{\Root{I}{i_k}}.$$
  The statement follows.
\end{proof}

\begin{example}
\label{exm:associahedron6}
Consider the $\cexm$-cluster~$\Xexm \eqdef \{e_2-e_1, e_4-e_1, e_4-e_3\}$.
Its upper roots are~$e_4-e_1$ and~$e_4-e_3$.
Consequently, its associated noncrossing partition is~$\clusterNcp{c}(\Xexm) = (34)(14) = 3241$.
\end{example}

\begin{remark}
  The composition of $\clusterNcp{c}$ with the map $\ncpNcs : \ncp c W \longrightarrow \ncs c W$ was described in~\cite[Theorem~11.1]{ReadingSpeyer} using some further recursive notions, see also the paragraph thereafter.
  We here provide a direct map, not using the recursive structure.
  Theorem~\ref{th:clusterNcp} provides as well new proofs of~\cite[Conjectures~11.2~and~11.3]{ReadingSpeyer} which were as well proven in~\cite[Theorem~4.10]{IngallsThomas}.
\end{remark}

Finally, we sketch an inductive description of~$\ncsFacet{c} : \ncs c W \longrightarrow \facet{c}{W}$.
To construct the image~$\ncsFacet{c}(L)$ of a $c$-noncrossing subspace~$L \in \ncs c W$, consider two cases depending on the last letter~$\q_\sizeQ$ of~$\cw{c}$.
We set $s \eqdef \psi(\q_\sizeQ)$, which we can assume to be final in $c$, and we define $c' \eqdef scs$ and~$c'' \eqdef cs = sc'$.
The two cases depend on the inductive description of the facets describing how to embed $\facet{c'}{W}$ and $\facet{c''}{W_{\langle s \rangle}}$ into $\facet{c}{W}$.

\begin{enumerate}[(i)]
\item If~$L$ is not contained in the reflection hyperplane of~$s$, $\ncsFacet{c}(L)$ is given by the facet~$\ncsFacet{c'}(sL)$ inside $\facet{c'}{W}$.
\item If~$L$ is contained in the reflection hyperplane of~$s$, we obtain~$\ncsFacet{c}(L)$ by adding the position~$\sizeQ$ to the facet~$\ncsFacet{c''}(L)$ inside $\facet{c''}{W_{\langle s \rangle}}$.
\end{enumerate}
Observe here, that this procedure considers a final letter in $c$ while the procedure for computing $\facetSortable{c} : \facet{c}{W} \longrightarrow \sortable{c}{W}$ considers an initial letter.

\begin{example}
\label{ex:ncpFacet}
\label{exm:associahedron7}
Consider the Coxeter element~$\cexm \eqdef \tau_1\tau_2\tau_3$ of Example~\ref{exm:associahedron1}, whose corresponding word is~$\cw{\cexm} = \sq{\tau}_1\sq{\tau}_2\sq{\tau}_3\sq{\tau}_1\sq{\tau}_2\sq{\tau}_3\sq{\tau}_1\sq{\tau}_2\sq{\tau}_1$.
Let~$\Lexm$ denote the fixed space of~$3241 = (13)(34) = (34)(14)$.
It is a $\cexm$-noncrossing subspace.
To compute~$\ncsFacet{\cexm}(\Lexm)$, we set~$s \eqdef \psi(\tau_1) = \tau_3$, $c' \eqdef scs = \tau_3\tau_1\tau_2$ and~${c'' \eqdef cs = sc'=\tau_1\tau_2}$.
Since~$\Lexm$ is contained in the reflection hyperplane of~$s$, we compute by induction~$\ncsFacet{c''}(\Lexm) = \{3,4\}$.
Embedding this facet in~$\cw{\cexm}$ and adding position~$9$, we obtain that~$\ncsFacet{\cexm}(\Lexm) = \{4, 6, 9\}$.
\end{example}

To conclude this section, we summarize in \fref{fig:bijectionsExamples} all correspondences obtained in Examples~\ref{exm:associahedron3},~\ref{exm:associahedron4},~\ref{exm:associahedron5},~\ref{exm:associahedron6}, and~\ref{exm:associahedron7}.
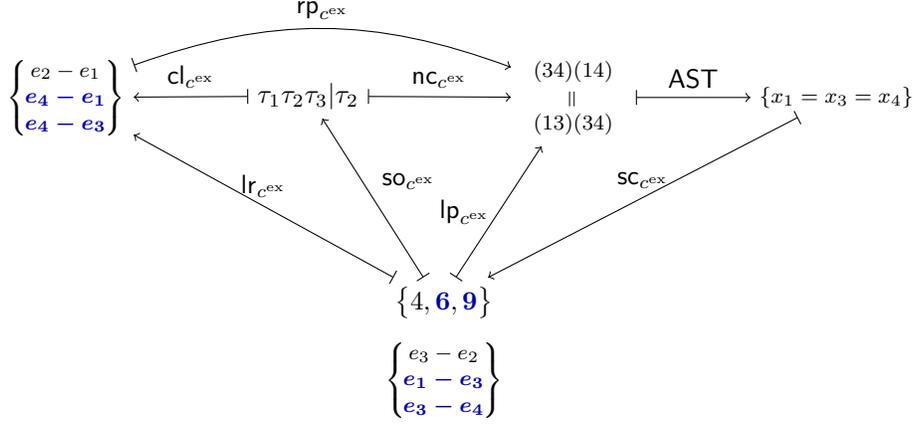
\begin{figure}[ht]
  \capstart
  \begin{tikzpicture}[description/.style={fill=white, inner sep=3pt}]
    \matrix (m) [matrix of math nodes, row sep=5em, column sep=0em]
      {{\mbox{\footnotesize $\left\{\!\!\!\!\begin{array}{ccc} e_2-e_1 \\ \boldsymbol{\darkblue e_4-e_1} \\ \boldsymbol{\darkblue e_4-e_3} \end{array}\!\!\!\!\right\}$}}
      & \phantom{\big\{4, \boldsymbol{6}, \boldsymbol{9} \big\}} & \tau_1\tau_2\tau_3|\tau_2 & &
        {\mbox{\footnotesize $\begin{array}{ccc} (34)(14) \\ \rotatebox[origin=c]{90}{$=$} \\ (13)(34)\end{array}$}}
        & \phantom{\big\{4, \boldsymbol{6}, \boldsymbol{9} \big\}} & {\mbox{\footnotesize$\{ x_1 = x_3 = x_4 \}$}} \\
      & & & \ \big\{4, \boldsymbol{\darkblue 6}, \boldsymbol{\darkblue 9} \big\} & & & \\[-1.7cm]
      & & & \ {\mbox{\footnotesize $\left\{\!\!\!\!\begin{array}{ccc} e_3-e_2 \\ \boldsymbol{\darkblue e_1-e_3} \\ \boldsymbol{\darkblue e_3-e_4} \end{array}\!\!\!\!\right\}$}} \\
          };
    \path[|->, font=\normalsize]
      (m-1-3) edge node [above] {$\sortableCluster{\cexm}$} (m-1-1)
      (m-1-1) edge [bend left=20, transform canvas={yshift=.1cm}] node [above] {$\clusterNcp{\cexm}$} (m-1-5)
      (m-1-3) edge node [above] {$\sortableNcp{\cexm}$} (m-1-5)
      (m-1-5) edge node [above] {$\ncpNcs$} (m-1-7)
      (m-2-4) edge node [above] {$\facetCluster{\cexm}$} (m-1-1)
      (m-2-4) edge node [above right] {$\facetSortable{\cexm}$} (m-1-3)
      (m-2-4) edge node [left] {$\facetNcp{\cexm}$} (m-1-5)
      (m-1-7) edge node [above] {$\ncsFacet{\cexm}$} (m-2-4);
  \end{tikzpicture}
  \caption{The complete picture of the ongoing example. The upper positions in the facet are drawn in {\darkblue{\bf bold blue}}, and below its root configuration is shown.}
  \label{fig:bijectionsExamples}
\end{figure}

%%%%%%%%%%%%%%%%%%%%%%%%%%%%%%%%%%%%%%

\subsection{Cambrian lattices and fans}
\label{sec:cambrian}

Using the connections established in the previous section, we relate in this section the Cambrian lattice with the increasing flip order and the Cambrian fan with the normal fan of the brick polytope.

The $c$-sortable elements of~$W$ form a sublattice and a quotient lattice of the weak order on~$W$, see~\cite{Reading-CambrianLattices}.
This lattice is called \defn{$c$-Cambrian lattice}.
The downward projection used to construct the quotient associates to any element~$w \in W$ the maximal $c$-sortable element below~$w$.
It is denoted by~$\pidown$, see~\cite[Corollary~6.2]{ReadingSpeyerInfinite}.
In~\cite[Theorem~6.3]{ReadingSpeyerInfinite}, N.~Reading and D.~Speyer give a geometric description of the downward projection $\pidown$.
This implies first that the bijection $\facetSortable{c}^{-1}$ between $c$-sortable elements and facets of the subword complex $\clustercomplex$ is the restriction to $c$-sortable elements of the surjective map $\projectionMap$ from $W$ to facets as defined in Section~\ref{subsec:surjectiveMap}.
It moreover yields the connection between~$\projectionMap$ and the map~$\pidown$.

\begin{theorem}
\label{theo:pi=kappa}
Let $v,w \in W$ and let $v$ be $c$-sortable.
Then
$$\pidown(w) = v \iff \projectionMap(w) = \projectionMap(v).$$
\end{theorem}

\begin{proof}
By \cite[Proposition~5.1, Theorem~6.3]{ReadingSpeyerInfinite} and polarity,~$\pidown(w) = v$ if and only if~$\Skips{c}{v} \subseteq w(\Phi^+)$.
As $v$ is $c$-sortable, Proposition~\ref{prop:thesetC} gives us that $\Skips{c}{v} = \Roots{\projectionMap(v)}$.
The statement follows with the observation that $\Roots{\projectionMap(v)} \subseteq w(\Phi^+)$ if and only if $\projectionMap(w) = \projectionMap(v)$ by definition.
\end{proof}

\begin{example}
\label{exm:associahedron8}
\fref{fig:CambrianLattice} shows the map~$\projectionMap : A_3 \to \clustercomplex$, for the Coxeter element~$\cexm \eqdef \tau_1\tau_2\tau_3$ of Example~\ref{exm:associahedron1}.

% tikz pictures
\newlength{\lskip}\setlength{\lskip}{2.5cm} % line skip
\newlength{\cskipA}\setlength{\cskipA}{3.2cm} % column skip in 1st line
\newlength{\cskipB}\setlength{\cskipB}{2.7cm} % column skip in 2nd line
\newlength{\cskipC}\setlength{\cskipC}{2.7cm} % column skip in 3rd line
\newlength{\cskipD}\setlength{\cskipD}{2.7cm} % column skip in 4th line
\newlength{\cskipE}\setlength{\cskipE}{3.2cm} % column skip in 5th line

\begin{figure}
  \capstart
  \centerline{
  \tikzstyle{every node}=[rectangle split, rectangle split parts=3, draw]
  \begin{tikzpicture}
    [
      inner sep=4pt, minimum width=2.3cm, align=center,
      spanningTree/.style={line width=.5mm, line cap=round, blue},
      sortable/.style={postaction={rectangle, line width=.5mm, draw=blue}}
    ]
%
    % NODES
    \node [sortable] (e)      at (          0,      0) {$1234$ \nodepart{second} $\phantom{|}e\phantom{|}$ \nodepart{third} $\{1,2,3\}$};
    \node [sortable] (1)      at (   -\cskipA, \lskip) {$2134$ \nodepart{second} $\phantom{|}\tau_1\phantom{|}$ \nodepart{third} $\{2,3,4\}$};
    \node [sortable] (2)      at (          0, \lskip) {$1324$ \nodepart{second} $\phantom{|}\tau_2\phantom{|}$ \nodepart{third} $\{1,3,7\}$};
    \node [sortable] (3)      at (    \cskipA, \lskip) {$1243$ \nodepart{second} $\phantom{|}\tau_3\phantom{|}$ \nodepart{third} $\{1,2,9\}$};
    \node [sortable] (12)     at (  -2\cskipB,2\lskip) {$2314$ \nodepart{second} $\phantom{|}\tau_1\tau_2\phantom{|}$ \nodepart{third} $\{3,4,5\}$};
    \node            (21)     at (   -\cskipB,2\lskip) {$3124$ \nodepart{second} $\tau_2|\tau_1$ \nodepart{third} $\{1,3,7\}$};
    \node [sortable] (13)     at (          0,2\lskip) {$2143$ \nodepart{second} $\phantom{|}\tau_1\tau_3\phantom{|}$ \nodepart{third} $\{2,4,9\}$};
    \node [sortable] (23)     at (    \cskipB,2\lskip) {$1342$ \nodepart{second} $\phantom{|}\tau_2\tau_3\phantom{|}$ \nodepart{third} $\{1,7,8\}$};
    \node            (32)     at (   2\cskipB,2\lskip) {$1423$ \nodepart{second} $\tau_3|\tau_2$ \nodepart{third} $\{1,2,9\}$};
    \node [sortable] (121)    at (-2.5\cskipC,3\lskip) {$3214$ \nodepart{second} $\tau_1\tau_2|\tau_1$ \nodepart{third} $\{3,5,7\}$};
    \node [sortable] (123)    at (-1.5\cskipC,3\lskip) {$2341$ \nodepart{second} $\phantom{|}\tau_1\tau_2\tau_3\phantom{|}$ \nodepart{third} $\{4,5,6\}$};
    \node            (231)    at ( -.5\cskipC,3\lskip) {$3142$ \nodepart{second} $\tau_2\tau_3|\tau_1$ \nodepart{third} $\{1,7,8\}$};
    \node            (132)    at (  .5\cskipC,3\lskip) {$2413$ \nodepart{second} $\tau_1\tau_3|\tau_2$ \nodepart{third} $\{2,4,9\}$};
    \node            (321)    at ( 1.5\cskipC,3\lskip) {$4123$ \nodepart{second} $\tau_3|\tau_2|\tau_1$ \nodepart{third} $\{1,2,9\}$};
    \node [sortable] (232)    at ( 2.5\cskipC,3\lskip) {$1432$ \nodepart{second} $\tau_2\tau_3|\tau_2$ \nodepart{third} $\{1,8,9\}$};
    \node [sortable] (1231)   at (  -2\cskipD,4\lskip) {$3241$ \nodepart{second} $\tau_1\tau_2\tau_3|\tau_1$ \nodepart{third} $\{5,6,7\}$};
    \node [sortable] (1232)   at (   -\cskipD,4\lskip) {$2431$ \nodepart{second} $\tau_1\tau_2\tau_3|\tau_2$ \nodepart{third} $\{4,6,9\}$};
    \node            (2312)   at (          0,4\lskip) {$3412$ \nodepart{second} $\tau_2\tau_3|\tau_1\tau_2$ \nodepart{third} $\{1,7,8\}$};
    \node            (1321)   at (    \cskipD,4\lskip) {$4213$ \nodepart{second} $\tau_1\tau_3|\tau_2|\tau_1$ \nodepart{third} $\{2,4,9\}$};
    \node            (2321)   at (   2\cskipD,4\lskip) {$4132$ \nodepart{second} $\tau_2\tau_3|\tau_2|\tau_1$ \nodepart{third} $\{1,8,9\}$};
    \node [sortable] (12312)  at (   -\cskipE,5\lskip) {$3421$ \nodepart{second} $\tau_1\tau_2\tau_3|\tau_1\tau_2$ \nodepart{third} $\{6,7,8\}$};
    \node            (12321)  at (          0,5\lskip) {$4231$ \nodepart{second} $\tau_1\tau_2\tau_3|\tau_2|\tau_1$ \nodepart{third} $\{4,6,9\}$};
    \node            (23121)  at (    \cskipE,5\lskip) {$4312$ \nodepart{second} $\tau_2\tau_3|\tau_1\tau_2|\tau_1$ \nodepart{third} $\{1,8,9\}$};
    \node [sortable] (123121) at (          0,6\lskip) {$4321$ \nodepart{second} $\tau_1\tau_2\tau_3|\tau_1\tau_2|\tau_1$ \nodepart{third} $\{6,8,9\}$};
%
    % EDGES
    \draw
      (e.north)     edge (1.south)
      (e.north)     edge (2.south)
      (e.north)     edge (3.south)
      (1.north)     edge (12.south)
      (1.north)     edge (13.south)
      (2.north)     edge (21.south)
      (2.north)     edge (23.south)
      (3.north)     edge (13.south)
      (3.north)     edge (32.south)
      (12.north)    edge (123.south)
      (12.north)    edge (121.south)
      (21.north)    edge (121.south)
      (21.north)    edge (231.south)
      (13.north)    edge (132.south)
      (23.north)    edge (231.south)
      (23.north)    edge (232.south)
      (32.north)    edge (232.south)
      (32.north)    edge (321.south)
      (123.north)   edge (1232.south)
      (123.north)   edge (1231.south)
      (121.north)   edge (1231.south)
      (231.north)   edge (2312.south)
      (132.north)   edge (1232.south)
      (132.north)   edge (1321.south)
      (232.north)   edge (2321.south)
      (321.north)   edge (1321.south)
      (321.north)   edge (2321.south)
      (1231.north)  edge (12312.south)
      (1232.north)  edge (12321.south)
      (2312.north)  edge (12312.south)
      (2312.north)  edge (23121.south)
      (1321.north)  edge (12321.south)
      (2321.north)  edge (23121.south)
      (12312.north) edge (123121.south)
      (12321.north) edge (123121.south)
      (23121.north) edge (123121.south)
    ;
%
    % FIBERS
    \draw [opacity=0.2, fill=blue]
      (2.south east) -- (2.north east) -- (21.south east) -- (21.north east) -- (21.north west) -- (21.south west) -- (2.north west) -- (2.south west) -- (2.south east)
      (3.south east) -- (3.north east) -- (32.south east) -- (32.north east) -- (321.south east) -- (321.north east) -- (321.north west) -- (321.south west) -- (32.north west) -- (32.south west) -- (3.north west) -- (3.south west) -- (3.south east)
      (13.south east) -- (13.north east) -- (132.south east) -- (132.north east) -- (1321.south east) -- (1321.north east) -- (1321.north west) -- (1321.south west) -- (132.north west) -- (132.south west) -- (13.north west) -- (13.south west) -- (13.south east)
      (23.south east) -- (23.north east) -- (231.south east) -- (231.north east) -- (2312.south east) -- (2312.north east) -- (2312.north west) -- (2312.south west) -- (231.north west) -- (231.south west) -- (23.north west) -- (23.south west) -- (23.south east)
      (232.south east) -- (232.north east) -- (2321.south east) -- (2321.north east) -- (23121.south east) -- (23121.north east) -- (23121.north west) -- (23121.south west) -- (2321.north west) -- (2321.south west) -- (232.north west) -- (232.south west) -- (232.south east)
      (1232.south east) -- (1232.north east) -- (12321.south east) -- (12321.north east) -- (12321.north west) -- (12321.south west) -- (1232.north west) -- (1232.south west) -- (1232.south east)
    ;
  \end{tikzpicture}}
  \caption{The map~$\projectionMap$ from~$\fS_4$ to the facets of the subword complex~$\subwordComplex(\cw{\cexm})$ for the Coxeter element~$\cexm = \tau_1\tau_2\tau_3$ of Example~\ref{exm:associahedron1}. Each node corresponds to a permutation of~$\fS_4$ and contains its one-line notation, its $\sq{\cexm}$-sorting word and its image under~$\projectionMap$. The \mbox{$\cexm$-sortable} elements are surrounded by strong boxes. The fibers of~$\projectionMap$ are represented by shaded regions. The quotient of the weak order by these fibers is the $\cexm$-Cambrian lattice and also the increasing flip order on the facets of~$\subwordComplex(\cw{\cexm})$.
  }
  \label{fig:CambrianLattice}
\end{figure}

\end{example}

The first consequence of Theorem~\ref{theo:pi=kappa} is the connection between the Cambrian lattice and the increasing flip order.

\begin{corollary}
\label{coro:CambrianLattice}
The bijection~$\facetSortable{c} : \facet c W \longrightarrow \sortable c W$ is an isomorphism between the increasing flip order on the facets of $\clustercomplex$ and the Cambrian lattice on $c$-sortable elements~in~$W$.
\end{corollary}

\begin{remark}
In fact, Corollary~\ref{coro:CambrianLattice} was already known for simply-laced types.
As pointed out in~\cite{CeballosLabbeStump}, the structure of the cluster complex as a subword complex can also be deduced from work by~K.~Igusa and R.~Schiffler~\cite[Theorem~2.8]{IgusaSchiffler}.
Together with the work of C.~Ingalls and H.~Thomas~\cite[Proposition~4.4]{IngallsThomas}, one can deduce the previous corollary in simply-laced types.
\end{remark}

\begin{remark}
One might be tempted to hope (and we actually did in a former version of this paper) that the increasing flip order of the facets of a root independent subword complex $\subwordComplex(\Q)$ is a lattice in general.
This does not hold, the smallest counterexample we are aware of is the subword complex $\subwordComplex(\tau_1 \tau_2 \tau_3 \tau_2 \tau_1 \tau_2 \tau_3 \tau_2 \tau_1)$ in type $A_3$.
See Remark~\ref{rem:latticeCounterExample}.
\end{remark}

\begin{remark}
In a subsequent paper, we extend the notion of ``Coxeter sorting trees'' on Cambrian lattices, and construct natural spanning trees of the increasing flip graph~\cite{PilaudStump-ELlabeling}.
\end{remark}

The second consequence of Theorem~\ref{theo:pi=kappa} is the connection between the Cambrian fan and the normal fan of the brick polytope, which was the heart of the construction of~\cite{HohlwegLangeThomas}.

\begin{corollary}
\label{coro:CambrianFan}
The normal fan of the brick polytope~$\brickPolytope(\cw{c})$ coincides with the $c$-Cambrian fan.
\end{corollary}

\begin{example}
\label{exm:associahedron9}
\fref{fig:CambrianFan} shows the $\cexm$-Cambrian fan as the normal fan of the brick polytope~$\brickPolytope(\cw{\cexm})$, for the Coxeter element~$\cexm \eqdef \tau_1\tau_2\tau_3$ of Example~\ref{exm:associahedron1}.

\begin{figure}
  \capstart
  \centerline{\includegraphics[scale=.7]{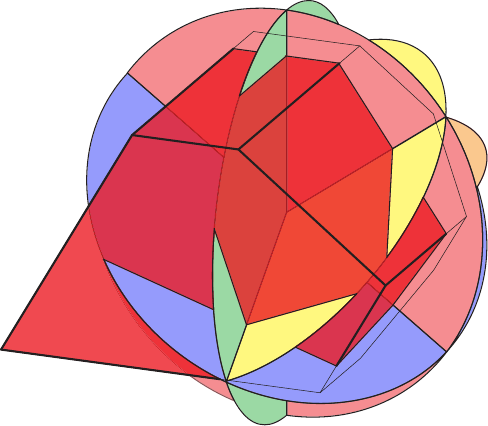}}
  \caption{The normal fan of~$\brickPolytope(\cw{\cexm})$ is the $\cexm$-Cambrian fan.}
  \label{fig:CambrianFan}
\end{figure}
\end{example}

Finally, we use Corollary~\ref{coro:ncphvector} to reobtain the following main result of~\cite{ABMCW2006} by C.~A.~Athanasiadis, T.~Brady, J.~McCammond, and C.~Watt.
As mentioned before, the language the statement is given here was exhibited in~\cite[Theorem~6.1]{Reading-sortableElements}, see also Theorem~\ref{thm:facetNcp}.
\begin{corollary}
  The rank generating function on the noncrossing partition lattice $\ncp c W$ coincides with the $h$-vector of the corresponding Cambrian lattice on $c$-sortable elements in~$W$.
\end{corollary}

\begin{proof}
  This is a direct consequence of Equation~\eqref{eq:ncphvector} and Corollary~\ref{coro:CambrianLattice}.
\end{proof}

%%%%%%%%%%%%%%%%%%%%%%%%%%%%%%%%%%%%%%

\subsection{The exchange matrix for facets of c-cluster complexes} \label{sec:exchangematrix}

To complete our study of $c$-cluster complexes as subword complexes, we now provide a way to compute the exchange matrix corresponding to a given facet of a $c$-cluster complex.
We first recall the needed definitions of skew-symmetrizable matrices and their mutation, and  again refer to~\cite{FominZelevinsky-ClusterAlgebrasII} and to~\cite{ReadingSpeyer} for further details.

A \defn{skew-symmetrizable matrix}~$B \eqdef (b_{uv})$ is a real $(n \times n)$-matrix for which there is a real diagonal $(n \times n)$-matrix~$D$ with positive diagonal entries~$d_1, \dots, d_n$ such that the product~$DB$ is skew-symmetric, \ie $d_ub_{uv} = -d_vb_{vu}$ for all~$u,v \in [n]$.
Often, one assumes as well that~$B$ and~$D$ are integer matrices.
Here, we relax this condition since we want to exhibit combinatorial matrix mutations for all finite Coxeter groups, including those of non-crystallographic types.
For an index~$w \in [n]$, the \defn{matrix mutation}~$\mutation{k}$ is the matrix operation $B \mapsto \mutation{w}(B) \eqdef B' \eqdef (b'_{uv})$ defined~by
\begin{align}
b'_{uv} \eqdef 
\begin{cases}
  -b_{uv} & \text{if } w \in \{u,v\}, \\
  b_{uv} + b_{uw}b_{wv} & \text{if } w \notin \{u,v\}, b_{uw} > 0, \text{ and } b_{wv} > 0, \\
  b_{uv} - b_{uw}b_{wv} & \text{if } w \notin \{u,v\}, b_{uw} < 0, \text{ and } b_{wv} < 0, \\
  b_{uv} & \text{otherwise.}
\end{cases}\label{eq:matrixmutation}
\end{align}
One can check that matrix mutations are involutions, \ie $\mutation{w}(\mutation{w}(B)) = B$ for all $(n \times n)$-matrix~$B$ and all~$w \in [n]$.
We moreover have the following lemma.

\begin{lemma}[{\cite[Proposition~4.5]{FominZelevinsky-ClusterAlgebrasI}}]
\label{lem:FominZelevinsky}
For any skew-symmetrizable $(n \times n)$-matrix~$B$ and any~$w \in [n]$, the matrix $\mutation{w}(B)$ is again skew-symmetrizable.
\end{lemma}

Matrix mutations play an essential role in the definition of cluster algebras~\cite{FominZelevinsky-ClusterAlgebrasI}: they drive the dynamics of the exchange relations constructing all clusters and cluster variables from an initial cluster seed.
Precise definitions can be found in the original papers~\cite{FominZelevinsky-ClusterAlgebrasI, FominZelevinsky-ClusterAlgebrasII}.
In this section, we only consider the exchange matrices and their combinatorial dynamics, and we voluntarily hide all algebraic aspects.

\medskip
Let~$(W,S)$ be a finite Coxeter system of rank~$n \eqdef |S|$ with generalized Cartan matrix~$A \eqdef (a_{st})_{s,t \in S}$ as defined in Section~\ref{sec:rootsandweights}.
Let~$\sq{c}$ be a reduced expression for a Coxeter element~$c \in W$.
Since $\sq{c}$ induces a (fixed, but not necessarily unique) linear ordering of the elements in~$S$, we consider $A$ to be indexed by positions in~$\sq{c}$ rather than by elements in~$S$.
The following procedure associates to a facet~$I$ of the $c$-cluster complex~$\clustercomplex$ a skew-symmetrizable matrix~$\mutmatrixold{I}$ with rows and columns indexed by the positions in~$I$.
First, one produces a skew-symmetrizable version of the generalized Cartan matrix by setting~$B^A = (b^A_{uv})$ to be the $(n \times n)$-matrix given by
\[
b^A_{uv} =
\begin{cases}
  -a_{uv} & \text{if } u < v, \\
  \phantom{-}a_{uv} & \text{if } u > v, \\
  \phantom{-}0 & \text{if } u = v. \\
\end{cases}
\]
To the initial facet~$I_e = [n]$ of~$\clustercomplex$, we associate~$\mutmatrixold{I_e} = B^A$.
Now, assume that~$I$ and~$J$ are two adjacent facets of~$\clustercomplex$ with~$I \ssm i = J \ssm j$, and that we know the matrix~$\mutmatrixold{I}$.
Then we define~$\mutmatrixold{J}$ as the matrix~$\mutation{i}(\mutmatrixold{I})$ where the row and column indexed by~$i$ in~$\mutmatrixold{I}$ become indexed by~$j$ in $\mutmatrixold{J}$.
We thus slightly changed the behavior of the mutation operator in the sense that we now have $\mutation{j}(\mutation{i}(\mutmatrixold{I})) = \mutmatrixold{I}$.
This corresponds to the fact that flipping position~$i \in I$ yields~$J$, and flipping then position~$j \in J$ yields again~$I$.
In other words, we index the rows and columns of~$\mutmatrixold{I}$ by the positions in~$I$.

A priori, it is unclear if this procedure is well defined, since it is not at all obvious that two sequences of flips from the initial facet~$I_e$ to a given facet~$I$ result in the same matrix~$\mutmatrixold{I}$.
Nevertheless, for crystallographic types, this is ensured by the fact that $\clustercomplex$ is isomorphic to the $c$-cluster complex and thus to the cluster complex for $(W,S)$.
The following theorem provides an alternative proof thereof for all finite Coxeter systems by giving an explicit, non-inductive description of~$\mutmatrixold{I}$.

\begin{theorem}
\label{thm:b-matrix}
For any facet~$I$ of the cluster complex~$\clustercomplex$, the exchange matrix~$\mutmatrixold{I}$ coincides with the matrix~$\mutmatrix{I}$ defined by
\begin{align*}
  \mutmatrix{I}_{uv} \eqdef
  \begin{cases}
      -\bigdotprod{\alpha_{q_u}^\vee}{\wordprod{\Q}{[u,v] \ssm I}(\alpha_{q_v})} & \text{ if } u < v \\[1pt]
      \phantom{-}\bigdotprod{\alpha_{q_v}}{\wordprod{\Q}{[v,u] \ssm I}(\alpha_{q_u}^\vee)} & \text{ if } u > v \\[1pt]
      \phantom{-}0 & \text{ if } u = v
  \end{cases}
\end{align*}
where~$\Q \eqdef \q_1 \cdots \q_\sizeQ = \cw{c}$.
In particular, $\mutmatrixold{I}$ is well-defined.
\end{theorem}

In the remainder of this section, we prove this statement by observing that $\mutmatrix{I_e} = B^A$ for the initial facet~$I_e$, and by showing that $\mutation{i}(\mutmatrix{I}) = \mutmatrix{J}$ for any two adjacent facets~$I,J$ of~$\clustercomplex$ with~$I \ssm i = J \ssm j$.

First, we record that~$\mutmatrix{I}$ is skew-symmetrizable by construction.
Although we do not necessarily need this statement and would finally obtain it from Lemma~\ref{lem:FominZelevinsky}, it will simplify some computations in the subsequent proof of Theorem~\ref{thm:b-matrix}.

\begin{lemma}
\label{lem:skew-symmetric}
For any facet~$I$ of~$\clustercomplex$, the matrix~$\mutmatrix{I}$ is skew-symmetri\-zable.
More precisely, for the diagonal matrix~$\diagonal$ defined by~$\diagonal_u \eqdef \dotprod{\alpha_{q_u}}{\alpha_{q_u}}$, the matrix~$\diagonal\mutmatrix{I}$ is skew-symmetric.
\end{lemma}

The following lemma ensures that the matrix~$\mutmatrix{I}$ is preserved under the isomorphism given in Lemma~\ref{lem:isomorphismlemma}.

\begin{lemma}
\label{lem:matrixrotation}
Let~$s$ be initial in~$c$, and let $c' = scs$ with a fixed reduced expression~$\sq{c'}$.
For any facet~$I$ of~$\clustercomplex$ and any positions~$u,v \in I$, sent by the isomorphism of Lemma~\ref{lem:isomorphismlemma} to the facet~$I'$ of~$\subwordComplex(\cw{c'})$ and to the positions~$u',v' \in I'$, we have~$\mutmatrix{I}_{uv} = \mutmatrix{I'}_{u' v'}$.
\end{lemma}

\begin{proof}
We only consider the situation $u < v$ in~$I$, the situation for $v < u$ being analogous.
If $1 \notin I$, the equality is ensured by the fact that $\dotprod{s(\beta)}{s(\beta')} = \dotprod{\beta}{\beta'}$.
If $1 \in I$ and $u \neq 1$, then $\mutmatrix{I}_{uv}$ and $\mutmatrix{I'}_{u'v'}$ are equal by definition.
Finally for~$u = 1 \in I$, we have
\[
\diagonal_1 \mutmatrix{I}_{1,v} = -\diagonal_v \mutmatrix{I}_{v,1} = -\diagonal_v \mutmatrix{I'}_{v-1,m} = \diagonal_1 \mutmatrix{I'}_{m,v-1},
\]
and therefore~$\mutmatrix{I}_{1v} = \mutmatrix{I'}_{u' v'}$.
\end{proof}

The following lemma ensures that the different cases in the matrix mutation given in~\eqref{eq:matrixmutation} are controlled by the relative order of the positions in the facets of~$\clustercomplex$.

\begin{lemma}
\label{lem:signCases}
Let~$I$ be a facet of~$\clustercomplex$, let~$u,i \in I$ and let $j$ be the unique position in $[m] \ssm I$ such that~$\Root{I}{i} = \pm \Root{I}{j}$.
Then
\begin{itemize}
\item $\mutmatrix{I}_{ui} \ge 0$ if~$u \le i < j$, or~$i < j < u$, or~$j < u \le i$,
\item $\mutmatrix{I}_{ui} \le 0$ if~$u < j < i$, or~$j < i \le u$, or~$i \le u < j$.
\end{itemize}
\end{lemma}

\begin{proof}
By Lemma~\ref{lem:matrixrotation}, we can assume~$u = 1$.
Lemma~\ref{lem:restrictionlemma} then yields~$\Root{I}{i} \in \Phi_{\langle \alpha_s \rangle}$ where~$s$ is the initial letter in~$\sq{c}$.
Since~$\Root{I}{i} \in \Phi^+$ if~$i < j$ and~$\Root{I}{i} \in \Phi^-$ if~$i > j$, the statement follows from the observation that~$\dotprod{\alpha_s}{\alpha} < 0$ for all~$\alpha \in \Delta \ssm \{\alpha_s\}$, and thus~$\dotprod{\alpha_s}{\beta} < 0$ for all~$\beta \in \Phi^+_{\langle \alpha_s \rangle}$.
\end{proof}

\begin{proposition}
\label{prop:matrixMutations}
Let~$I,J$ be two facets of~$\clustercomplex$ with~${I \ssm i = J \ssm j}$, and let~$u,v \in I \ssm i$.
Then $\mutmatrix{J}_{uj} = -\mutmatrix{I}_{ui}$, and $\mutmatrix{J}_{jv} = -\mutmatrix{I}_{iv}$, and
\[
\mutmatrix{J}_{uv} = 
\begin{cases}
  \mutmatrix{I}_{uv} + \mutmatrix{I}_{ui} \cdot \mutmatrix{I}_{iv} & \text{if } \mutmatrix{I}_{ui} \ge 0$ and~$\mutmatrix{I}_{iv} \ge 0, \\
  \mutmatrix{I}_{uv} - \mutmatrix{I}_{ui} \cdot \mutmatrix{I}_{iv} & \text{if } \mutmatrix{I}_{ui} \le 0$ and~$\mutmatrix{I}_{iv} \le 0, \\
  \mutmatrix{I}_{uv} & \text{otherwise}.
\end{cases}
\]
\end{proposition}

\begin{proof}
Using Lemma~\ref{lem:matrixrotation}, we assume all throughout the proof that $1 = u \leq v,i,j$.
We first show that $\mutmatrix{J}_{uj} = -\mutmatrix{I}_{ui}$ if $i<j$.
Indeed,
\begin{align*}
  \mutmatrix{J}_{uj}
  &= - \bigdotprod{\alpha_{q_u}^\vee}{\wordprod{\Q}{[u, j] \ssm J}(\alpha_{q_j})} \\
  &= - \bigdotprod{\alpha_{q_u}^\vee}{\wordprod{\Q}{[u, i] \ssm I} \cdot q_i \cdot \wordprod{\Q}{[i, j] \ssm I} (\alpha_{q_j})} \\
  &= - \bigdotprod{\alpha_{q_u}^\vee}{\wordprod{\Q}{[u, i] \ssm I} \cdot q_i (\alpha_{q_i})} \\
  &= \phantom{-} \bigdotprod{\alpha_{q_u}^\vee}{\wordprod{\Q}{[u, i] \ssm I} (\alpha_{q_i})} \\
  &= - \mutmatrix{I}_{ui}.
\end{align*}
By symmetry, an analogous computation solves the case of~$j < i$.
We could prove similarly that~$\mutmatrix{J}_{jv} = -\mutmatrix{I}_{iv}$, or equivalently, use Lemma~\ref{lem:skew-symmetric} to write
\[
\diagonal_j\mutmatrix{J}_{jv} = - \diagonal_v\mutmatrix{J}_{vj} = \diagonal_v\mutmatrix{I}_{vi} = -\diagonal_i\mutmatrix{I}_{iv}.
\]
The simple roots $\alpha_{q_i}$ and~$\alpha_{q_j}$ are obtained from each other by the orthogonal transformation $\wordprod{(\cw{c})}{[i,j] \ssm I}$ if~$i < j$ and~$\wordprod{(\cw{c})}{[j,i] \ssm I}$ if~$i > j$.
Therefore, we have~$\diagonal_j = \dotprod{\alpha_{q_j}}{\alpha_{q_j}} =  \dotprod{\alpha_{q_i}}{\alpha_{q_i}} = \diagonal_i$, and thus~$\mutmatrix{J}_{jv} = -\mutmatrix{I}_{iv}$.

\medskip
It thus only remains to treat the three different cases for $\mutmatrix{J}_{uv}$.
We first observe that we can use Lemma~\ref{lem:signCases} to write the conditions on the signs of~$\mutmatrix{I}_{ui}$ and~$\mutmatrix{I}_{iv}$ in terms of the positions of~$u$ and~$v$ with respect to~$i$.
We therefore distinguish the following three cases.

\para{Case 1}
If~$i < v < j$, we obtain
\begin{align*}
\mutmatrix{J}_{uv}
  &= -\dotprod{\alpha_{q_u}^\vee}{\wordprod{\Q}{[u, v] \ssm J}(\alpha_{q_v})} \\
  &= -\dotprod{\alpha_{q_u}^\vee}{\wordprod{\Q}{[u, i] \ssm I} \cdot q_i \cdot \wordprod{\Q}{[i, v] \ssm I}(\alpha_{q_v})} \\
  &= -\bigdotprod{\alpha_{q_u}^\vee}{\wordprod{\Q}{[u, i] \ssm I} \bigg( \wordprod{\Q}{[i, v] \ssm I}(\alpha_{q_v}) - \smalldotprod{\alpha_{q_i}^\vee}{\wordprod{\Q}{[i, v] \ssm I}(\alpha_{q_v})} \, \alpha_{q_i} \bigg)} \\
  &= \big(\mutmatrix{I}_{uv} + \mutmatrix{I}_{ui} \cdot \mutmatrix{I}_{iv} \big).
\end{align*}
In this computation, the first equality is the definition of~$\mutmatrix{J}_{uv}$, the second one holds since~$I \ssm i = J \ssm j$, the third one is obtained developing $q_i \big( \wordprod{\Q}{[i, v] \ssm I}(\alpha_{q_v}) \big)$, and the last one holds by bi-linearity of the dot product.

\para{Case 2}
If~$j < v < i$, we can make a similar proof as in Case~1, or simply use Lemma~\ref{lem:skew-symmetric} to write
\begin{align*}
\diagonal_u\mutmatrix{J}_{uv}
& = - \diagonal_v \mutmatrix{J}_{vu} = -\diagonal_v \big( \mutmatrix{I}_{vu} + \mutmatrix{I}_{vi} \cdot \mutmatrix{I}_{iu} \big) \\
& = -\diagonal_v\mutmatrix{I}_{vu} - \frac{1}{\diagonal_i} \cdot \diagonal_v\mutmatrix{I}_{vi} \cdot \diagonal_i\mutmatrix{I}_{iu} \\
& =  \diagonal_u\mutmatrix{I}_{uv} - \frac{1}{\diagonal_i} \cdot (-\diagonal_i\mutmatrix{I}_{iv}) \cdot (-\diagonal_u\mutmatrix{I}_{ui}) \\
& = \diagonal_u \big(\mutmatrix{I}_{uv} - \mutmatrix{I}_{ui} \cdot \mutmatrix{I}_{iv} \big).
\end{align*}

\vspace*{-.1cm}
\para{Case 3}
If~$v < i, j$ or~$i, j < v$, then we have ${\wordprod{\Q}{[u, v] \ssm I} = \wordprod{\Q}{[u, v] \ssm J}}$ and therefore~$\mutmatrix{I}_{uv} = \mutmatrix{J}_{uv}$.
\end{proof}

We can now deduce Theorem~\ref{thm:b-matrix} from the similarity between Proposition~\ref{prop:matrixMutations} and the definition of matrix mutation in Equation~\eqref{eq:matrixmutation}.

\begin{proof}[Proof of Theorem~\ref{thm:b-matrix}]
For the initial facet~$I_e$, the equality $\mutmatrix{I_e} = \mutmatrixold{I_e}$ holds by definition.
For two adjacent facets~$I,J$ of~$\clustercomplex$ with~$I \ssm i = J \ssm j$, Proposition~\ref{prop:matrixMutations} shows that~$\mutmatrix{I} = \mutmatrixold{I}$ implies~$\mutmatrix{J} = \mutmatrixold{J}$.
Observe here that the weak inequalities in the cases can as well be replaced by strict inequalities without changing them.
The statement thus directly follows.
\end{proof}

\begin{typeA}[Quivers and pseudoline arrangements]
In type~$A_n$, the exchange matrix of a cluster is just the adjacency matrix of a directed graph associated to the cluster, called its \defn{quiver}.
This quiver can be directly visualized
\begin{enumerate}[(i)]
\item either on the corresponding triangulation of the $(n+3)$-gon: rotate counter-clockwise around each triangle and connect the middles of consecutive edges;
\item or on the corresponding pseudoline arrangement: orient all pseudolines from left to right and quotient the resulting oriented graph by the contacts.
\end{enumerate}
The connection between these two perspectives is given by the duality developed in~\cite{PilaudPocchiola}.
This observation motivated the statements presented in this section.
\end{typeA}

%%%%%%%%%%%%%%%%%%%%%%%%%%%%%%%%%%%%%%
%%%%%%%%%%%%%%%%%%%%%%%%%%%%%%%%%%%%%%

\section{Root dependent subword complexes}
\label{sec:rootdependent}

We proposed in this paper a new approach to the construction of generalized associahedra of finite types, based on the interpretation of cluster complexes in terms of subword complexes~\cite{CeballosLabbeStump}.
This approach generalizes to all root independent subword complexes, thus providing polytopal realizations for a larger class of subword complexes.
However, this approach seems to reach its limits for root dependent subword complexes.
Indeed, when a subword complex $\subwordComplex(\Q)$ is root dependent, it cannot be realized by its brick polytope since the dimension of the brick polytope is given by the rank of the underlying Coxeter group (after a possible restriction to a parabolic subgroup, see Section~\ref{sec:parabolicrestriction}), while the dimension of a polytope realizing a given spherical subword complex $\subwordComplex(\Q)$ is given by the size of any facet of $\subwordComplex(\Q)$.
These dimensions do match if and only if $\subwordComplex(\Q)$ is root independent.
Even worst, it was actually shown in~\cite[Proposition~5.9]{PilaudSantos-brickPolytope} that the brick polytope may not even be a possible projection of a hypothetic realization of~$\subwordComplex(\Q)$.
Thus, realizations of all spherical subword complexes (or equivalently of all generalized multi-cluster complexes~\cite{CeballosLabbeStump}) cannot be achieved using our present approach.

\medskip
Nonetheless, the brick polytopes of root dependent subword complexes are still interesting polytopes, even if they do not provide polytopal realizations of their subword complexes.
It seems that their geometric and combinatorial properties are still driven by the root configurations of their facets. 
More precisely, the following conjecture and its corollary are known to hold in type~$A$, see~\cite{PilaudSantos-brickPolytope}.

\begin{conjecture}
\label{conj:rootDependent}
One can relax the condition for $\subwordComplex(\Q)$ to be root independent in Proposition~\ref{prop:preservingFlips}.
This is, for any subword complex $\subwordComplex(\Q)$ and a linear functional $f : V \to \R$, the set $\subwordComplex_f(\Q) \eqdef \bigset{I \text{ facet of } \subwordComplex(\Q)}{\forall i \in I, f(\Root{I}{i}) \ge 0}$ forms a connected component of the graph of $f$-preserving flips on $\subwordComplex(\Q)$.
\end{conjecture}

\begin{corollary}
If Conjecture~\ref{conj:rootDependent} holds, for any facet~$I$ of~$\subwordComplex(\Q)$, the cone of the brick polytope~$\brickPolytope(\Q)$ at the brick vector~$\brickVector(I)$ coincides with the cone generated by the negative of the root configuration~$\Roots{I}$ of~$I$.
In particular, the brick vector~$\brickVector(I)$ for a facet~$I$ of~$\subwordComplex(\Q)$ is a vertex of the brick polytope~$\brickPolytope(\Q)$ if and only if the cone~$\rootCone(I)$ is pointed.
\end{corollary}

Finally, we remark that in~\cite{PilaudStump-ELlabeling}, the authors study geometric, combinatorial, and algorithmic properties of the increasing flip order.
In that context, it is shown that these properties can naturally be derived for a strictly larger class of subword complexes than root independent subword complexes, namely~\defn{double root free} subword complexes.

%%%%%%%%%%%%%%%%%%%%%%%%%%%%%%%%%%%%%%
%%%%%%%%%%%%%%%%%%%%%%%%%%%%%%%%%%%%%%

\section*{Acknowledgements}
\label{sec:acknowledgements}

The first author thanks Francisco Santos and Michel Pocchiola for their inspiration and collaboration on the premises of brick polytopes in type~$A$.
He is also grateful to LaCIM for the visit opportunity where this work started, to the Fields Institute of Toronto for his postdoctoral stay in Canada, and to Antoine Deza for his supervision.
The second author wants to thank Pavlo Pylyavskyy for pointing him to~\cite{BerensteinZelevinsky} which was the starting point for the present paper.
We thank Cesar Ceballos, Fr\'ed\'eric Chapoton, Christophe Hohlweg, Jean-Philippe Labb\'e, Carsten Lange, Nathan Reading, and Hugh Thomas for fruitful discussions on the topic.
Finally, we are grateful to an anonymous referee for relevant suggestions on the submitted version.

%%%%%%%%%%%%%%%%%%%%%%%%%%%%%%%%%%%%%%
%%%%%%%%%%%%%%%%%%%%%%%%%%%%%%%%%%%%%%

\bibliographystyle{alpha}
\bibliography{PilaudStump}
\label{sec:biblio}

\end{document}